\documentclass[10pt]{article}
\usepackage[utf8]{inputenc}
\usepackage[english]{babel}
\usepackage{amsmath}
\usepackage{graphicx}
\usepackage{amsfonts}
\usepackage{amssymb}
\usepackage{amsthm}
\usepackage{dsfont}    
\usepackage{mathrsfs}
\usepackage{amsthm}
\usepackage{verbatim}
\usepackage{esint}
\usepackage{hyperref}
\usepackage{color}
\usepackage[left=2.5cm,right=2.5cm,top=4.0cm,bottom=3.5cm]{geometry}
\author{Iandoli Felice}
\title{Birkhoff Normal Forms}

\numberwithin{equation}{section}

\newtheorem{hyp}{Hypothesis}[section]

\newtheorem{theorem}{Theorem}
\newtheorem{theo}{Theorem}[section]
\newtheorem{coro}[theorem]{Corollary}
\newtheorem{de}{Definition}[section]
\newtheorem{lemma}{Lemma}[section]
\newtheorem{prop}{Proposition}[section]
\newtheorem{rmk}{Remark}[section]
\newtheorem{const}[theorem]{Constraint}

\newcommand{\zerarcounters}{\setcounter{equation}{0}}
\newcommand{\R}{\mathbb{R}}

\newcommand{\N}{\mathbb{N}}
\newcommand{\C}{\mathbb{C}}

\newcommand{\norm}[2]{\left \lVert {#1} \right \rVert_{{#2}}}
\newcommand{\asso}[1]{\left \lvert {#1} \right \rvert}
\newcommand{\set}[1]{\left\{{#1}\right\}}

\newcommand{\tonde}[1]{\left({#1}\right)}

\newcommand{\ZZZ}{\mathds{Z}}
\newcommand{\CCC}{\mathds{C}}
\newcommand{\NNN}{\mathds{N}}

\newcommand{\RRR}{\mathds{R}}
\newcommand{\TTT}{\mathbb{T}}
\newcommand{\uno}{\mathds{1}}

\newcommand{\calA}{{\calX}}
\newcommand{\BB}{{\mathcal B}}

\newcommand{\calF}{{\mathcal F}}
\newcommand{\calG}{{\mathcal G}}
\newcommand{\calH}{{\mathcal H}}

\newcommand{\calL}{{\mathcal L}}
\newcommand{\MM}{{\mathcal M}}

\newcommand{\RR}{{\mathcal R}}

\newcommand{\calX}{{\mathcal X}}
\newcommand{\calY}{{\mathcal Y}}

\newcommand{\gotL}{{\mathfrak L}}

\newcommand{\weyl}{{\rm Op}^{W}}
\newcommand{\bony}{{\rm Op}^{\mathcal{B}}}
\newcommand{\bonyw}{{\rm Op}^{\mathcal{B}W}}
\newcommand{\ol}{\overline}

\newcommand{\e}{\varepsilon}
\newcommand{\al}{\alpha}
\newcommand{\be}{\beta}

\newcommand{\x}{\xi}

\newcommand{\g}{\gamma}

\newcommand{\h}{\eta}
\newcommand{\la}{\lambda}

\newcommand{\s}{\sigma}

\newcommand{\del}{\partial}

\newcommand{{\resonance}}{relevant self-energy cluster }

\newcommand{\ii}{{\rm i}}

\newcommand{\hcic}{{\bf{H}}}

\begin{document}

\title{\bf Local well-posedness for quasi-linear NLS with large Cauchy data on the circle}

\author{\bf 
R. Feola$^{**}$, F. Iandoli$^\dag $
\\
\small
${}^{**}$ SISSA, Trieste, rfeola@sissa.it; 
\\
\small
${}^\dag$ SISSA, Trieste,  fiandoli@sissa.it\footnote{
This research was supported by PRIN 2015 ``Variational methods, with applications to problems in mathematical physics and geometry''.      
}}



\date{}

\maketitle

\begin{abstract}
We prove local in time well-posedness for a large class of quasilinear Hamiltonian, or parity preserving, Schr\"odinger equations
on the circle. After a paralinearization of the equation, we perform several paradifferential changes of coordinates in order to transform the system 
into a paradifferential one with symbols which, at the positive order, are constant and purely imaginary.  
  This allows to obtain a priori energy estimates on the Sobolev norms of the solutions.  
\end{abstract}


\tableofcontents

\zerarcounters

\section{Introduction}

\subsection{Main results}

In this paper we study the initial value problem (IVP)
\begin{equation}\label{NLS}
\left\{\begin{aligned}
&\ii \partial_t u+\del_{xx} u+P*u+f(u,u_x,u_{xx})=0,
\quad u=u(t,x), \quad x\in \TTT,\\
&u(0,x)=u_0(x)
\end{aligned}\right.
\end{equation}
where $\TTT:=\RRR/2\pi\ZZZ$, the nonlinearity $f$ is in $C^{\infty}(\CCC^3;\CCC)$ in the \emph{real sense} (i.e. $f(z_1,z_2,z_3)$ is $C^{\infty}$ as  function of ${\rm Re}(z_i)$ and ${\rm Im}(z_i)$ for $i=1,2,3$) vanishing at order $2$ at the origin, the potential $P(x)=\sum_{j\in\ZZZ} \hat{p}(j) \frac{e^{\ii jx}}{{\sqrt{2\pi}}}$ is a function in $C^1(\TTT;\CCC)$  with real Fourier coefficients $\hat{p}(j)\in\RRR$ for any $j\in\ZZZ$
and $P*u$ denotes the convolution between $P$ and $u=\sum_{j\in \ZZZ}\hat{u}(j)\frac{e^{\ii jx}}{\sqrt{2\pi}}$
\begin{equation}\label{convpotential}
P*u(x):=\int_{\TTT}P(x-y)u(y)dy=\sum_{j\in\ZZZ} \hat{p}(j)\hat{u}(j) e^{\ii j x}.
\end{equation}
Our aim is to prove the local existence, uniqueness and regularity of the classical solution of \eqref{NLS}
on  Sobolev spaces
\begin{equation}\label{SpaziodiSobolev}
H^{s}:=H^{s}(\TTT;\CCC):\left\{u(x)=\sum_{k\in\ZZZ}\hat{u}(k)\frac{e^{\ii kx}}{\sqrt{2\pi}}\; : \; \|u\|^{2}_{H^s}:=\sum_{j\in \ZZZ}\langle j\rangle^{2 s}|\hat{u}(j)|^{2}<\infty
\right\},
\end{equation}
where $\langle j\rangle:=\sqrt{1+|j|^{2}}$ for $j\in \ZZZ$, for $s$ large enough.

Similar problems have been studied in the case $x\in\RRR^n$, $n\geq 1$. For $x\in\RRR$, in the paper \cite{Pop1}, it was considered  the fully nonlinear Schr\"odinger type equation
$\ii \partial_t u= F(t,x,u,u_x,u_{xx})$;
it has been shown that the IVP associated to this equation is locally in time well posed in $H^{\infty}(\RRR;\CCC)$ (where  $H^{\infty}(\RRR;\CCC)$ denotes the intersection of all Sobolev spaces $H^{s}(\RRR;\CCC)$, $s\in \RRR$) if the function $F$ satisfies some suitable ellipticity hypotheses. 

Concerning the $n$-dimensional case the IVP for quasi-linear Schr\"odinger equations has been studied in  \cite{KPV1} in the Sobolev spaces $H^{s}(\RRR^n;\CCC)$ with $s$ sufficiently large. 
Here the key ingredient used to prove energy estimates  is a Doi's type lemma which involves pseudo-differential calculus for symbols defined on the Euclidean space $\RRR^n$.

Coming back to the case $x\in\TTT$ we mention \cite{BHM}. In this paper it is shown that  if $s$ is big enough and if the size of the initial datum $u_0$ is sufficiently small, then \eqref{NLS} is well posed  in the Sobolev space $H^{s}(\TTT)$ if $P=0$ and $f$ is  \emph{Hamiltonian} (in the sense of Hypothesis \ref{hyp1}). The proof is based on a Nash-Moser-H\"ormander implicit function theorem and the required energy estimates are obtained by means of a procedure of reduction to constant coefficients of the equation (as done in \cite{F}, \cite{FP}).


We remark that, even for the short time behavior of the solutions, there are deep differences between the problem \eqref{NLS} with periodic boundary conditions ($x\in\TTT$) and \eqref{NLS} with $x\in\RRR$. Indeed Christ proved in \cite{Cris} that the following family of problems
\begin{equation}\label{christ-eq}
 \left\{ \begin{aligned}
 & \partial_t u+\ii u_{xx}+u^{p-1}u_x=0\\
 & u(0,x)=u_0(x)
 \end{aligned}\right.
 \end{equation}
 is ill-posed  in all Sobolev spaces $H^s(\TTT)$ for any integer $p\geq 2$ and it is well-posed in $H^s(\RRR)$  for $p\geq 3$ and $s$ sufficiently large. The ill-posedness of \eqref{christ-eq} is very strong, in \cite{Cris} it has been shown that its solutions have  the following norm inflation phenomenon: for any $\e>0$ there exists a solution $u$ of \eqref{christ-eq} and a time $t_{\e}\in(0,\e)$ such that
\begin{equation*}
\norm{u_0}{H^s}\leq\e \quad {\rm{and}} \quad \norm{u(t_{\e})}{H^s}>\e^{-1}.
\end{equation*}
The examples exhibited in \cite{Cris} suggest that some assumptions on the nonlinearity $f$ in \eqref{NLS} are needed. In this paper we prove local well-posedness for   \eqref{NLS} in two cases. The first one is the \emph{Hamiltonian} case.
We assume that equation \eqref{NLS} can be written
in the complex  Hamiltonian form
\begin{equation}\label{mega3}
\partial_t u=\ii \nabla_{\bar{u}}\calH(u), 
\end{equation}
with Hamiltonian function
\begin{equation}\label{HAMILTONIANA}
 \calH(u)=\int_{\TTT}-|u_{x}|^{2} +(P*u) \bar{u}+
F(u,u_{x}) dx,
\end{equation}
for some real valued function $F\in\C^{\infty}(\CCC^2;\R)$ and where $\nabla_{\bar{u}}:=(\nabla_{{\rm Re}(u)}+\ii \nabla_{{\rm Im}(u)})/2$  and $\nabla$ denotes the $L^{2}(\TTT;\RRR)$
gradient. 
Note that the assumption $\hat{p}(j)\in\RRR$ implies that the Hamiltonian $\int_{\TTT}(P*u) \bar{u}dx$ is real valued. We denote by
$\del_{z_i}:=(\del_{{\rm Re}({z}_i)}-\ii \del_{{\rm Im}(z_i)})/2$ and 
$\del_{\bar{z}_i}:=(\del_{{\rm Re}({z}_i)}+\ii \del_{{\rm Im}({z}_i)})/2$ for $i=1,2$ the Wirtinger derivatives.
We assume the following.
\begin{hyp}[{\bf Hamiltonian structure}]\label{hyp1}
We assume that the nonlinearity $f$ in equation \eqref{NLS} has the form
\begin{equation}\label{NLS5}
\begin{aligned}
f(z_1,z_2,z_{3})&=(\del_{\bar{z}_1}F)(z_1,z_2)-\Big((\del_{z_1\bar{z}_2}F)(z_1,z_2)z_2+\\
&(\del_{\bar{z}_1\bar{z}_2}F)(z_1,z_2)\bar{z}_2+
(\del_{z_2\bar{z}_2}F)(z_1,z_2)z_3+(\del_{\bar{z}_2\bar{z}_2}F)(z_1,z_2)\bar{z}_3\Big),
\end{aligned}
\end{equation}
where $F$ is a real valued $C^{\infty}$ function  (in the real sense) defined on  $\CCC^2$ vanishing at $0$ at order $3$. 
\end{hyp}
Under the hypothesis above equation \eqref{NLS} is \emph{quasi-linear} in the sense that
the non linearity depends linearly on the variable $z_3$.
We remark that Hyp. \ref{hyp1} implies that the nonlinearity $f$ in \eqref{NLS} has the Hamiltonian form
\begin{equation*}
f(u,u_x,u_{xx})=(\del_{\bar{z}_1}F)(u,u_{x})-
\frac{d}{dx}[(\del_{\bar{z}_{2}}F)(u,u_{x})].
\end{equation*}

The second case is the \emph{parity preserving} case.
\begin{hyp}[{\bf Parity preserving structure}]\label{hyp2}
Consider the equation \eqref{NLS}. Assume that $f$   is a $C^{\infty}$ function in the real sense defined on  $\CCC^{3}$ and that it  vanishes   at order $2$ at the origin. Assume  $P$ has real Fourier coefficients. Assume moreover that $f$ and $P$ satisfy the following
\begin{enumerate}
\item  $f(z_1,z_2,z_3)=f(z_1,-z_2,z_3)$;
\item $(\del_{z_3}f)(z_1,z_2,z_3)\in \RRR$;
\item $P(x)=\sum_{j\in\ZZZ}\hat{p}(j) e^{\ii j x}$ is such that $\hat{p}(j)=\hat{p}(-j)\in\RRR$ (this means that $P(x)=P(-x)$).
\end{enumerate}
\end{hyp}

Note that item 1 in Hyp. \ref{hyp2} implies that if $u(x)$ is even in $x$ then $f(u,u_x,u_{xx})$ is even in $x$; item 3 implies that if $u(x)$ is even in $x$ so is $P*u$. Therefore the space of functions even in $x$ is invariant for \eqref{NLS}.
We assume item 2 to avoid parabolic terms in the non linearity, so that \eqref{NLS} is a \emph{Schr\"odinger-type} equation;
note that in this case the equation may be \emph{fully-nonlinear}, i.e. the dependence on the variable $z_{3}$ is not necessary linear.

In order to treat initial data with big size we shall assume also the following \emph{ellipticity condition}.
\begin{hyp}[{\bf Global ellipticity}]\label{hyp3}
We assume that there exist  constants $\mathtt{c_1},\,\mathtt{c_2}>0$ such that the following holds.
If $f$ in \eqref{NLS} satisfies Hypothesis \ref{hyp1} (i.e. has the form \eqref{NLS5}) then
\begin{equation}\label{constraint}
\begin{aligned}
&1-\del_{z_2}\partial_{\bar{z}_2}F(z_1,z_2)\geq
\mathtt{c_1},\\
&\big((1-\del_{z_2}\partial_{\bar{z}_2}F)^{2}-|\del_{\bar{z}_2}\partial_{\bar{z}_2}F|^{2}\big)(z_1,z_2)\geq
\mathtt{c_2}\\
\end{aligned}
\end{equation}
for any $(z_1,z_2)$ in $\mathbb{C}^2$.
If $f$ in \eqref{NLS} satisfies Hypothesis \ref{hyp2} then 
\begin{equation}\label{constraint2}
\begin{aligned}
&1+\del_{z_3}f(z_1,z_2,z_3)\geq \mathtt{c_1},\\
&\big((1+\del_{z_3}f)^{2}-|\del_{\bar{z}_{3}}f|^{2}\big)(z_1,z_2,z_3)\geq\mathtt{c_2}
\end{aligned}
\end{equation}
for any $(z_1,z_2,z_3)$ in $\mathbb{C}^3$.
\end{hyp}

The main result of the paper is the following.
\begin{theo}[{\bf Local existence}]\label{teototale}
Consider equation \eqref{NLS}, assume 
 Hypothesis \ref{hyp1} (respectively Hypothesis \ref{hyp2})
 and Hypothesis \ref{hyp3}. Then
there exists $s_0>0$ such that 
for any  $s\geq s_0 $ and
for any $u_0$ in
 $H^{s}(\TTT;\CCC)$ 
 (respectively any $u_0$ even in $x$ in the case of Hyp. \ref{hyp2}) 
 there exists $T>0$, depending only on $\|u_0\|_{H^{s}}$, such that  the equation \eqref{NLS} 
 with initial datum $u_0$ has a unique classical solution $u(t,x)$ (resp. $u(t,x)$ even in $x$) such that
$$
u(t,x) \in C^{0}\Big([0,T); H^{s}(\TTT)\Big)\bigcap C^{1}\Big([0,T); H^{s-2}(\TTT)\Big).
$$
Moreover there is a constant $C>0$ depending on $\norm{u_0}{H^{s_0}}$ and on $\norm{P}{C^1}$ such that
$$
\sup_{t\in [0,T)}\|u(t,\cdot)\|_{H^{s}}\leq C\|u_0\|_{H^{s}}.
$$
\end{theo}
We make some comments about Hypotheses \ref{hyp1},  \ref{hyp2} and  \ref{hyp3}. We remark that the class of Hamiltonian equations satisfying Hyp.  \ref{hyp1} is different from the parity preserving one satisfying Hyp.  \ref{hyp2}. For instance the equation
\begin{equation}\label{manuela}
\partial_t u=\ii\Big[(1+|u|^2)u_{xx}+u_x^2\bar{u}+(u-\bar{u})u_x\Big]
\end{equation}
 has the form \eqref{mega3} with Hamiltonian function 
\begin{equation*}
\mathcal{H}=\int_{\TTT}-|u_x|^2(1+|u|^2)+|u|^2(u_x+\bar{u}_x)dx,
\end{equation*}
but does not have the parity preserving structure (in the sense of Hyp. \ref{hyp2}). On the other hand the equation
\begin{equation}\label{manuela1}
\partial_tu=\ii(1+|u|^2)u_{xx}
\end{equation}
has the parity preserving structure but is not Hamiltonian with respect to the symplectic form $(u,v)\mapsto {\rm Re}\int_{\TTT} \ii u\bar{v}d x$. To check this fact one can reason as done in the appendix of \cite{ZGY}.
Both the examples \eqref{manuela} and \eqref{manuela1} satisfy the ellipticity Hypothesis \ref{hyp3}.  Furthermore there are examples of equations that satisfy Hyp. \ref{hyp1} or Hyp. \ref{hyp2} but do not satisfy Hyp. \ref{hyp3}, for instance 
\begin{equation}\label{manuela2}
\partial_tu=\ii(1-|u|^2)u_{xx}.
\end{equation}
The equation \eqref{manuela2} has the parity preserving structure and it has the form \eqref{NLS} with $P\equiv 0$ and $f(u,u_x,u_{xx})=-|u|^2u_{xx}$, therefore such an $f$ violates \eqref{constraint2} for $|u|\geq 1.$ Nevertheless we are  able to 
prove local existence for equations with
this kind of non-linearity if the size of the initial datum is sufficiently small; indeed, since $f$ in \eqref{NLS} is a $C^{\infty}$ function vanishing at the origin,  conditions \eqref{constraint2} in the case of Hyp. \ref{hyp2}  and \eqref{constraint} in the case of Hyp. \ref{hyp1}  are always locally fulfilled for $|u|$  small enough.
More precisely we have the following theorem.
\begin{theo}[{\bf Local existence for small data}]\label{teototale1}
Consider equation \eqref{NLS} and assume only  
Hypothesis \ref{hyp1} (respectively Hypothesis \ref{hyp2}).
Then there exists $s_0>0$ such that for any $s\geq s_0$
there exists $r_0>0$ such that, for any $0\leq r\leq r_0$,
the thesis of Theorem \ref{teototale} holds for any initial datum $u_0$ in the ball of radius $r$ of $H^{s}(\TTT;\CCC)$ centered at the origin.
\end{theo}
\vspace{1em}

Our method requires a high regularity of the initial datum. 
In the rest of the paper we have not been sharp in quantifying the minimal value of $s_0$ in Theorems \ref{teototale} and
\ref{teototale1}. 
 The reason for which we need regularity is to perform suitable changes of 
 coordinates and 
 having a symbolic calculus at a sufficient order, 
 which requires smoothness of the functions of the phase space.

The  convolution potential $P$ in equation \eqref{NLS}
is motivated by possible
future applications. For instance 
the potential $P$ can be used, as external parameter, in order to  modulate the linear frequencies with  the aim 
of studying 
 the long time stability of the small amplitude solutions of \eqref{NLS}
 by means of Birkhoff Normal Forms techniques.
For semilinear NLS-type equation  this has been done in \cite{BG}. 
As far as we know there are no results regarding quasi-linear NLS-type equations. 
For quasi-linear equations we quote \cite{Delort-2009}, \cite{Delort-Sphere}
for the Klein-Gordon and \cite{maxdelort} for the capillary Water Waves.

\subsection{Functional setting and ideas of  the proof}

\vspace{0.5em}
Here we introduce the phase space of functions  and  we give some ideas of the proof.
It is useful for our purposes to work  on the product space $H^s\times H^s:=H^{s}(\TTT;\CCC)\times H^{s}(\TTT;\CCC)$, in particular we will often use its subspace 
\begin{equation}\label{Hcic}
\begin{aligned}
&{\bf{ H}}^s:={\bf{ H}}^s(\TTT,\CCC^2):=\big(H^{s}\times H^{s}\big)\cap \mathcal{U},\\
&\mathcal{U}:=\{(u^{+},u^{-})\in L^{2}(\TTT;\CCC)\times L^{2}(\TTT;\CCC)\; : \; u^{+}=\ol{u^{-}}\},\\
\end{aligned}
\end{equation}
endowed with the product topology. On ${\bf{ H}}^0$ we define the scalar product
\begin{equation}\label{comsca}
(U,V)_{{\bf{ H}}^0}:=
\int_{\TTT}U\cdot\ol{V}dx.
\end{equation}
We introduce also the following subspaces of $H^s$ and of $\hcic^{s}$ made of even functions in $x\in\TTT$
\begin{equation}\label{spazipari}
\begin{aligned}
H^{s}_{e}&:=\{u\in H^{s}\; : \; u(x)=u(-x)\},\qquad 
{\bf{ H}}_{e}^s&:=(H^{s}_{e}\times H^{s}_{e})\cap \hcic^{0}.
\end{aligned}
\end{equation}

We 
define 
the operators $\lambda[\cdot]$ and $\bar{\lambda}[\cdot]$   by linearity as
\begin{equation}\label{NLS1000}
\begin{aligned}
\lambda [e^{\ii jx}]&:= \lambda_{j} e^{\ii jx}, \qquad \lambda_{j}:=(\ii j)^{2}+\hat{p}(j), \quad \;\; j\in \ZZZ,\\
\bar{\lambda}[e^{\ii jx }]&:= \lambda_{-j}e^{\ii jx},
\end{aligned}
\end{equation}
where $\hat{p}(j)$ are the Fourier coefficients of the potential $P$ in \eqref{convpotential}.
Let us introduce the following matrices
\begin{equation}\label{matrici}
E:=\left(\begin{matrix} 1 &0\\ 0 &-1\end{matrix}\right), \quad J:=\left(\begin{matrix} 0 &1\\ -1 &0\end{matrix}\right), \quad \uno:=\left(\begin{matrix} 1 &0\\ 0 &1\end{matrix}\right),
\end{equation}
and
set
\begin{equation}\label{DEFlambda}
 \Lambda U:=\left(\begin{matrix} \lambda [u] \\ \ol{\lambda}\; [\bar{u}]\end{matrix}\right), \qquad \; \forall \;\; U=(u,\bar{u})\in {\bf H}^{s}.
\end{equation}
We denote by 
$\mathfrak{P}$  the linear operator on $\hcic^{s}$ defined by
\begin{equation}\label{convototale}
\mathfrak{P}[U]:=\left(\begin{matrix} P*u\\
\bar{P}*\bar{u}
\end{matrix}\right), \quad U=(u,\bar{u})\in \hcic^s,
\end{equation}
where $P*u$ is defined in \eqref{convpotential}.
With this formalism we have that the operator $\Lambda$ in \eqref{DEFlambda} and \eqref{NLS1000} can be written as
\begin{equation}\label{DEFlambda2}
\Lambda:=\left(
\begin{matrix}
\del_{xx} & 0 \\0 & \del_{xx}
\end{matrix}
\right)+\mathfrak{P}.
\end{equation}
It is useful  to rewrite the equation \eqref{NLS}
as the equivalent system
\begin{equation}\label{NLSnaif}
\begin{aligned}
&\del_{t}U=\ii E\Lambda U+\mathtt{F}(U), \qquad \mathtt{F}(U):=\left(
\begin{matrix} f(u,u_x,u_{xx})\\
\ol{f(u,u_x,u_{xx})}\end{matrix}
\right),
\end{aligned}
\end{equation}
where $U=(u,\bar{u})$. The first step is to rewrite   \eqref{NLSnaif} as a paradifferential system
by using the paralinearization formula of Bony (see for instance \cite{Metivier}, \cite{Tay-Para}).
In order to do that, we will introduce rigorously classes of symbols in Section \ref{capitolo3}, here we follow the approach used in \cite{maxdelort}. Roughly speaking 
we shall  deal with functions $\TTT\times \RRR\ni (x,\x)\to a(x,\x)$ with limited smoothness in $x$ satisfying, for some $m\in \RRR$, the following estimate
\begin{equation}\label{falsisimboli}
|\del_{\x}^{\be}a(x,\x)|\leq C_{\be}\langle \x\rangle^{m-\beta}, \;\; \forall \; \be\in \NNN,
\end{equation}
where $\langle \x\rangle :=\sqrt{1+|\x|^{2}}$. These functions will have limited smoothness in $x$ because they will depend on $x$ through the dynamical variable $U$ which is in $\hcic^{s}(\TTT)$ for some $s$.
From the symbol $a(x,\x)$ one can define the \emph{paradifferential} operator $\bony(a(x,\x))[\cdot]$,
acting on periodic functions of the form $u(x)=\sum_{j\in\ZZZ}\hat{u}(j)\frac{e^{\ii j x}}{\sqrt{2\pi}}$, 
 in the following way:
\begin{equation}\label{sanbenedetto}
\bony(a(x,\x))[u]:=\frac{1}{2\pi}\sum_{k\in \ZZZ}e^{\ii k x}\left(
\sum_{j\in \ZZZ}\chi\left(\frac{k-j}{\langle j\rangle}\right)\hat{a}(k-j,j)\hat{u}(j)
\right),
\end{equation}
where $\hat{a}(k,j)$ is the $k^{th}$-Fourier coefficient of the $2\pi$-periodic in $x$ function $a(x,\x)$, and where
$\chi(\h)$ is a $C^{\infty}_0$ function  supported in a sufficiently small neighborhood of the origin.
With this formalism 
 \eqref{NLSnaif} is equivalent to the paradifferential system
\begin{equation}\label{falsaparali}
\del_{t}U= \ii E\mathcal{G}(U)[U]+\mathcal{R}(U),
\end{equation}
where 
 $\calG(U)[\cdot]$ is 
\begin{equation}\label{sanfrancesco}
\begin{aligned}
\calG(U)[\cdot]&:=\left(\begin{matrix} \bony((\ii\x)^{2}+a(x,\x))[\cdot] & \bony(b(x,\x))[\cdot]
\\
\bony(\ol{b(x,-\x)})[\cdot] & \bony((\ii\x)^{2}+\ol{a(x,-\x)})[\cdot]
\end{matrix}
\right),\\
a(x,\x)&:=a(U;x,\x)=\del_{u_{xx}}f (\ii\x)^{2}+\del_{u_x}f (\ii\x)+\del_{u}f,\\
b(x,\x)&:=b(U;x,\x)=\del_{\bar{u}_{xx}}f (\ii\x)^{2}+\del_{\bar{u}_x}f (\ii\x)+\del_{\bar{u}}f,\\
\end{aligned}
\end{equation}
and 
where $\mathcal{R}(U)$ is a smoothing operator 
\[
\mathcal{R}(\cdot) : \hcic^{s}\to \hcic^{s+\rho},
\]
for any $s>0$ large enough and $\rho\sim s$. 
Note that the symbols in \eqref{sanfrancesco} are of order $2$, i. e. they satisfy \eqref{falsisimboli} with $m=2$.
One of the most important property of being a paradifferential operator 
is the following: if $U$ is sufficiently regular, namely $U\in \hcic^{s_0}$ with $s_0$ large enough, then $\mathcal{G}(U)[\cdot]$ extends to a bounded linear operator from  
$\hcic^{s}$ to  $\hcic^{s-2}$ for any $s$ in $\RRR$.   
This paralinearization procedure will be discussed in detail in Section \ref{PARANLS}, in particular in Lemma \ref{paralinearizza}
and Proposition \ref{montero}.
Since equation \eqref{NLS} is quasi-linear the proofs of Theorems \ref{teototale}, \ref{teototale1} do not rely on  direct fixed point arguments; these arguments are used to study the local theory for the semi-linear equations (i.e. when the nonlinearity $f$ in \eqref{NLS} depends only on $u$). The local theory for the semi-linear Schr\"odinger type equations is, nowadays, well understood; for a complete overview we refer to \cite{caze}.
Our approach is based on the following  quasi-linear iterative scheme (a similar one is used  
for instance in \cite{ABK}).
We consider the sequence of linear problems
 \begin{equation}\label{sequenz0}
\mathcal{A}_0:=\left\{
\begin{aligned}
&\del_{t}U_0-\ii E\del_{xx}U_0=0 ,\\
&U_0(0)=U^{(0)},
\end{aligned}\right.
\end{equation}
and for $n\geq1$
\begin{equation}\label{sequenzn}
\mathcal{A}_n:=\left\{
\begin{aligned}
&\del_{t}U_n-\ii E\calG(U_{n-1})[U_{n}]-\mathcal{R}(U_{n-1})=0 ,\\
&U_n(0)=U^{(0)},
\end{aligned}\right.
\end{equation}
where  $U^{(0)}(x)=(u_0(x),\ol{u_0}(x))$ with $u_0(x)$ given in \eqref{NLS}.
The goal is to show that there exists $s_0>0$ such that for any $s\geq s_0$ the following facts hold:
\begin{enumerate}
\item the iterative scheme is well-defined, i.e. 
there is $T>0$ such that for any $n\geq0$ there exists a unique solution $U_{n}$ of the problem
$\mathcal{A}_{n}$ which belongs to the space $C^{0}([0,T)]; \hcic^{s})$; 

\item the sequence $\{U_{n}\}_{n\geq0}$ is bounded in 
$C^{0}([0,T)]; \hcic^{s})$;

\item $\{U_n\}_{n\geq0}$ is a Cauchy sequence in $C^{0}([0,T)]; \hcic^{s-2})$.

\end{enumerate}

From these properties the limit function $U$ belongs to the space $L^{\infty}([0,T); \hcic^{s})$.
 In the final part of Section \ref{local} we show that actually  $U$
 is a \emph{classical} solution of \eqref{NLS}, namely $U$ solves \eqref{NLSnaif}
 and it belongs to $C^{0}([0,T);\hcic^{s})$.

Therefore the key point is to obtain energy estimates for the linear problem in $V$
\begin{equation}\label{seqseqseq}
\left\{
\begin{aligned}
&\del_{t}V-\ii E\calG(U)[V]-\mathcal{R}(U)=0 ,\\
&V(0)=U^{(0)},
\end{aligned}\right.
\end{equation}
where  $\mathcal{G}$ is  given in \eqref{sanfrancesco}
and $U=U(t,x)$ is
a fixed function defined for $t\in[0,T], T>0$, regular enough and $\mathcal{R}(U)$ is regarded as a non homogeneous forcing term.
Note that the regularity in time and space of the coefficients of operators $\mathcal{G}, \mathcal{R}$
depends on the regularity of the function $U$.
Our strategy is to perform a paradifferential
change of coordinates 
$W:=\Phi(U)[V]$ such that the system \eqref{seqseqseq} in the new coordinates reads
\begin{equation}\label{seqseqseq2}
\left\{
\begin{aligned}
&\del_{t}W-\ii E\widetilde{\mathcal{G}}(U)[W]-\widetilde{\mathcal{R}}(U)=0 ,\\
&W(0)=\Phi(U^{(0)})[U^{(0)}],
\end{aligned}\right.
\end{equation}
where the operator  $\widetilde{\mathcal{G}}(U)[\cdot]$  is self-adjoint with constant coefficients
in $x\in \TTT$ and $\widetilde{\mathcal{R}}(U)$ is a bounded term. More precisely we show that 
the operator $\widetilde{\calG}(U)[\cdot]$ has the form
\begin{equation}\label{scopo}
\begin{aligned}
\widetilde{\calG}(U)[\cdot]&:=\left(\begin{matrix} \bony((\ii\x)^{2}+m(U;\x))[\cdot] & 0
\\
0 & \bony((\ii\x)^{2}+m(U;\x))[\cdot]
\end{matrix}
\right),\\
m(U;\x)&:=m_{2}(U)(\ii\x)^{2}+m_{1}(U)(\ii\x)\in \RRR,\\
\end{aligned}
\end{equation}
with $m(U;\x)$ real valued and independent of $x\in \TTT$. Since the symbol $m(U;\xi)$ is real valued the linear operator $\ii E \widetilde{\mathcal{G}}(U)$ generates a well defined flow on $L^2\times L^2$, since it has also constant coefficients in $x$ it generates a flow on $H^{s}\times H^s$ for $s\geq 0$.
This idea of conjugation to constant coefficients up to bounded remainder has been developed in order to study the linearized equation associated to quasi-linear system in the context of Nash-Moser iterative scheme. For instance we quote the papers 
\cite{BBM}, \cite{BBM1}
on the KdV equation, \cite{FP}, \cite{FP2} on the NLS equation and \cite{IPT}, \cite{BM1}, \cite{BBHM}, \cite{alaz-baldi-periodic} on the water waves equation, in which such techniques are used in studying the existence of periodic and quasi-periodic solutions.
 Here, dealing with the paralinearized equation \eqref{falsaparali}, we adapt the changes of coordinates, for instance performed in \cite{FP}, to the paradifferential context following the strategy introduced in \cite{maxdelort} for the water waves equation.

 \vspace{1em}
{\bf Comments on Hypotheses \ref{hyp1}, \ref{hyp2} and \ref{hyp3}.}

\vspace{0.5em}
Consider the following linear system
\begin{equation}\label{IVPlineare}
\begin{aligned}
\del_{t}V-\ii E \mathfrak{L}(x)\del_{xx} V=0 ,\\
\end{aligned}
\end{equation}
where $\gotL(x)$ is the non constant coefficient matrix
\begin{equation}\label{isabella}
\gotL(x):=\left(
\begin{matrix}
1+a_{2}(x) & b_{2}(x) \\
\ol{b_{2}(x)} & 1+a_{2}(x)
\end{matrix}
\right), \quad a_{2} \in C^{\infty}(\TTT;\RRR), \;\; b_{2}\in C^{\infty}(\TTT;\CCC),
\end{equation}
Here we explain how to diagonalize and conjugate to constant coefficients the system \eqref{IVPlineare} at the highest order, we also discuss the role of the Hypotheses \ref{hyp1}, \ref{hyp2} and \ref{hyp3}.
The analogous analysis for  the paradifferential system \eqref{seqseqseq} is performed in Section \ref{descent1}.

{\emph{First step: diagonalization at the highest order.}}
We want to transform \eqref{IVPlineare} into the system
\begin{equation}\label{IVPlineare2}
\begin{aligned}
&\del_{t}V_{1}=\ii E \left(A^{(1)}_2(x)\del_{xx} V_1+A_{1}^{(1)}(x)\del_{x}V_1+A_0^{(1)}(x)V_1\right) ,\\
\end{aligned}
\end{equation}
where $A_{1}^{(1)}(x),A_0^{(1)}(x)$ are $2\times2$ matrices of functions, 
and $A_{2}^{(1)}(x)$ is the diagonal matrix of functions
\[
A_{2}^{(1)}(x)=\left(
\begin{matrix}
1+a^{(1)}_{2}(x) & 0 \\
0 & 1+a_{2}^{(1)}(x)
\end{matrix}
\right),
\]
for some real valued function $a_{2}^{(1)}(x)\in C^{\infty}(\TTT;\RRR)$. See Section \ref{secondord} for the paradifferential linear system \eqref{seqseqseq}.
 The matrix $E\gotL(x)$ can be diagonalized 
 through a regular transformation if the determinant of $E\gotL(x)$ is strictly positive, i.e. there exists $c>0$ such that
 \begin{equation}\label{condlineare}
 {\rm det}\Big(E\gotL(x)\Big)=(1+a_{2}(x))^{2}-b_{2}(x)^{2}\geq c,
 \end{equation}
for any $x\in \TTT$.
Note that the eigenvalues of  $E\gotL(x)$ are $\lambda_{1,2}(x)=\pm \sqrt{{\rm det}E\gotL(x)}$. Let  $\Phi_1(x)$ be the matrix of functions such that
\[
\Phi_1(x)\big(E\gotL(x)\big)\Phi_1^{-1}(x)=EA_{2}^{(1)}(x),
\]
where $(1+a_{2}^{(1)}(x))$ is the positive eigenvalue of $E\gotL(x)$.
One obtains the system \eqref{IVPlineare2} by setting
 $V_1:=\Phi_1(x)V$.

Note that condition \eqref{condlineare}  is the transposition  at the linear level of the second inequality in  \eqref{constraint} or \eqref{constraint2}.
Note also  that if $\|a_{2}\|_{L^{\infty}},\|b_{2}\|_{L^{\infty}}\leq r$ then condition \eqref{condlineare}
is automatically fulfilled for $r$ small enough.

{\emph{Second step: reduction to constant coefficients at the highest order.}} In order to understand the role of the first bound in  conditions \eqref{constraint} and \eqref{constraint2}
we perform a further step in which we reduce the system \eqref{IVPlineare2}
to 
\begin{equation}\label{IVPlineare3}
\begin{aligned}
&\del_{t}V_{2}=\ii E \left(A^{(2)}_2\del_{xx} V_2+A_{1}^{(2)}(x)\del_{x}V_2+A_{0}^{(2)}(x)V_2\right) ,\\
\end{aligned}
\end{equation}
where $A_{1}^{(2)}(x),A_{0}^{(2)}(x)$ are $2\times2$ matrices of functions, 
and 
\[
A_{2}^{(2)}=\left(
\begin{matrix}
m_2 & 0 \\
0 & m_2
\end{matrix}
\right),
\]
for some constant $m_2\in \RRR$, $m_2>0$.
See Section \ref{ridu2} for the reduction of the paradifferential linear system \eqref{seqseqseq}.
In order to do this
we use the torus diffeomorphism $x\to x+\be(x)$ for some periodic function $\beta(x)$
with inverse given by $y\to y+\gamma(y)$ with $\gamma(y)$ periodic in $y$. We define the following  linear operator   $(Au)(x)=u(x+\be(x))$, such operator is invertible with inverse given by
 $(A^{-1}v)(y)=v(y+\gamma(y))$.
This change of coordinates
transforms \eqref{IVPlineare2} into \eqref{IVPlineare3} where 
 \begin{equation}
A_{2}^{(2)}(x)= \left(
\begin{matrix}
A[(1+a_{2}^{(1)}(y))(1+\gamma_{y}(y))^{2}] & 0 \\
0 & A[(1+a_{2}^{(1)}(y))(1+\gamma_{y}(y))^{2}]
\end{matrix}
\right).
 \end{equation}
 Then the highest order coefficient does not depend on $y\in\TTT$
 if 
 \[
 (1+a_{2}^{(1)}(y))(1+\gamma_{y})^{2}=m_2,
 \]
 with $m_2\in \RRR$ independent of $y$. This equation can be solved by setting
 \begin{equation}\label{monastero}
 \begin{aligned}
 m_{2}&:=\left[2\pi \left(\int_{\TTT}\frac{1}{\sqrt{1+a_{2}^{(1)}(x)}}d x\right)^{-1}\right]^{2},\\
 \gamma(y)&:=\del_{y}^{-1}\left(\sqrt{\frac{m_2}{1+a_{2}^{(1)}(y)}}-1\right),
 \end{aligned}
 \end{equation}
 where $\del_{y}^{-1}$ is the Fourier multiplier with symbol $1/(\ii\x)$, hence it is defined only on zero mean functions. This justifies  the choice of $m_2$. 
 Note that $m_2$, $\gamma$ in \eqref{monastero} are well-defined if
 $(1+a_{2}^{(1)}(x))$ is real and strictly positive for any $x\in \TTT$.
 This is the first condition in \eqref{constraint} and \eqref{constraint2}.

{\emph{Third step: reduction at lower orders.}} 
One can show
that it is always possible to conjugate system \eqref{IVPlineare3} to a system of the form
\begin{equation}\label{esempio}
\del_{t}V_3=\ii E \left(A^{(3)}_2\del_{xx} V_2+A_{1}^{(3)}(x)\del_{x}V_2+A_{0}^{(3)}(x)V_2\right),
\end{equation}
where $A_{2}^{(3)}\equiv A_{2}^{(2)}$ and
\[
A_{1}^{(3)}:=\left(\begin{matrix}m_{1} & 0\\
0 & \ol{m_1}
\end{matrix}
\right),
\]
with $m_{1}\in \CCC$
and $A_{0}^{(3)}(x)$ is a matrix of  functions up to bounded operators.
See Sections \ref{diago1}, \ref{ridu1} for the analogous reduction for paradifferential linear system \eqref{seqseqseq}.

It turns out that no extra hypotheses are needed to perform this third step. 
We obtained that 
the unbounded term in the r.h.s. of  \eqref{esempio} is pseudo-differential constant coefficients operator 
with symbol 
$m(\x):=m_2(\ii\x)^{2}+m_1(\ii\x)$. 
This is not enough to get energy estimates because the operator $A_{2}^{(3)}\del_{xx}+A_{1}^{(3)}\del_{x}$
is not self-adjoint since the symbol $m(\x)$ is not \emph{a-priori} real valued.

This example gives the idea that the global ellipticity hypothesis Hyp. \ref{hyp3} (or the smallness of the initial datum), are needed to conjugate the highest order term of $\mathcal{G}$ in \eqref{seqseqseq} to a diagonal and constant coefficient operator. 
 Of course  there are no \emph{a-priori} reasons to conclude that $\widetilde{\calG}$ is self-adjoint.
 This operator  is self-adjoint if and only if its symbol $m(U;\x)$ in \eqref{scopo} is real valued for any $\x\in \RRR$. 
 The Hamiltonian hypothesis \ref{hyp1} implies that $m_1(U)$ in \eqref{scopo} is purely imaginary,
 while the parity preserving assumption \ref{hyp2} guarantees
 that $m_1(U)\equiv0$. 
  Indeed 
 it is shown  Lemma \ref{struttura-ham-para} 
 that
if $f$ is Hamiltonian (i.e. satisfies 
Hypothesis \ref{hyp1}) then the operator $\cal{G}(U)[\cdot]$ is formally self-adjoint
w.r.t. the scalar product of $L^{2}\times L^{2}$. 
In our reduction procedure we use transformations which preserve this structure.
On the other hand in the case that
$f$ is parity preserving (i.e. satisfies Hyp. \ref{hyp2})
then, in Lemma  \ref{struttura-rev-para} it is shown that 
the operator $\calG(U)[\cdot]$ maps even functions in even functions if $U$ is even in $x\in \TTT$.
In this case we apply only transformations which preserve the parity of the functions.
An operator of the form $\widetilde{\mathcal{G}}$ as in \eqref{scopo}
preserves the subspace of even function  only if $m_1(U)=0$.
 \section*{Acknowledgements}
We warmly thanks Prof. Massimiliano Berti for the interesting discussions and for the very useful advices.

\section{Linear operators}
We define some special classes of linear operators on spaces of functions.

\begin{de}\label{barrato}
Let $A : H^{s}\to H^{s'}$, for some $s,s'\in \RRR$, be a linear operator.
We define the  operator $\ol{A}$ 
as 
\begin{equation}\label{barrato2}
\ol{A}[h]:= \ol{A[\bar{h}]},  \qquad h\in H^{s}.
\end{equation}
\end{de}
\begin{de}[{\bf Reality preserving}]\label{realpre}
Let $A, B : H^{s}\to H^{s'}$, for some $s,s'\in \RRR$, be  linear operators.
We say that a matrix of 
 linear operators $\mathfrak{F}$ is \emph{reality} preserving if has the form
 \begin{equation}\label{barrato4}
\mathfrak{F}:=\left(\begin{matrix} A & B \\ \ol{B} & \ol{A}\end{matrix}\right), 
\end{equation}
for $A$ and $B$ linear operators. 
\end{de}

\begin{rmk}\label{operatoripreserving1}
Given $s,s'\in \RRR$,
one can easily check that 
a \emph{reality} preserving linear operator $\mathfrak{F} $ of the form \eqref{barrato4}
is such that
\begin{equation}\label{mappare}
\mathfrak{F} : \hcic^{s}\to \hcic^{s'}.
\end{equation}
\end{rmk}

Given an operator $\mathfrak{F}$  of the form \eqref{barrato4} 
we denote by $\mathfrak{F^*}$ its adjoint with respect to the scalar product $\eqref{comsca}$
\begin{equation*}
(\mathfrak{F}U,V)_{{\bf{ H}}^0}=(U,\mathfrak{F}^{*}V)_{{\bf{ H}}^0}, \quad \forall\,\, U,\, V\in {\bf{ H}}^s.
\end{equation*}
One can check that 
\begin{equation}\label{natale}
\mathfrak{F}^*:=\left(\begin{matrix} A^* & \ol{B}^* \\ {B}^* & \ol{A}^*\end{matrix}\right),
\end{equation}
where $A^*$ and $B^*$ are respectively the adjoints of the operators $A$ and $B$ with respect to
the complex scalar product on $L^{2}(\TTT;\CCC)$
\[
(u,v)_{L^{2}}:=\int_{\TTT}u\cdot \bar{v}dx, \quad u,v\in L^{2}(\TTT;\CCC).
\]
\begin{de}[{\bf Self-adjointness}]\label{selfi}
Let $\mathfrak{F}$   be a reality preserving linear operator
of the form \eqref{barrato4}.
We say that  $\mathfrak{F}$ 
is \emph{self-adjoint} if $A,A^*,B,B^* : H^{s}\to H^{s'}$, for some $s,s'\in \RRR$ and
\begin{equation}\label{calu}
A^{*}=A,\;\;
\;\; \ol{B}=B^{*}.
\end{equation}
\end{de}

We have the following definition.

\begin{de}[{\bf Parity preserving}]\label{revmap}
Let $A : H^{s}\to H^{s'}$, for some $s,s'\in \RRR$ be a linear operator. 
We say that  $A$ 
is \emph{parity} preserving if
\begin{equation}\label{revmap2000}
A : H^{s}_{e}\to H^{s'}_{e},
\end{equation}
i.e. maps even functions in even functions of $x\in \TTT$.
 Let $\mathfrak{F} :\hcic^{s} \to \hcic^{s'} $ be a reality preserving operator of the form \eqref{barrato4}.
We say that $\mathfrak{F}$ is parity preserving if the operators
$A,B$ are parity preserving operators. 
%
\end{de}

\begin{rmk}\label{operatoripreserving2}
Given $s,s'\in \RRR$,
and let $\mathfrak{F} :\hcic^{s} \to \hcic^{s'} $ be a reality and parity preserving operator of the form \eqref{barrato4}.
One can check that
 that
\begin{equation}\label{revmap2}
\mathfrak{F} : {\bf H}_{e}^{s}\to {\bf H}_{e}^{s'}.
\end{equation}
\end{rmk}
%
%
%
%
%

 We note that $\Lambda$ in \eqref{DEFlambda} has the following properties: 
 \begin{itemize}
 \item   the operator $\Lambda$ is reality preserving (according to Def. \ref{realpre}).

 \item
  the operator $\Lambda$  is self-adjoint 
  according to Definition \ref{selfi}
since the coefficients $\hat{p}(j)$ for $j\in \ZZZ$
 are real;
 
 \item under the parity preserving assumption Hyp. \ref{hyp2} the operator $\Lambda$
 is  parity preserving according to Definition \ref{revmap},
 since $\hat{p}(j)=\hat{p}(-j)$ for $j\in \ZZZ$.
 
 \end{itemize}

\paragraph{Hamiltonian and parity preserving vector fields.}
Let $\mathfrak{F}$  be a reality preserving, self-adjoint (or parity preserving respectively) operator as in \eqref{barrato4}
and consider the linear system
\begin{equation}\label{linearsystem}
\partial_t U=\ii E\mathfrak{F}U, 
\end{equation}
on $\hcic^{s}$ where $E$ is given in \eqref{matrici}.
We want to analyze   how the properties of the system \eqref{linearsystem}
change under the conjugation through maps 
$$
\Phi : \hcic^{s}\to \hcic^{s},
$$
which are reality preserving.
We have the following lemma. 
\begin{lemma}\label{lemmalemma2}
Let $\calA : \hcic^{s}\to \hcic^{s-m}$, for some $m\in \RRR$ and $s>0$ be a reality preserving, self-adjoint operator according to Definitions
\ref{realpre}, \ref{selfi} 
and assume that its flow 
\begin{equation}\label{system15}
\del_{\tau}\Phi^{\tau}=\ii E\calA \Phi^{\tau}, \qquad \Phi^{0}=\uno,
\end{equation}
satisfies the following. The map $\Phi^{\tau}$ is a continuous function in $\tau\in[0,1]$ with values in the space of bounded linear operators from $\hcic^s$ to $\hcic^s$  and $\del_{\tau}\Phi^{\tau}$ is continuous as well in $\tau\in[0,1]$ with values in the space of bounded linear operators from $\hcic^s$ to $\hcic^{s-m}$.

Then the map $\Phi^{\tau}$ satisfies the  condition 
\begin{equation}\label{symsym10}
(\Phi^{\tau})^{*}(-\ii E )\Phi^{\tau}=- \ii E.
\end{equation}

\end{lemma}

\begin{proof}

First we note that the adjoint operator $(\Phi^{\tau})^{*}$  satisfies the equation
$\del_{\tau}(\Phi^{\tau})^{*}=(\Phi^{\tau})^{*}\calX(-\ii E)$. Therefore one can note that 
$$
\del_{\tau}\Big[ (\Phi^{\tau})^{*}(-\ii E )\Phi^{\tau}
\Big]=0,
$$
which implies $(\Phi^{\tau})^{*}(-\ii E )\Phi^{\tau}=(\Phi^{0})^{*}(-\ii E )\Phi^{0}=-\ii E$.
\end{proof}

\begin{lemma}\label{lemmalemma}
Consider a reality preserving, self-adjoint  linear operator $\mathfrak{F}$ (i.e. which satisfies \eqref{barrato4} and  \eqref{calu})
and a reality preserving  map $\Phi$. 
Assume that 
$\Phi$ satisfies condition \eqref{symsym10}
and
consider  the system
\begin{equation}\label{system}
\partial_t W=\ii E \mathfrak{F} W, \qquad W\in {\bf H}^{s}.
\end{equation}
By setting $V=\Phi W$ one has that the system \eqref{system} reads
\begin{equation}\label{system22}
\partial_t V=\ii E \calY V,
\end{equation}
\begin{equation}\label{system2}
\calY:=-\ii E\Phi(\ii E)\mathfrak{F}\Phi^{-1}-\ii E(\del_{t}\Phi)\Phi^{-1},
\end{equation}
and $\calY$ is self-adjoint, i.e. it satisfies conditions \eqref{calu}.
\end{lemma}

\begin{proof}
One applies the changes of coordinates and one gets the form in \eqref{system2}.
We prove that separately each term of $\calY$ is self-adjoint.
Note that by \eqref{symsym10} one has that $(-\ii E )\Phi=(\Phi^{*})^{-1}(-\ii E)$, hence
$ -\ii E\Phi(\ii E)\mathfrak{F}\Phi^{-1}=(\Phi^{*})^{-1} \mathfrak{F} \Phi^{-1} $. Then
\begin{equation}\label{system12}
\Big( (\Phi^{*})^{-1} \mathfrak{F} \Phi^{-1} \Big)^{*}=(\Phi^{-1})^{*} \mathfrak{F} [(\Phi^{*})^{-1}]^{*},
\end{equation}
since $\mathfrak{F}$ is self-adjoint.
Moreover we have that $(\Phi^{-1})^{*}=(\Phi^{*})^{-1}$. Indeed again by \eqref{symsym10} one has that
$$
\Phi^{-1}=(\ii E)\Phi^{*}(-\ii E), \qquad  (\Phi^{-1})^{*}=(\ii E) \Phi (-\ii E), \quad \Phi^{*}=(-\ii E)\Phi^{-1}(\ii E)
$$
Hence one has
\begin{equation}\label{system13}
(\Phi^{-1})^{*}\Phi^{*}=(\ii E) \Phi (-\ii E)(-\ii E)\Phi^{-1}(\ii E)=-(\ii E)(\ii E)=\uno.
\end{equation}
Then by \eqref{system12} we conclude that $(-\ii E)\Phi \ii E \Phi^{-1}$ is self-adjoint. Let us study the second term of \eqref{system2}.
First note that 
\begin{equation}\label{palomba}
\del_{t}[\Phi^{*}]=-(\Phi^{*})(-\ii E)(\del_{t}\Phi)\Phi^{-1}(\ii E), \qquad (\del_{t}\Phi)^{*}=\Phi^{*}(\ii E)(\del_{t}(\Phi^{*}))^{*}\Phi^{-1}(\ii E)
\end{equation}
then
\begin{equation}\label{system14}
\Big(
(-\ii E)(\del_{t}\Phi)(\Phi^{-1})
\Big)^{*}=(\Phi^{-1})^{*}(\del_{t}\Phi)^{*}(\ii E)=(-\ii E)(\del_{t}(\Phi^{*}))^{*}\Phi^{-1}.
\end{equation}
By \eqref{palomba} we have $\del_{t}(\Phi^{*})=(\del_{t}\Phi)^{*}$, hence we get the result.
\end{proof}

\begin{lemma}\label{revmap100}
Consider a reality  and parity preserving linear operator $\mathfrak{F}$ (i.e. \eqref{barrato4} and \eqref{revmap2} hold)
and a map $\Phi$ as in \eqref{barrato4} which is parity preserving (see Def. \ref{revmap}).
Consider  the system
\begin{equation}\label{system100}
\partial_t W=\ii E \mathfrak{F} W, \qquad W\in {\bf H}^{s}.
\end{equation}
By setting $V=\Phi W$ one has that the system \eqref{system} reads
\begin{equation}\label{system22100}
\partial_t V=\ii E \calY V,
\end{equation}
\begin{equation}\label{system2100}
\calY:=-\ii E\Phi(\ii E)\mathfrak{F}\Phi^{-1}-\ii E(\del_{t}\Phi)\Phi^{-1},
\end{equation}
and $\calY$ is reality preserving and parity preserving, i.e. satisfies condition \eqref{barrato4} and \eqref{revmap2}.
\end{lemma}

\begin{proof}
It follows straightforward by the Definitions
\ref{revmap} and \ref{realpre}.
\end{proof}

\zerarcounters
\section{Paradifferential calculus}\label{capitolo3}
\subsection{Classes of symbols}
We introduce some notation. If $K\in\N$, $I$ is an interval of $\R$ containing the origin, $s\in\R^{+}$ we denote by $C^K_{*\R}(I,{\bf{H}}^{s}(\TTT,\CCC^2))$, sometimes  by $C^K_{*\R}(I,{\bf{H}}^{s})$, the space of continuous functions $U$ of $t\in I$ with values in  ${\bf{H}}^{s}(\TTT,\CCC^2)$, which are $K$-times differentiable and such that the $k-$th derivative is continuous with values in ${\bf{H}}^{s-2k}(\TTT,\CCC^2)$ for any $0\leq k\leq K$. We endow the space  $C^K_{*\R}(I,{\bf{H}}^{s})$ with the norm
\begin{equation}\label{spazionorm}
\sup_{t\in I}\norm{U(t,\cdot)}{K,s}, \quad \mbox {where} \quad \norm{U(t,\cdot)}{K,s}:=\sum_{k=0}^{K}\norm{\partial_t^k U(t,\cdot)}{{\bf{H}}^{s-2k}}.
\end{equation}
Moreover if $r\in\R^{+}$ we set
\begin{equation}\label{pallottola}
B_{s}^K(I,r):=\set{U\in C^K_{*\R}(I,{\bf{H}}^{s}):\, \sup_{t\in I}\norm{U(t,\cdot)}{K,s}<r}.
\end{equation}

\begin{de}[\bf{Symbols}]\label{nonomorest}
Let $m\in\R$, $K'\leq K$ in $\N$, $r>0$. We denote by $\Gamma^m_{K,K'}[r]$ the space of functions $(U;t,x,\xi)\mapsto a(U;t,x,\xi)$, defined for $U\in B_{\s_0}^K(I,r)$, for some large enough $\s_0$, with complex values such that for any $0\leq k\leq K-K'$, any $\s\geq \s_0$, there are $C>0$, $0<r(\s)<r$ and for any $U\in B_{\s_0}^K(I,r(\s))\cap C^{k+K'}_{*\R}(I,{\bf{H}}^{\s})$ and any $\alpha, \beta \in\N$, with $\alpha\leq \s-\s_0$
\begin{equation}\label{simbo}
\asso{\partial_t^k\partial_x^{\alpha}\partial_{\xi}^{\beta}a(U;t,x,\xi)}\leq C \norm{U}{k+K',\s}\langle\xi\rangle^{m-\beta},
\end{equation}
for some constant $C=C(\s, \norm{U}{k+K',\s_0})$ depending only on $\s$ and $\norm{U}{k+K',\s_0}$.
\end{de}

\begin{rmk}\label{notazionetempo}
In the rest of the paper the time $t$ will be treated as a parameter. 
In order to simplify the notation we shall write $a(U;x,\x)$ instead of $a(U;t,x,\x)$.
On the other hand we will emphasize the $x$-dependence of a symbol $a$.
We shall denote by $a(U;\x)$ only those symbols which are independent of the variable $x\in \TTT$. 
\end{rmk}

\begin{rmk}\label{differenzaclassidisimbo}
If one compares the latter definition 
of class of symbols with the one given in Section $2$ in \cite{maxdelort}
one note that they have been more precise on the expression of the constant $C$ in the r.h.s. of \eqref{simbo}. First of all we do not need such precision since we only want to study local theory.
Secondly their classes are modeled in order to work in a small neighborhood  of the origin.
\end{rmk}

\begin{lemma}\label{unasbarretta}
Let $a\in \Gamma^{m}_{K,K'}[r]$ and $U\in B_{\s_0}^{K}(I,r)$ for some $\s_0$.
One has that
\begin{equation}\label{Normaunabarra2}
\sup_{\x\in\RRR}\langle\x\rangle^{-m}\|a(U;\cdot,\x)\|_{K-K',s}\leq C \norm{U}{K,s+\s_0+1}.
\end{equation}
for $s\geq 0$.
\end{lemma}

\begin{proof}
Assume that $s \in \NNN$. We have
\begin{equation}\label{Normaunabarra}
\begin{aligned}
\|a(U;x,\x)\|_{K-K',s}&\leq C_1 \sum_{k=0}^{K-K'}\sum_{j=0}^{s-2k} \|\del_{t}^{k}\del_{x}^{j}a(U; \cdot,\x)\|_{L^{\infty}}\\
&\leq C_2 \langle \x\rangle^{m}\sum_{k=0}^{K-K'}\|U\|_{k+K',s+\s_0},
\end{aligned}
\end{equation}
with $C_{1},C_{2}>0$ depend only on $s,K$ and $\norm{U}{k+K',\s_0}$, 
and where we used formula \eqref{simbo} with $\s=s+\s_0$. Equation \eqref{Normaunabarra} implies
\eqref{Normaunabarra2} for $s\in \NNN$. The general case $s\in \RRR_{+}$, follows by using the log-convexity of the Sobolev norm
by writing $s=[s]\tau+(1-\tau)(1+[s])$ where $[s]$ is the integer part of $s$ and $\tau\in[0,1]$.
\end{proof}

We define  the following special subspace  of $\Gamma^0_{K,K'}[r]$.
\begin{de}[\bf{Functions}]\label{apeape}
Let  $K'\leq K$ in $\N$, $r>0$. We denote by $\calF_{K,K'}[r]$  the subspace of $\Gamma^0_{K,K'}[r]$ made of those symbols which are independent of $\xi$.
\end{de}

\subsection{Quantization of symbols}\label{opmulti2}

Given a smooth symbol $(x,\x) \to a(x,\x)$,
we define, for any $\s\in [0,1]$,  the quantization of the symbol $a$ as the operator 
acting on functions $u$ as 
\begin{equation}\label{operatore}
{\rm Op}_{\s}(a(x,\x))u=\frac{1}{2\pi}\int_{\RRR\times\RRR}e^{\ii(x-y)\x}a(\s x+(1-\s)y,\x)u(y)dy d\x.
\end{equation}
This definition is meaningful in particular if $u\in C^{\infty}(\TTT)$ (identifying $u$ to a $2\pi$-periodic function). By decomposing 
$u$ in Fourier series  as $u=\sum_{j\in\ZZZ}\hat{u}(j)\frac{e^{\ii jx}}{\sqrt{2\pi}}$, we may calculate the oscillatory integral in \eqref{operatore} obtaining
\begin{equation}\label{bambola}
{\rm Op}_{\s}(a)u:=\frac{1}{2\pi}\sum_{k\in \ZZZ}\left(\sum_{j\in\ZZZ}\hat{a}\big(k-j,(1-\s)k+\s j\big)\hat{u}(j)\right)e^{\ii k x}, \quad \forall\;\; \s\in[0,1],
\end{equation}
where $\hat{a}(k,\xi)$ is the $k^{th}-$Fourier coefficient of the $2\pi-$periodic function $x\mapsto a(x,\xi)$.
In the paper we shall use two particular quantizations:

\paragraph{Standard quantization.}
We define the standard quantization by specifying formula \eqref{bambola} for $\s=1$:
\begin{equation}\label{bambola2}
{\rm Op}(a)u:={\rm Op}_{1}(a)u=\frac{1}{2\pi}\sum_{k\in \ZZZ}\left(\sum_{j\in\ZZZ}\hat{a}\big(k-j, j\big)\hat{u}(j)\right)e^{\ii k x};
\end{equation}

\paragraph{Weyl quantization.}
We define the Weyl quantization by specifying formula \eqref{bambola} for $\s=\frac{1}{2}$:
\begin{equation}\label{bambola202}
{\rm Op}^{W}(a)u:={\rm Op}_{\frac{1}{2}}(a)u=\frac{1}{2\pi}\sum_{k\in \ZZZ}\left(\sum_{j\in\ZZZ}\hat{a}\big(k-j, \frac{k+j}{2}\big)\hat{u}(j)\right)e^{\ii k x}.
\end{equation}

Moreover one can transform the symbols between different quantization by using the formulas
\begin{equation}\label{bambola5}
{\rm Op}(a)={\rm Op}^{W}(b), \qquad {\rm where} \quad \hat{b}(j,\x)=\hat{a}(j,\x-\frac{j}{2}).
\end{equation}
In order to define operators starting from the classes of symbols introduced before, we reason as follows. Let $n\in \ZZZ$, we define the projector on $n-$th Fourier mode as
\begin{equation}\label{Fou}
\left(\Pi_{n} u\right)(x):=\hat{u}({n})\frac{e^{\ii n x}}{\sqrt{2\pi}}; \quad u(x)=\sum_{n\in\ZZZ}\hat{u}(n)\frac{e^{\ii jx}}{\sqrt{2\pi}}.
\end{equation}
For $U\in B^K_s(I,r)$ (as in Definition \ref{nonomorest}), 
 a symbol $a$ in $ \Gamma^{m}_{K,K'}[r]$, and $v\in C^{\infty}(\TTT,\CCC)$
we define
\begin{equation}\label{quanti2}
{\rm Op}(a(U;j))[v]:=\sum_{k\in \ZZZ}\left(\sum_{j\in \ZZZ}\Pi_{k-j}a(U;j)\Pi_{j}v \right).
\end{equation}
Equivalently one can define ${\rm Op}^{W}(a)$
 according to
\eqref{bambola202}.

We want to define a \emph{paradifferential} quantization.
First we give the following definition.
\begin{de}[{\bf Admissible cut-off functions}]\label{cutoff1}
We say that a function $\chi\in C^{\infty}(\R\times\R;\R)$ is an admissible cut-off function if it is even with respect to each of its arguments and there exists $\delta>0$ such that
\begin{equation*}
\rm{supp}\, \chi \subset\set{(\xi',\xi)\in\R\times\R; |\xi'|\leq\delta \langle\xi\rangle},\qquad \xi\equiv 1\,\,\, \rm{ for } \,\,\, |\xi'|\leq\frac{\delta}{2} \langle\xi\rangle.
\end{equation*}
We assume moreover that for any derivation indices $\alpha$ and $\beta$
\begin{equation*}
|\partial_{\xi}^{\alpha}\partial_{\xi'}^{\beta}\chi(\xi',\xi)|\leq C_{\alpha,\beta}\langle\xi\rangle^{-\alpha-\beta},\,\,\forall \alpha, \,\beta\in\NNN.
\end{equation*}
\end{de}
An example of function satisfying the condition above, and that will be extensively used in the rest of the paper, is $\chi(\xi',\xi):=\widetilde{\chi}(\xi'/\langle\xi\rangle)$, where $\widetilde{\chi}$ is a function in $C_0^{\infty}(\RRR;\RRR)$  having a small enough support and equal to one in a neighborhood of zero.
For any $a\in C^{\infty}(\TTT)$ we shall use the following notation
\begin{equation}\label{pseudoD}
(\chi(D)a)(x)=\sum_{j\in\ZZZ}\chi(j)\Pi_{j}{a}.
\end{equation}

\begin{prop}[{\bf Regularized  symbols}]
Fix $m\in \RRR$, $p,K,K'\in \NNN$, $K'\leq K$ and $r>0$. 
Consider $a\in  {\Gamma}^{m}_{K,K'}[r]$ and  $\chi$ an admissible cut-off function according to Definition \ref{cutoff1}. Then the function
\begin{equation}\label{nsym2}
a_{\chi}(U;x,\x) := 
\sum_{n\in \ZZZ}\chi\left(n,\x
\right) \Pi_{n}a(U;x,\x)
\end{equation}
belongs to ${\Gamma}^{m}_{K,K'}[r]$.
\end{prop}
For the proof we refer the reader to the remark after Definition $2.2.2$ in \cite{maxdelort}.

We define  the Bony quantization in the following way. 
Consider 
an admissible cut-off function $\chi$  and a symbol $a$ belonging to the class $ {\Gamma}^{m}_{K,K'}[r]$, we set
\begin{equation}\label{boninon}
\textrm{Op}^{\BB}(a(U;x,j))[v]:= \textrm{Op}(a_{\chi}(U;x,j))[v],
\end{equation}
where $a_{\chi}$ is defined in \eqref{nsym2}. Analogously we define the Bony-Weyl quantization
\begin{equation}\label{bweylq}
\bonyw(b(U;x,j))[v]:= {\rm Op}^{W}(b_{\chi}(U;x,j))[v].
\end{equation}
The definition of the operators ${\rm Op}^{\BB}(b)$ and $\bonyw(b)$ is independent of the choice of the cut-off function  $\chi$  modulo smoothing operators that we define now.


\begin{de}[\bf{Smoothing remainders}]\label{nonomoop}
Let $K'\leq K\in\N$,  $\rho\geq0$ and $r>0$. We define the class of remainders $\mathcal{R}^{-\rho}_{K,K'}[r]$ as the space of maps $(V,u)\mapsto R(V)u$ defined on $B^K_{s_0}(I,r)\times C^K_{*\R}(I,H^{s_0}(\TTT,\CCC))$ which are linear in the variable $u$ and such that the following holds true. For any $s\geq s_0$ there exists a constant $C>0$ and $r(s)\in]0,r[$ such that for any $V\in B^K_{s_0}(I,r)\cap C^K_{*\R}(I,H^{s}(\TTT,\CCC^2))$, any $u\in C^K_{*\R}(I,H^{s}(\TTT,\CCC))$, any $0\leq k\leq K-K'$ and any $t\in I$ the following estimate holds true
\begin{equation}\label{porto20}
\norm{\partial_t^k\left(R(V)u\right)(t,\cdot)}{H^{s-2k+\rho}}\leq \sum_{k'+k''=k}C\Big[\norm{u}{k'',s}\norm{V}{k'+K',s_0}+\norm{u}{k'',s_0}\norm{V}{k'+K',s}\Big],
\end{equation}
where $C=C(s,\norm{V}{k'+K',s_0})$ is a constant depending only on $s$  and $\norm{V}{k'+K',s_0}$.
\end{de}
\begin{lemma}\label{equiv}
Consider $\chi_{1}$ and $\chi_2$  admissible cut-off functions.
Fix $m\in\RRR$, $r>0$,  $K'\leq K\in\N$.
Then for $a\in {\Gamma}^{m}_{K,K'}[r]$,
we have ${\rm Op}(a_{\chi_1}-a_{\chi_{2}})\in \RR^{-\rho}_{K,K'}[r]$ for any $\rho\in \NNN$.
\end{lemma}
For the proof we refer the reader to the remark after the proof of Proposition $2.2.4$ in \cite{maxdelort}.

Now we state a proposition  describing the action of paradifferential operators defined in \eqref{boninon} and in \eqref{bweylq}.
\begin{prop}[{\bf Action of paradifferential operators}]\label{boni2}
Let $r>0$, $m\in\R,$  $K'\leq K\in\N$ and consider a symbol $a\in\Gamma^{m}_{K,K'}[r]$. There exists $s_0>0$ such that for any $U\in B^{K}_{s_0}(I,r)$, the operator $\bonyw(a(U;x,\xi))$ extends, for any $s\in\R$, as a bounded operator from the space $C^{K-K'}_{*\R}(I,H^{s}(\TTT,\CCC))$ to $C^{K-K'}_{*\R}(I,H^{s-m}(\TTT,\CCC))$. Moreover 
there is a constant
$C>0$ depending on $s$ and on the constant in \eqref{simbo} such that
\begin{equation}\label{paraparaest}
\|\bonyw(\del_{t}^{k}a(U;x,\cdot))\|_{\calL(H^{s},H^{s-m})}\leq C \|U\|_{k+K',s_0},
\end{equation}
for $k\leq K-K'$, so that

\begin{equation}\label{paraest}
\norm{\bonyw(a(U;x,\xi))(v)}{K-K',{s-m}}
\leq C \norm{U}{{K,{s_0}}}\norm{v}{K-K',s},
\end{equation}
for any $v\in C^{K-K'}_{*\R}(I,H^{s}(\TTT,\CCC))$.
\end{prop}
For the proof we refer to Proposition 2.2.4 in \cite{maxdelort}.
\begin{rmk}\label{ronaldo10}
Actually the estimates \eqref{paraparaest} and  \eqref{paraest}
follow by 
\[
\norm{\bonyw(a(U;x,\xi))(v)}{K,{s-m}}
\leq C_1
\sup_{\x\in\RRR}\langle\x\rangle^{-m}\|a(U;\cdot,\x)\|_{K-K',s_0}\|v\|_{K-K',s},
\]
where $C_1>0$ is some constant depending only on $s,s_0$ and 
Remark \ref{unasbarretta}.
\end{rmk}

\begin{rmk}\label{ronaldo2}
We remark that   Proposition \ref{boni2} (whose proof is given in \cite{maxdelort})
applies if $a$ satisfies \eqref{simbo} with $|\al |\leq 2$ and $\be=0$.
Moreover, by following the same proof, one can show that

\begin{equation}
\|\weyl(\del_{t}^{k}a_{\chi}(U;x,\cdot))\|_{\calL(H^{s},H^{s-m})}\leq C \|U\|_{k+K',s_0},
\end{equation}
if $\chi(\h,\x)$ is supported for $|\h|\leq \delta \langle\x\rangle$
for $\delta>0$ small.
Note that this is slightly different from the Definition \ref{cutoff1} of admissible cut-off function 
since we are not requiring that $\chi\equiv1$ for $|\eta|\leq \frac{\delta}{2}\langle\xi\rangle$.
\end{rmk}

\begin{rmk}\label{inclusione-nei-resti}
Note that, if $m<0$, and $a\in \Gamma^{m}_{K,K'}[r]$, then estimate \eqref{paraparaest}
implies that the operator $\bonyw(a(U;x,\x))$ belongs to the class of smoothing operators 
$\RR^{m}_{K,K'}[r]$.
\end{rmk}

We  consider paradifferential operators  of the  form:
\begin{equation}\label{prodotto}
 \quad \bonyw(A(U;x,\xi)):=\bonyw\left(\begin{matrix} {a}(U;x,\x) & {b}(U;x,\x)\\
{\ol{b(U;x,-\x)}} & {\ol{a(U;x,-\x)}}
\end{matrix}
\right) :=\left(\begin{matrix} \bonyw({a}(U;x,\x)) & \bonyw({b}(U;x,\x))\\
\bonyw({\ol{b(U;x,-\x)}}) & \bonyw({\ol{a(U;x,-\x))}}
\end{matrix}
\right),
\end{equation}
where $a$ and $b$ are symbols in $\Gamma^{m}_{K,K'}[r]$ and $U$ is a function belonging to $B^{K}_{s_0}(I,r)$ for some $s_0$ large enough. Note that the matrix of operators in \eqref{prodotto} is of the form \eqref{barrato4}. Moreover it is self-adjoint (see \eqref{calu}) if and only if
\begin{equation}\label{quanti801}
a(U;x,\xi)=\ol{a(U;x,\xi)}\,,\quad b(U;x,-\xi)= b(U;x,\xi),
\end{equation}
indeed conditions \eqref{calu} on these operators read
\begin{equation}\label{megaggiunti}
\left(\bonyw(a(U;x,\xi))\right)^{*}=\bonyw\left(\ol{a(U;x,\xi)}\right)\, ,\quad \ol{\bonyw(b(U;x,\xi))}=\bonyw\left(\ol{b(U;x,-\xi)}\right).
\end{equation}
Analogously, given $R_{1}$ and $R_{2}$ in $\RR^{-\rho}_{K,K'}[r]$,
one can define a reality preserving smoothing operator on ${\bf{H}}^s(\TTT,\CCC^2)$ as follows
\begin{equation}\label{vinello}
R(U)[\cdot]:=\left(\begin{matrix} R_{1}(U)[\cdot] & R_{2}(U)[\cdot] \\
\ol{R_{2}}(U)[\cdot] & \ol{R_{1}}(U)[\cdot] 
\end{matrix}
\right).
\end{equation}
We use the following notation for matrix of operators.
\begin{de}[{\bf Matrices}]\label{matrixmatrix} 
We denote by 
$\Gamma^{m}_{K,K'}[r]\otimes\MM_2(\CCC)$ the matrices $A(U;x,\xi)$ of the form \eqref{prodotto}
whose components are symbols in the class
 $\Gamma^{m}_{K,K'}[r]$. 
 In the same way we denote by
 $ \RR^{-\rho}_{K,K'}[r]\otimes\MM_2(\CCC)$ the operators $R(U)$ of the form 
 \eqref{vinello} whose
 components are smoothing operators in the class $\RR^{-\rho}_{K,K'}[r]$. 
 \end{de}

\begin{rmk}\label{compsimb}
An important class of \emph{parity preserving} maps according to Definition \ref{revmap} is the following. Consider a matrix of symbols $C(U;x,\xi)$, with $U$ even in $x$, in $\Gamma^m_{K,K'}[r]\otimes\MM_2(\CCC)$ with $m\in \NNN$, if 
\begin{equation}\label{compmatr}
C(U;x,\xi)=C(U;-x,-\xi)
\end{equation}
then one can check that $\bonyw(C(U;x,\x))$ preserves the subspace of even functions. 

Moreover consider the 
system 
\[
\left\{\begin{aligned}
&\del_{\tau}\Phi^{\tau}(U)[\cdot]=\bonyw(C(U;x;\x))\Phi^{\tau}(U)[\cdot],\\
&\Phi^{0}(U)=\uno.
\end{aligned}\right.
\]
If the flow $\Phi^{\tau}$ is well defined  for $\tau\in [0,1]$, 
then it defines a family of \emph{parity preserving} maps according to Def. \ref{revmap}.
\end{rmk}

%

\subsection{Symbolic calculus}
We define the following differential operator
\begin{equation}\label{cancello}
\s(D_{x},D_{\x},D_{y},D_{\h})=D_{\x}D_{y}-D_{x}D_{\h},
\end{equation} 
where $D_{x}:=\frac{1}{\ii}\del_{x}$ and $D_{\x},D_{y},D_{\h}$ are similarly defined. 
%
%
%
If $a$ is a symbol in $\Gamma^{m}_{K,K'}[r]$ and $b$ in $\Gamma^{m'}_{K,K'}[r]$, if $U\in B^K_{s_0}(I,r)$ with $s_0$ large enough, we define 
\begin{equation}\label{sbam8}
(a\sharp b)_{\rho}(U;x,\xi):=\sum_{\ell=0}^{\rho-1}\frac{1}{ \ell!}\tonde{\frac{i}{2}\s(D_x,D_{\xi},D_{y},D_{\eta})}^{\ell}\left[a(U;x,\xi)b(U;y,\eta)\right]_{|_{x=y; \,y=\eta}},
\end{equation}
modulo symbols in $\Gamma^{m+m'-\rho}_{K,K'}[r]$. 
Assume also that the $x$-Fourier  transforms $\hat{a}(\h,\x)$, $\hat{b}(\h,\x)$ are supported
for $|\eta|\leq \delta \langle\x \rangle$ for small enough $\delta>0$. Then we define
\begin{equation}\label{sbam8infinito}
(a\sharp b)(x,\x):=
\frac{1}{4\pi^{2}}\int_{\RRR^{2}}e^{\ii x (\x^* + \h^*)}\hat{a}(\h^*,\x+\frac{\x^*}{2})\hat{b}(\x^*,\x-\frac{\h^*}{2})
d\x^* d\h^*.
\end{equation}
Thanks to the hypothesis on the support of the $x$-Fourier transform of $a$ and $b$, this integral is well defined as a distribution in $(\xi^*,\eta^*)$ acting on the $C^{\infty}$-function $(\xi^*,\eta^*)\mapsto e^{\ii x(\xi^*+\eta^*)}$.
Lemma 2.3.4 in \cite{maxdelort} guarantees that according to the notation above 
one has 
\begin{equation}\label{ronaldo}
\bonyw(a)\circ\bonyw(b)=\weyl(c), \quad c(x,\x):=(a_{\chi}\sharp b_{\chi})(x,\x),
\end{equation}
where $a_{\chi}$ and $b_{\chi}$ are defined in \eqref{nsym2}.
We state here a Proposition asserting that the symbol $(a\sharp b)_{\rho}$ is the symbol of the composition up to smoothing operators.
\begin{prop}[{\bf Composition of Bony-Weyl operators}]\label{componiamoilmondo}
Let $a$ be a symbol in $\Gamma^{m}_{K,K'}[r]$ and $b$ a symbol in $\Gamma^{m'}_{K,K'}[r]$, if $U\in B^K_{s_0}(I,r)$ with $s_0$ large enough then 
\begin{equation}\label{sharp}
\bonyw(a(U;x,\xi))\circ\bonyw(b(U;x,\xi))-\bonyw((a\sharp b)_{\rho}(U;x,\xi))
\end{equation}
belongs to the class $\mathcal{R}^{-\rho+m+m'}_{K,K'}[r]$.
\end{prop}
For the proof we refer to Proposition 2.3.2 in \cite{maxdelort}.
In the following we will need to compose smoothing operators and paradifferential ones, the next Proposition asserts that the outcome is another smoothing operator.
\begin{prop}\label{componiamoilmare}
Let $a$ be a symbol in $\Gamma^{m}_{K,K'}[r]$ with $m\geq 0$ and $R$ be a smoothing operator in $\mathcal{R}^{-\rho}_{K,K'}[r]$. If $U$ belongs to $B^{K}_{s_0}[I,r]$ with $s_0$ large enough, then  the composition operators
\begin{equation*}
\bonyw(a(U;x,\xi))\circ R(U)[\cdot]\,, \quad R(U) \circ \bonyw(a(U;x,\xi))[\cdot]
\end{equation*} 
belong to the class $\mathcal{R}^{-\rho+m}_{K,K'}[r]$.
\end{prop}
For the proof we refer to Proposition 2.4.2 in \cite{maxdelort}. We can compose smoothing operators with smoothing operators as well.
\begin{prop}
Let $R_1$ be a smoothing operator in $\mathcal{R}^{-\rho_1}_{K,K'}[r]$ and $R_2$ in $\mathcal{R}^{-\rho_2}_{K,K'}[r]$. If $U$ belongs to $B^{K}_{s_0}[I,r]$ with $s_0$ large enough, then the operator $R_1(U)\circ R_{2}(U)[\cdot]$ belongs to the class $\mathcal{R}^{-\rho}_{K,K'}[r]$, where $\rho=\min(\rho_1,\rho_2)$.
\end{prop}

We need also the following.
\begin{lemma}\label{est-prod}
Fix $K,K'\in \NNN$, $K'\leq K$ and $r>0$. Let $\{c_i\}_{i\in \NNN}$
a sequence in $\calF_{K,K'}[r]$ such that for any $i\in\NNN$
\begin{equation}\label{puta}
\asso{\partial_t^k\partial_x^{\alpha}c_i(U;x)}\leq M_i \norm{U}{k+K',s_0},
\end{equation}
for any $0\leq k\leq K-K'$ and $|\alpha|\leq 2$ and for some $s_0>0$ big enough. Then for any $s\geq s_0$ and any $0\leq k\leq K-K'$ there exists a constant $C>0$ (independent of $n$) such that for any $n\in\NNN$
\begin{equation}\label{prodotti1}
\norm{\partial_t^k\left[\bonyw\Big(\prod_{i=1}^nc_i(U;x)\Big)h\right]}{H^{s-2k}}\leq C^{n}\prod_{i=1}^n M_i
\sum_{k_1+k_2=k}\norm{U}{k_1+K',s_0}^{n}\norm{h}{k_2,s},
\end{equation}
for any $h\in C^{K-K'}_{*\RRR}(I,H^{s}(\TTT;\CCC))$.
Moreover there exists $\widetilde{C}$ such that
\begin{equation}\label{prodotti}
\|\bonyw\Big(\prod_{i=1}^n c_i\Big)h\|_{K-K',s}\leq \widetilde{C}^{n} \prod_{i=1}^nM_i\|U\|_{K,s_0}^{n}\|h\|_{K-K',s},
\end{equation}
for any $h\in C^{K-K'}_{*\RRR}(I,H^{s}(\TTT;\CCC))$.
\end{lemma}
\begin{proof}
Let $\chi$ an admissible cut-off function and set $b(U;x,\xi):=(\prod_{i=1}^nc_i(U;x))_{\chi}$. By Liebniz rule and interpolation one can prove that
\begin{equation}\label{est-prod1}
|\partial_t^k\partial_x^{\alpha}\partial_{\xi}^{\beta}b(U;x,\xi)|\leq C^{n}\norm{U}{k+K',s_0}^n\prod_{i=1}^n M_i
\end{equation}
for any $0\leq k\leq K-K'$, $\alpha\leq 2 $, any $\xi\in\RRR$ and where the constant $C$ is independent of $n$. Denoting by $\widehat{b}(U;\ell,\xi)=\widehat{b}(\ell,\xi)$ the $\ell^{th}$ Fourier coefficient of the function $b(U;x,\xi)$, from \eqref{est-prod1} with $\alpha=2$ one deduces the following decay estimate 
\begin{equation}\label{est-prod2}
|\partial_t^k\widehat{b}(\ell,\xi)|\leq C^{n} \norm{U}{k+K',s_0}^n\prod_{i=1}^n M_i\langle\ell\rangle^{-2}.
\end{equation}
With this setting one has
\begin{equation*}
\begin{aligned}
\bonyw\Big(\prod_{i=1}^n&c_i(U;x)\Big)h=\weyl(b(U;x,\xi))h\\
&=\frac{1}{2\pi}\sum_{\ell\in\ZZZ}\left(\sum_{n'\in\ZZZ}\widehat{b}\Big(\ell-n',\frac{\ell+n'}{2}\Big)\widehat{h}(n')\right)e^{\ii\ell x},
\end{aligned}\end{equation*}
where the sum is restricted to the set of indices such that $|\ell-n'|\leq\delta\frac{|\ell+n'|}{2}$ with $0<\delta<1$ (which implies that $\ell\sim n'$). Let $0\leq k\leq K-K'$, one has
\begin{equation*}\begin{aligned}
&\norm{\partial_t^k\left[\bonyw\Big(\prod_{i=1}^nc_i(U;x)\Big)h\right]}{H^{s-2k}}^2\\
\leq &C^{n}\sum_{k_1+k_2=k}\sum_{\ell\in\ZZZ}\langle \ell\rangle^{2(s-2k)}\left|\sum_{n'\in\ZZZ}\partial_t^{k_1}\left(\widehat{b}\Big(\ell-n',\frac{\ell+n'}{2}\Big)\right)\partial_t^{k_2}\Big(\widehat{h}(n')\Big)\right|^2 \\
\leq&C^{n}\prod_{i=1}^n M_i^2
\sum_{k_1+k_2=k}\norm{U}{k_1+K',s_0}^{2n}\sum_{\ell\in\ZZZ}\left(\sum_{n'\in\ZZZ}\langle\ell-n'\rangle^{-2}\langle n'\rangle^{s-2k}\left|\partial_t^{k_2}\widehat{h}(n')\right|\right)^2,
\end{aligned}\end{equation*}
where in the last passage we have used \eqref{est-prod2} and that $\ell\sim n'$. By using Young inequality for sequences one can continue the chain of inequalities above and finally obtain 
the \eqref{prodotti1}.
The estimate \eqref{prodotti} follows summing over $0\leq k\leq K-K'$.
\end{proof}

\begin{prop}\label{diff-prod-est}
Fix $K,K'\in \NNN$, $K'\leq K$ and $r>0$.
Let $\{c_i\}_{i\in \NNN}$
a sequence in $\calF_{K,K'}[r]$ satisfying the hypotheses of Lemma \ref{est-prod}.
Then the operator
\begin{equation}\label{QNN}
Q^{(n)}_{c_1,\ldots,c_n}:=\bonyw(c_{1})\circ\cdots\circ\bonyw(c_{n})-\bonyw(c_1\cdots c_n)
\end{equation}
belongs to the class $\RR^{-\rho}_{K,K'}[r]$ for any $\rho\geq 0$. 
More precisely there exists $s_0>0$ such that for any $s\geq s_0$ the following holds. For any $0\leq k\leq K-K'$ and  any $\rho\geq0$   there exists a constant $\mathtt{C}>0$ (depending on $\norm{U}{K,s_0}$, $s,s_0,\rho,k$   and independent of $n$) such that
\begin{equation}\label{qn}
\norm{\partial_t^k\left(Q^{(n)}_{c_1,\ldots,c_n} [h]\right)}{s+\rho-2k}\leq  \mathtt{C}^{n}\mathtt{M}\sum_{k_1+k_2=k}\left(\|U\|_{K'+k_1,s_0}^{n}\|h\|_{k_2,s}+
\|U\|_{K'+k_1,s_0}^{n-1}\|h\|_{k_2,s_0}\|U\|_{K'+k_1,s}\right),
\end{equation}
for any   $n\geq1$, any $h$ in $C^K_{*\R}(I,H^s(\TTT,\CCC))$, any $U\in C^K_{*\R}(I,\hcic^s)\cap B^K_s(I,r)$
and where $\mathtt{M}=M_1\cdots M_n$ (see \eqref{puta}).
\end{prop}

\begin{proof}
We proceed by induction.
For $n=1$ is trivial. 
Let us study the case $n=2$.
Since $c_1,c_2$ belong to $\calF_{K,K'}[r]$,   then $c_1 \cdot c_2=(c_1\sharp c_{2})_{\rho}$
for any $\rho>0$.
Then by \eqref{ronaldo} there exists an admissible cut-off function $\chi$ such that
\begin{equation}\label{ronaldo3}
\begin{aligned}
\bonyw(c_1)&\circ\bonyw(c_2)-\bonyw(c_1\cdot c_2)=
\bonyw(c_1)\circ\bonyw(c_2)-\bonyw( (c_1\sharp c_2)_{\rho} )\\
&=\weyl((c_1)_{\chi}\sharp (c_{2})_{\chi})-\weyl((c_1\sharp c_{2})_{\rho,\chi})=
\weyl(r_1)+\weyl(r_2),
\end{aligned}
\end{equation}
where
\begin{equation}
\begin{aligned}
r_{1}(x,\x)&=(c_1)_{\chi}\sharp (c_{2})_{\chi}-((c_1)_{\chi}\sharp (c_{2})_{\chi})_{\rho},\\
r_2(x,\x)&=((c_1)_{\chi}\sharp (c_{2})_{\chi})_{\rho}-(c_1\sharp c_{2})_{\rho,\chi}.
\end{aligned}
\end{equation}
Then, by Lemma $2.3.3$ in \cite{maxdelort} and \eqref{puta},
one has that
$r_1$ satisfies the bound
\begin{equation}\label{ronaldo4}
|\del_{t}^{k}\del_{x}^{\ell}r_1(U;x,\x)|\leq \widetilde{C}M_1 M_2\langle\x\rangle^{-\rho+\ell}\|U\|_{k+K',s_0}^{2}
\end{equation}
for any $|\ell|\leq 2$
and some universal constant $\widetilde{C}>0$ depending only on $s,s_0,\rho$. Therefore Proposition \ref{boni2} and Remark \ref{ronaldo2} 
imply that
\begin{equation}\label{ronaldo5}
\norm{ \weyl(\partial_t^kr_1(U;x,\cdot)  ) }{\calL(H^s, H^{s+\rho-2})}\leq 
\widetilde{C}M_1 M_2\|U\|_{k+K',s_0}^{2},
\end{equation}
for $\widetilde{C}>0$ possibly larger than the one in \eqref{ronaldo4}, but still depending only 
on $k,s, s_0,\rho$.
From the bound \eqref{ronaldo5} one deduces the estimate \eqref{qn}
for some $\mathtt{C}\geq 2\widetilde{C}$.
One can argue in the same way to estimate the term $\weyl(r_2)$
 in \eqref{ronaldo3}.
 
 Assume now that \eqref{qn} holds for $j\leq n-1$ for $n\geq3$.
We have that
\begin{equation}
\begin{aligned}
\bonyw(c_{1})\circ\cdots\circ\bonyw(c_{n})=\big(\bonyw(c_1\cdots c_{n-1})+Q_{n-1}\big)\circ \bonyw(c_n),
\end{aligned}
\end{equation}
where $Q_{n-1}$ satisfies condition \eqref{qn}. 
For the term $\bonyw(c_1\cdots c_{n-1})\circ\bonyw(c_n)$ one has to argue as done in the case $n=2$. 

Consider the term $Q_{n-1}\circ \bonyw(c_n)$ and let $C>0$ be the universal constant given by Lemma \ref{est-prod}.

Using the inductive hypothesis on $Q_{n-1}$ and  estimate \eqref{prodotti1} in Lemma \ref{est-prod} (in the case $n=1$) we have 
\begin{equation*}
\begin{aligned}
&\|\del_{t}^{k}\big(Q_{n-1}\circ\bonyw(c_n) h\big)\|_{s+\rho-2k}
\leq \mathtt{K}\mathtt{C}^{n-1}M_1\cdots M_{n-1}\sum_{k_1+k_2=k}\sum_{j_1+j_2=k_{2}} CM_{n}\|U\|^{n-1}_{K'+k_1,s_0}\|U\|_{K'+j_1,s_0}\|h\|_{j_2,s}\\
&\qquad+
\mathtt{K}\mathtt{C}^{n-1}M_1\cdots M_{n-1}\sum_{k_1+k_2=k}\sum_{j_1+j_2=k_{2}} CM_{n}\|U\|^{n-2}_{K'+k_1,s_0}\|U\|_{K'+k_1,s}\|U\|_{K'+j_1,s_0}\|h\|_{j_2,s_0}\\
&\qquad\leq \mathtt{K}\mathtt{M}\mathtt{C}^{n-1}C \sum_{k_1=0}^{k}\sum_{j_1=0}^{k-k_1}
\|U\|^{n}_{K'+k_1+j_1,s_0}\|h\|_{k-k_1-j_1,s}\\
&\qquad+ \mathtt{K}\mathtt{M}\mathtt{C}^{n-1}C \sum_{k_1=0}^{k}\sum_{j_1=0}^{k-k_1}
\|U\|^{n-1}_{K'+k_1+j_1,s_0}\|U\|_{K'+k_1+j_1,s}\|h\|_{k-k_1-j_1,s_0}\\
&\qquad\leq  \mathtt{K}\mathtt{M}\mathtt{C}^{n-1}C \sum_{m=0}^{k}
(\|U\|^{n}_{K'+m,s_0}\|h\|_{k-m,s}+
\|U\|^{n-1}_{K'+m,s_0}\|U\|_{K'+m,s}\|h\|_{k-m,s_0})(m+1),
\end{aligned}
\end{equation*}
for constant $\mathtt{K}$ depending only on $k$.
This implies  
\eqref{qn} by choosing $\mathtt{C}> (k+1)C\mathtt{K}$.
\end{proof}

\begin{coro}\label{esponanziale} Fix $K,K'\in \NNN$, $K'\leq K$ and $r>0$.
Let $s(U;x)$ and $z(U;x)$ be symbols in the class $\calF_{K,K'}[r]$. Consider the following two matrices 
\begin{equation}
S(U;x):=\left(\begin{matrix}s(U;x) & 0\\  0 & \ol{s(U;x)}\end{matrix}\right),\quad 
Z(U;x):=\left(\begin{matrix}0 & z(U;x)\\  \ol{z(U;x)} & 0\end{matrix}\right)\in \calF_{K,K'}[r]\otimes\MM_{2}(\CCC).
\end{equation}
Then one has the following
\begin{equation*}
\begin{aligned}
&\exp\left\{\bonyw(S(U;x))\right\}-\bonyw(\left\{\exp{S(U;x)}\right\}) \in \RR^{-\rho}_{K,K'}[r]\otimes\MM_2(\CCC), \\
&\exp\left\{\bonyw(Z(U;x))\right\}-\bonyw(\left\{\exp{Z(U;x)}\right\}) \in \RR^{-\rho}_{K,K'}[r]\otimes\MM_2(\CCC),
\end{aligned}
\end{equation*}
for any $\rho\geq 0$.
\end{coro}
\begin{proof}
Let us prove the result for the matrix $S(U;x)$.

Since $s(U;x)$ belongs to $\calF_{K,K'}[r]$ then there exists $s_0>0$ such that if $U\in B^K_{s_0}(I,r)$, then there is  a constant $\mathtt{N}>0$ such that
$$\asso{\partial_t^k\partial_x^{\alpha}s(U;x)}\leq \mathtt{N} \norm{U}{k+K',s_0},$$
for any $0\leq k\leq K-K'$ and $|\alpha|\leq 2$.
By definition one has 
\begin{equation*}
\begin{aligned}
\exp&\Big(\bonyw(S(U;x))\Big)=\sum_{n=0}^{\infty}\frac{\big(\bonyw(S(U;x))\big)^n}{n!}\\
&=\sum_{n=0}^{\infty}\frac{1}{n!}\left(\begin{matrix}
& \big(\bonyw(s(U;x))\big)^n & 0\\
& 0 & \big(\bonyw(\ol{s(U;x)})\big)^n
\end{matrix}\right),
\end{aligned}
\end{equation*}
on the other hand 
\begin{equation*}
\begin{aligned}
\bonyw&\Big(\exp\big(S(U;x)\big)\Big)=\sum_{n=0}^\infty\frac{1}{n!}\bonyw\left(\begin{matrix}
&\big[s(U;x)\big]^n & 0\\
& 0 & \big[\ol{s(U;x)}\big]^n
\end{matrix}\right)\\
&=\sum_{n=0}^\infty\frac{1}{n!}\left(\begin{matrix}
&\bonyw\big(\big[s(U;x)\big]^n\big) & 0\\
& 0 &\bonyw \big(\big[\ol{s(U;x)}\big]^n\big)
\end{matrix}\right).
\end{aligned}
\end{equation*}
We  argue componentwise. Let $h$ be a function in  $C^K_{*\R}(I,H^s(\TTT,\CCC))$, then using Proposition
 \ref{diff-prod-est}, one has 
 \begin{equation*}
 \begin{aligned}
& \norm{\sum_{n=0}^{\infty}\frac{1}{n!}\partial_t^k\Big(\big[\bonyw(s(U;x))\big]^n[h]-\bonyw\big(s(U;x)^n\big)[h]\Big)}{s+\rho-2k}\leq\\
& \sum_{n=1}^{\infty}\frac{\mathtt{C}^n\mathtt{N}^n}{n!}\sum_{k_1+k_2=k}
\left(\norm{U}{K'+k_1,s_0}^n\norm{h}{k_2,s}+\norm{U}{K'+k_1,s_0}^{n-1}\norm{h}{k_2,s_0}\norm{U}{K'+k_1,s}\right)
\leq\\ &
\sum_{k_1+k_2=k}\left(\norm{U}{K'+k_1,s_0}\norm{h}{k_2,s}+\norm{U}{K'+k_1,s}\norm{h}{k_2,s_0}\right)\sum_{n=1}^{\infty} \frac{\mathtt{C}^n\mathtt{N}^n}{n!}\norm{U}{K'+k_1,s_0}^{n-1}.
 \end{aligned}\end{equation*}
Therefore we have proved the \eqref{porto20} with constant 
$$C=\sum_{n=1}^{\infty} \frac{\mathtt{C}^n\mathtt{N}^n}{n!}\norm{U}{K'+k_1,s_0}^{n-1}=\frac{\exp(\mathtt{C}\mathtt{N}  \norm{U}{K'+k_1,s_0})-1}{\norm{U}{K'+k_1,s_0}}.$$
For the other non zero component of the matrix the argument is the same.

In order to simplify  the notation,  set $z(U;x)=z$ and $\ol{z(U;x)}=\ol{z}$, therefore for the matrix $Z(U;x)$, by definition, one has 
\begin{equation*}
\bonyw\left(\exp(Z(U;x))\right)=\bonyw\left(\sum_{n=0}^{\infty}\frac{1}{n!}\left(\begin{matrix}
& |z|^{2n} & |z|^{2n+1}z\\
& |z|^{2n+1} \ol{z} & |z|^{2n}
\end{matrix}\right)\right).
\end{equation*}
On the other hand, setting $A^n_{z,\bar{z}}=\big(\bonyw(z)\circ\bonyw(\bar{z})\big)^n$ and $B^n_{z,\bar{z}}=A^n_{z,\bar{z}}\circ\bonyw(z)$, one has
\begin{equation*}
\exp\left(\bonyw(Z(U;x))\right)=\sum_{n=0}^{\infty}\frac{1}{n!}\left(\begin{matrix}
& A^n_{z,\bar{z}} & B^n_{z,\bar{z}}\\
& \ol{B^n_{z,\bar{z}}} & \ol{A^n_{z,\bar{z}}}
\end{matrix}\right).
\end{equation*}
Therefore one can study each component of the matrix $\exp\left(\bonyw(Z(U;x))\right)-\bonyw\left(\exp{Z(U;x)}\right)$ in the same way as done in the case of the matrix $S(U,x)$.
\end{proof}

\section{Paralinearization of the equation}\label{PARANLS}

In this section we give a paradifferential  formulation of the equation \eqref{NLS}. In order to paralinearize the equation \eqref{NLS} we need to ``double" the variables. We consider a system of equations for the variables $(u^+,u^-)$ in $H^{s}\times H^{s}$ which is equivalent to \eqref{NLS} if $u^+=\bar{u}^-$. More precisely we give the following definition.

\begin{de}\label{vectorNLS}
Let $f$ be
the  $C^{\infty}(\CCC^3;\CCC)$ function
in the equation \eqref{NLS}.
 We define the 
``vector'' NLS as
\begin{equation}\label{sistemaNLS}
\begin{aligned}
&\partial_t U=\ii E\left[\Lambda U
+\mathtt{F}(U,U_x,U_{xx})\right],  \quad U\in H^{s}\times H^{s},
\\
&\mathtt{F}(U,U_x,U_{xx}):=
\left(\begin{matrix} f_1(U,{ U}_{x},{ U}_{xx})\\ 
{{f_2({U },{U}_{x},{U}_{xx})}}
\end{matrix}
\right),
\end{aligned}
\end{equation}
where 
\[
\mathtt{F}(Z_1,Z_{2},Z_{3})=\left(\begin{matrix}
f_1(z_{1}^{+},z_{1}^{-},z_{2}^{+},z_{2}^{-},z_{3}^{+},z_{3}^{-})\\
f_2(z_{1}^{+},z_{1}^{-},z_{2}^{+},z_{2}^{-},z_{3}^{+},z_{3}^{-})\end{matrix}\right), \quad Z_{i}=\left(\begin{matrix}
z_{i}^{+} \\
z_{i}^{-}
\end{matrix}\right), \quad i=1,2,3,
\] 
extends $(f,\ol{f})$ in the following sense.
The functions $f_{i}$ for $i=1,2$ are $C^{\infty}$ on $\CCC^{6}$ (in the real sense). Moreover
one has the following:
\begin{equation}\label{sistNLS1}
\begin{aligned}
\left(\begin{matrix}
f_1(z_1,\bar{z}_1,z_2,\bar{z}_2,z_{3},\bar{z}_{3}) \\
f_2(z_1,\bar{z}_1,z_2,\bar{z}_2,z_{3},\bar{z}_{3})
\end{matrix}
\right)=\left(\begin{matrix}
f(z_1,z_2,z_{3}) \\
\ol{f(z_1,z_2,z_{3})}
\end{matrix}
\right),
\end{aligned}
\end{equation}
and
\begin{equation}\label{sistNLS2}
\begin{aligned}
& \del_{z^{+}_{3}}f_{1}=\del_{z^{-}_{3}}f_2, \quad  \del_{z^{+}_{i}}f_1=\ol{\del_{z^{-}_{i}}f_{2}}, \;\; i=1,2,\;\;\;
\del_{z^{-}_{i}}f_1=\ol{\del_{z^{+}_{i}}f_2}, \quad i=1,2,3 \\ 
&\del_{\overline {z^{+}_{i}}}f_{1} =\del_{\overline {z^{+}_{i}}}f_{2}=\del_{\overline {z^{-}_{i}}}f_{1} 
=\del_{\overline {z^{-}_{i}}}f_{2}=0 \;\;
\end{aligned}
\end{equation}
 where $ \del_{\overline {z^{\s}_{j}}}= \del_{{\rm Re}\,z^{\s}_j}+ \ii \del_{{\rm Im}\,z^{\s}_j}, \;\; \s=\pm.$
\end{de}

\begin{rmk}
In the case that $f$ has the form
\[
f(z_1,z_2,z_3)=C z_{1}^{\al_1}{\bar{z}}_{1}^{\be_1}z_{2}^{\al_2}\bar{z}_{2}^{\be_{2}}
\]
for some $C\in\CCC$, $\al_{i},\be_i\in \NNN$ for $i=1,2$, a possible extension is the following:
\[
\begin{aligned}
&f_{1}(z^{+}_1,z^{-}_1,z^{+}_2,z^{-}_2)=C (z^{+}_1)^{
\al_1}(z^{-}_1)^{\be_1}(z^{+}_2)^{\al_2}(z^{-}_2)^{\be_2},\\
& f_{2}(z^{+}_1,z^{-}_1,z^{+}_2,z^{-}_2)=\ol{C} (z^{-}_1)^{
\al_1}(z^{+}_1)^{\be_1}(z^{-}_2)^{\al_2}(z^{+}_2)^{\be_2}.
\end{aligned}
\]
\end{rmk}

\begin{rmk} 
Using \eqref{sistNLS1} one deduces the following relations between  the derivatives  of $f$ and $f_j$ with $j=1,2$:
\begin{equation}\label{derivate}
\begin{aligned}
&\partial_{z_i}f(z_1,z_2,z_3)=(\partial_{z_i^+}f_1)(z_1,\bar{z}_1,z_2,\bar{z}_2,z_3,\bar{z}_3)\\
&\partial_{\bar{z}_i}f(z_1,z_2,z_3)=(\partial_{z_i^-}f_1)(z_1,\bar{z}_1,z_2,\bar{z}_2,z_3,\bar{z}_3)\\
& \ol{\partial_{\bar{z}_i}f(z_1,z_2,z_3)}={(\partial_{z_i^+}f_2)(z_1,\bar{z}_1,z_2,\bar{z}_2,z_3,\bar{z}_3)}\\
& \ol{\partial_{{z}_i}f(z_1,z_2,z_3)}={(\partial_{z_i^-}f_2)(z_1,\bar{z}_1,z_2,\bar{z}_2,z_3,\bar{z}_3)}.
\end{aligned}
\end{equation}
\end{rmk}

In the rest of the paper we shall use the following notation. 
Given a  function
$g(z_{1}^{+},z_{1}^{-},z_{2}^{+},z_{2}^{-},z_{3}^{+},z_{3}^{-})$ defined on $\CCC^{6}$ which is differentiable
in the real sense, we shall write
\begin{equation}\label{notazione}
\begin{aligned}
&(\del_{\del_{x}^{i}u}g)(u,\bar{u},u_{x},\bar{u}_{x},u_{xx},\bar{u}_{xx}):=(\del_{z_{i+1}^{+}}g)(u,\bar{u},u_{x},\bar{u}_{x},u_{xx},\bar{u}_{xx}),\\
&(\del_{\ol{\del_{x}^{i}u}}g)(u,\bar{u},u_{x},\bar{u}_{x},u_{xx},\bar{u}_{xx}):=(\del_{z_{i+1}^{-}}g)(u,\bar{u},u_{x},\bar{u}_{x},u_{xx},\bar{u}_{xx}),\quad i=0,1,2.
\end{aligned}
\end{equation}

By Definition \ref{vectorNLS} one has that
equation \eqref{NLS} is equivalent
to the system \eqref{sistemaNLS} on the subspace $\hcic^{s}$.

We state the Bony paralinearization lemma, which is adapted to our case from Lemma 2.4.5 of \cite{maxdelort}.
\begin{lemma}[\bf{Bony paralinearization of the composition operator}]\label{paralinearizza}
Let $f$ be 
a  complex-valued function of  class $C^{\infty}$ in the real sense 
defined in a ball centered at $0$ of radius $r>0$, in $\CCC^6$, vanishing at $0$ at order $2$. 
There exists  a $1\times 2$ matrix of symbols $q\in\Gamma^2_{K,0}[r]$ 
and a $1\times 2$ matrix of  smoothing operators $Q(U)\in\RR^{-\rho}_{K,0}[r]$, for any $\rho$, such that
\begin{equation}\label{finalmentesiparalinearizza1}
f(U,U_x,U_{xx})=\bonyw(q(U,U_x,U_{xx};x,\x))[U]+ Q(U)U.
\end{equation}
Moreover the symbol $q(U;x,\x)$ has the form
\begin{equation}\label{PPPPPP1}
q(U;x,\x):=d_{2}(U;x)(\ii\x)^{2}+d_1(U;x)(\ii\x)+d_{0}(U;x), 
\end{equation}
where $d_{j}(U;x)$ are $1\times2$ matrices of symbols in $ \calF_{K,0}[r]$, for $j=0,1,2$.
\end{lemma}
\begin{proof}
By the paralinearization formula of Bony, we know that
\begin{equation}\label{finalmentesiparalinearizza2}
f(U,U_{x},U_{xx})=T_{D_U f} U+T_{D_{U_{x}}f} U_x+ T_{D_{U_{xx}}f}U_{xx}+R_0(U)U,
\end{equation}
where $R_0(U)$ satisfies estimates   \eqref{porto20} and where 
\begin{equation*}
\begin{aligned}
&T_{D_U f} U=\frac{1}{2\pi}\int e^{\ii(x-y)\xi} \chi(\langle\xi\rangle^{-1} D)[c_U(U;x,\xi)]U(y)dyd\xi,\\
&T_{D_{U_x} f} U_{x}=\frac{1}{2\pi}\int e^{\ii(x-y)\xi} 
\chi(\langle\xi\rangle^{-1} D)[c_{U_x}(U;x,\xi)]U(y)dyd\xi,\\
&T_{D_{U_{xx}} f}U_{xx}=\frac{1}{2\pi}\int e^{\ii(x-y)\xi} \chi(\langle\xi\rangle^{-1} D)[c_{U_{xx}}(U;x,\xi)]U(y)dyd\xi,
\end{aligned}\end{equation*}
with 
\begin{equation}\label{quaderno}
\begin{aligned}
&c_{U}(U;x,\xi)= D_U f,\\
&c_{U_x}(U;x,\xi)=D_{U_x} f(\ii\xi),\\
&c_{U_{xx}}(U;x,\xi)=D_{U_{xx}} f(\ii\xi)^2,\\
\end{aligned}\end{equation}
for some $\chi\in\ C^{\infty}_0(\R)$ with small enough support and equal to $1$ close to $0$. Using \eqref{bambola5} we define the $x$-periodic function $b_i(U;x,\xi)$, for $i=0,1,2$, 
through its Fourier coefficients 
\begin{equation}\label{diego}
\hat{b}_i(U;n,\xi):= \hat{c}_{U_i}(U;n,\xi-n/2)
\end{equation}
where $U_{i}:=\del_{x}^{i}U$.
In the same way we define the function $d_{i}(U;x,\x)$, for $i=0,1,2$, 
as
\begin{equation}\label{armando}
\hat{d}_i(U;n,\xi):= \chi\tonde{n\langle\xi-n/2\rangle^{-1}}
\hat{c}_{U_i}(U;n,\xi-n/2).
\end{equation}
We have that
 $T_{D_U f} U={\rm{Op}}^{W}(d_0(U,\xi))U$. 
 We observe the following
\begin{equation}\label{mara}
\hat{d}_0(U;n,\xi)=\chi\tonde{n\langle\xi\rangle^{-1}}\widehat{D_Uf}(n)+\tonde{\chi\tonde{n\langle\xi-n/2\rangle^{-1}}-n\langle\xi\rangle^{-1}}\widehat{D_Uf}(n)
\end{equation}
therefore if the support of $\chi$ is small enough, thanks to Lemma \ref{equiv}, we obtained
\begin{equation}\label{ciaone}
T_{D_U f} U=\bonyw(b_0(U;x,\x))U+R_{1}(U)U,
\end{equation}
for some smoothing reminder $R_1(U)$.
Reasoning in the same way we get
\begin{equation}\label{ciaone1}
\begin{aligned}
& T_{D_{U_x} f} U_{x}=\bonyw\big(b_1(U;\xi)\big)U+R_{2}(U)U\\
& T_{D_{U_{xx}} f} U_{xx}=\bonyw\big(b_2(U;\xi))\big)U+R_{3}(U)U.
\end{aligned}
\end{equation}
The theorem is proved defining $Q(U)=\sum_{k=0}^{3}R_{k}(U)$ and $q(U;x,\xi)=b_2(U;\xi)+b_1(U;\xi)+b_0(U;\xi)$. Note that the symbol $q$ satisfies conditions \eqref{PPPPPP1} by \eqref{quaderno} and formula \eqref{bambola5}.
\end{proof}

We have the following Proposition.

\begin{prop}[\bf{Paralinearization of the system}]\label{montero}
There are  
 a matrix $A(U;x,\xi)$ in $\Gamma^2_{K,0}[r]\otimes\MM_{2}(\CCC)$
 and a smoothing operator $R$ in $\RR^{-\rho}_{K,0}[r]\otimes\MM_{2}(\CCC)$, for any $K, r>0$ and $\rho\geq0$
such that the system \eqref{sistemaNLS}
is equivalent to
\begin{equation}\label{6.666para}
\partial_t U:=\ii E\Big[\Lambda U
+\bonyw(A(U;x,\x))[U]
+R(U)[U]\Big],
 \end{equation}
 on the subspace $\mathcal{U}$ (see \eqref{Hcic} and Def. \ref{vectorNLS})
 and where $\Lambda$ is defined in \eqref{DEFlambda} and \eqref{NLS1000}.
 Moreover the operator
 $R(U)[\cdot]$ has the form \eqref{vinello}, 
 the matrix $A$ has the form  \eqref{prodotto}, i.e.
\begin{equation}\label{matriceA}
 A(U;x,\x)
:=\left(
\begin{matrix}
a(U;x,\x) & b(U;x,\x)\\
\ol{b(U;x,-\x)} & \ol{a(U;x,-\x)}
\end{matrix}
\right)
 \in \Gamma^{2}_{K,0}[r]\otimes\MM_{2}(\CCC)
 \end{equation}
with $a$, $b$ in $\Gamma^{2}_{K,0}[r]$. 
In particular we have that
\begin{equation}\label{PPPPA}
 A(U;x,\x)=A_{2}(U;x)(\ii\x)^{2}+A_{1}(U;x)(\ii\x)+A_0(U;x), \quad A_{i}\in \calF_{K,0}[r]\otimes\MM_{2}(\CCC), \;\; i=0,1,2.
\end{equation}
\end{prop}

\begin{proof}
The functions $f_{1},f_{2}$ in \eqref{sistemaNLS} satisfy the hypotheses of Lemma \ref{paralinearizza} for any $r>0$. 
Hence 
the result follows
by setting $q(U;x,\x)=:(a(U;x,\x),b(U;x,\x))$.
\end{proof}

In the following we study some properties of the system in \eqref{6.666para}.

We first prove some lemmata which 
translate the Hamiltonian Hyp. \ref{hyp1}, parity-preserving Hyp. \ref{hyp2} and global ellipticity Hyp. \ref{hyp3} in the paradifferential  setting.
\begin{lemma}[{\bf Hamiltonian structure}]\label{struttura-ham-para}\
Assume that $f$ in \eqref{NLS} satisfies 
 Hypothesis \ref{hyp1}. 
Consider the matrix
$A(U;x,\x)$ in \eqref{matriceA}
given by Proposition \ref{montero}. Then the term
\[
A_{2}(U;x)(\ii\x)^{2}+A_{1}(U;x)(\ii\x)
\]
in \eqref{PPPPA}
satisfies conditions \eqref{quanti801}. More explicitly one has
\begin{equation}\label{matriceAA}
 A_{2}(U;x)
:=\left(
\begin{matrix}
a_{2}(U;x) & b_2(U;x)\\
\ol{b_2(U;x)} & {a_2(U;x)}
\end{matrix}
\right), \quad
  A_{1}(U;x)
:=\left(
\begin{matrix}
a_{1}(U;x) & 0\\
0 & \ol{a_1(U;x)}
\end{matrix}
\right),
\end{equation}
with $a_{2},a_1,b_2\in \calF_{K,0}[r]$ and $a_{2}\in \RRR$.
\end{lemma}

\begin{proof}
Recalling the notation introduced in \eqref{notazione} 
we
shall write
\begin{equation}\label{derivate2}
\del_{\del_{x}^{i}u}f:=\del_{z_{i+1}^{+}}f_1, \quad 
\del_{\ol{\del_{x}^{i}u}}f:=\del_{z_{i+1}^{-}}f_1, \quad i=0,1,2,
\end{equation}
when restricted to the real subspace $\mathcal{U}$ (see \eqref{Hcic}).
Using conditions 
\eqref{sistNLS1}, \eqref{sistNLS2} and \eqref{derivate}
one has
 that 
\begin{equation}\label{computer}
\begin{aligned}
&\left(\begin{matrix}
f(u,u_x,u_{xx}) \\
\ol{f(u,u_x,u_{xx}) }
\end{matrix}
\right)=
\left(
\begin{matrix}
f_{1}(U,U_x,U_{xx})\\
{f_2(U,U_x,U_{xx})}
\end{matrix}
\right)\\
&\qquad =\bony\left[\left(  \begin{matrix} \del_{u_{xx}}f & \del_{\bar{u}_{xx}}f \\
\ol{\del_{\bar{u}_{xx}}f} & \ol{\del_{u_{xx}}f}
\end{matrix}\right)(\ii\x)^{2}\right]U+
\bony\left[\left(  \begin{matrix} \del_{u_x}f & \del_{\bar{u}_x}f \\
\ol{\del_{\bar{u}_x}f} & \ol{\del_{u_x}f}
\end{matrix}\right)(\ii\x)\right]U+R(U)[U]
\end{aligned}
\end{equation}
where $R(U)$ belongs to $\RR^{0}_{K,0}[r]$. By Hypothesis \ref{hyp1} we have that

\begin{equation}\label{5}
\begin{aligned}&\del_{u_{xx}}f=
-\del_{u_{x}\bar{u}_{x}}F,\\   
&\del_{\bar{u}_{xx}}f=-\del_{\bar{u}_{x}\bar{u}_{x}}F,\\
&\del_{u_{x}}f=-
\frac{\rm d}{{\rm d}x}\left[\del_{u_{x}\bar{u}_{x}}F\right]-\del_{u\bar{u}_{x}}F
+\del_{u_{x}\bar{u}}F,\\
&\del_{\bar{u}_{x}}f=-
\frac{\rm d}{{\rm d}x}\left[\del_{\bar{u}_{x}\bar{u}_{x}}F\right].
\end{aligned}
\end{equation}
We now pass to the Weyl quantization in the following way. Set
\[
c(x,\x)=\del_{u_{xx}}f(x)(\ii\x)^{2}+\del_{u_{x}}f(x)(\ii\x).
\]
Passing to the Fourier side
we have that
\[
\widehat{c}(j,\x-\frac{j}{2})=\widehat{(\del_{u_{xx}}f)}(j)(\ii\x)^{2}+\Big[\widehat{(\del_{u_{x}}f)}(j)-(\ii j)\widehat{(\del_{u_{xx}}f)}(j)\Big](\ii\x)
+\Big[
\frac{(\ii j)^{2}}{4}\widehat{(\del_{u_{xx}}f)}(j)-\frac{(\ii j)}{2}\widehat{(\del_{u_{x}}f)}(j)
\Big],
\]
therefore by using formula \eqref{bambola5}  we have that
$\bony(c(x,\x))=\bonyw(a(x,\x))$,
where
\[
a(x,\x)=\del_{u_{xx}}f(x)(\ii\x)^{2}+[\del_{u_{x}}f(x)-\frac{{\rm d}}{{\rm d}x}(\del_{u_{xx}}f)](\ii\x)+\frac{1}{4}\frac{{\rm d}^2}{{\rm d}x^2}(\del_{u_{xx}}f)-\frac{1}{2}\frac{{\rm d}}{{\rm d}x}(\del_{u_x}f).
\]

Using the relations in \eqref{5}
we obtain
a matrix $A$ as in \eqref{matriceAA}, and in particular we have
\begin{equation}\label{quantiselfi}
a_{2}(U;x)= -\del_{u_{x}\bar{u}_{x}}F, \quad a_{1}(U;x)=-\del_{u\bar{u}_{x}}F
+\del_{u_{x}\bar{u}}F, \quad b_{2}(U;x)=-\del_{\bar{u}_{x}\bar{u}_{x}}F.
\end{equation}
Since $F$ is real then $a_{2}$ is real, while $a_{1}$ is purely imaginary. This implies conditions \eqref{quanti801}.
\end{proof}

\begin{lemma}[{\bf Parity preserving structure}]\label{struttura-rev-para}
Assume that $f$ in \eqref{NLS} satisfies 
 Hypothesis \ref{hyp2}. 
Consider the matrix
$A(U;x,\x)$ in \eqref{matriceA}
given by Proposition \ref{montero}.
One has that 
$A(U;x,\x)$ has the form \eqref{PPPPA} where
\begin{equation}\label{matriceAAA}
 \begin{aligned}
 &A_{2}(U;x)
:=\left(
\begin{matrix}
a_{2}(U;x) & b_2(U;x)\\
\ol{b_2(U;x)} & {a_2(U;x)}
\end{matrix}
\right), \quad
  A_{1}(U;x)
:=\left(
\begin{matrix}
a_{1}(U;x) & b_{1}(U;x)\\
\ol{b_{1}(U;x)} & \ol{a_1(U;x)}
\end{matrix}
\right),\\
& A_{0}(U;x)
:=\left(
\begin{matrix}
a_{0}(U;x) & b_0(U;x)\\
\ol{b_0(U;x)} & \ol{a_0(U;x)}
\end{matrix}\right),
\end{aligned}
\end{equation}
with $a_{2},b_{2},a_1,b_1, a_0, b_0\in \calF_{K,0}[r]$ such that, for $U$ even in $x$,
the following holds:
\begin{subequations}\label{simbolirev}
\begin{align}
& a_{2}(U;x)=a_{2}(U;-x), \quad  b_{2}(U;x)=b_{2}(U;-x),\\
& a_{1}(U;x)=-a_{1}(U;-x), \quad  b_{1}(U;x)=-b_{1}(U;-x),\\
 & a_{0}(U;x)=a_{0}(U;-x), \quad  b_{0}(U;x)=b_{0}(U;-x),
 \quad U\in {\bf H}^{s}_{e},
\end{align}
\end{subequations}
and 
\begin{equation}\label{simbolirev2}
a_{2}(U;x)\in \RRR.
\end{equation}

The matrix $R(U)$ in \eqref{6.666para} is parity preserving 
according to Definition \ref{revmap}.

\end{lemma}

\begin{proof}
Using the same notation introduced in the proof of Lemma  \ref{struttura-ham-para} (recall \eqref{derivate})
we have that formula \eqref{computer}
holds. 
Under the Hypothesis \ref{hyp2} 
one has that 
the functions $\del_u f, \del_{\ol{u}}f, \del_{u_{xx}}f, \del_{\bar{u}_{xx}}f$
are \emph{even} in $x$ while $\del_{u_{x}}f, \del_{\bar{u}_{x}}f$ are \emph{odd} in $x$.
Passing to the Weyl quantization by formula \eqref{bambola5} we get
\begin{equation}\label{quantiselfi2}
\begin{aligned}
&a_{2}(U;x)= \del_{u_{xx}}f, \\
&a_{1}(U;x)=\del_{u_{x}}f-\del_{x}(\del_{u_{xx}}f),  \\
& a_0(U;x)=\del_{u}f+\frac14 \del_x^2(\del_{u_{xx}}f)-\frac12\del_x(\del_{u_x}f),
\end{aligned}
\begin{aligned}
&b_{2}(U;x)=\del_{\bar{u}_{xx}}f,\\
&b_{1}(U;x)=\del_{\bar{u}_{x}}f-\del_{x}(\del_{\bar{u}_{xx}}f),\\
&b_0(U;x)=\del_{\bar{u}}f+\frac14 \del_x^2(\del_{\bar{u}_{xx}}f)-\frac12\del_x(\del_{\bar{u}_x}f)
\end{aligned}
\end{equation}
which imply conditions \eqref{simbolirev}, while  \eqref{simbolirev2} is implied by item $2$
of Hypothesis \ref{hyp2}. 
The term $R$ is parity preserving by difference.
\end{proof}
\begin{lemma}[\bf Global ellipticity]\label{simboli-ellittici}
Assume that $f$ in \eqref{NLS} satisfies  Hyp. \ref{hyp1} (respectvely Hyp. \ref{hyp2}). If $f$ satisfies also Hyp. \ref{hyp3} then the matrix $A_2(U;x)$ in \eqref{matriceAA} 
(resp.  in  \eqref{matriceAAA}) is such that
\begin{equation}\label{determinante}
\begin{aligned}
& 1+a_{2}(U;x)\geq \mathtt{c_1}\\
& (1+a_{2}(U;x))^{2}-|b_{2}(U;x)|^{2}\geq \mathtt{c_2}>0,
\end{aligned}
\end{equation}
where $\mathtt{c_1}$ and $\mathtt{c_2}$ are the constants given in \eqref{constraint} and \eqref{constraint2}.
\end{lemma} 
\begin{proof}
It follows from \eqref{quantiselfi} in the case of Hyp. \ref{hyp1} and from \eqref{quantiselfi2} in the case of Hyp. \ref{hyp2}.
\end{proof}

\begin{lemma}[{\bf Lipschitz estimates}]\label{stimelip-dei-simboli}
Fix $r>0$, $K>0$ and consider the matrices $A$ and $R$ given in Proposition \ref{montero}.  Then there exists $s_0>0$ such that for any $s\geq s_0$ the following holds true. For any $U,V\in C_{*\RRR}^{K}(I;\hcic^{s})\cap B^K_{s_0}(I,r)$ there are constants $C_1>0$ and $C_2>0$, depending on $s$, $\|U\|_{K,s_0}$ and $\|V\|_{K,s_0}$, such that for any $H\in C_{*\RRR}^{K}(I;\hcic^{s})$
one has
\begin{equation}\label{nave77}
\|\bonyw(A(U;x,\x))[H]-\bonyw(A(V;x,\x))[H]\|_{{K},{s-2}}\leq C_{1}
\|H\|_{{K},{s}}\|U-V\|_{{K},{s_0}}
\end{equation}
\begin{equation}\label{nave101}
\|R(U)[U]-R(V)[V]\|_{{K},{s+\rho}}\leq
 C_2 (\| U\|_{{K},{s}}+\| V\|_{{K},{s}})\|U-V\|_{{K},s},
\end{equation}
for any $\rho\geq0$.
\end{lemma}

\begin{proof}
We prove bound \eqref{nave77} on each component of the matrix $A$ in \eqref{matriceA} in the case that $f$ satisfies Hyp. \ref{hyp2}. The Hamiltonian case of Hyp. \ref{hyp1} follows by using the same arguments.
From the proof of Lemma \ref{struttura-rev-para} we know that the symbol $a(U;x,\xi)$ of the matrix in \eqref{matriceA} is such that $a(U;x,\x)=a_{2}(U;x)(\ii\x)^{2}+a_{1}(U;x)(\ii\x)+a_{0}(U;x)$
where $a_{i}(U;x)$
for $i=0,1,2$ are given in \eqref{quantiselfi2}.

By Remark \ref{ronaldo10} there exists $s_0>0$ such that for any $s\geq s_0$ one has
\begin{equation}\label{vialli}
\|\bonyw\big((a_{2}(U;x)-a_{2}(V;x))(\ii \x)^{2}\big)h\|_{K,s-2}\leq C
\sup_{\x}\langle\x\rangle^{-2}\|(a_{2}(U;x)-a_{2}(V;x))(\ii \x)^{2}\big)\|_{K,s_0}\|h\|_{K,s}.
\end{equation}
with $C$ depending on $s,s_0$. Let $U,V\in C_{*\RRR}^{K}(I;\hcic^{s})\cap B^K_{s_0+2}(I,r)$,
by Lagrange theorem, recalling the relations in \eqref{derivate}, \eqref{notazione} and \eqref{derivate2}, 
one has that
\begin{equation}\label{peruzzi}
\begin{aligned}
\big(a_{2}(U;x)-a_{2}(V;x)\big)(\ii\x)^{2}&=\big((\del_{u_{xx}}f_1)(U,U_x,U_{xx})-
(\del_{u_{xx}}f_{1})(V,V_x,V_{xx})\big)(\ii\x)^{2}\\
&=(\del_U\del_{u_{xx}}f_1)(W^{(0)},U_x,U_{xx})(U-V)(\ii\x)^{2}+\\
&+
(\del_{U_{x}}\del_{u_{xx}}f_1)(V,W^{(1)},U_{xx})(U_x-V_x)(\ii\x)^{2}+\\
&+(\del_{U_{xx}}\del_{u_{xx}}f_1)(V,V_x,W^{(2)})(U_{xx}-V_{xx})(\ii\x)^{2}
\end{aligned}
\end{equation}
where $W^{(j)}=\del_{x}^{j}V+t_{j}(\del_{x}^{j}U-\del_{x}^{j}V)$,  for some $t_{j}\in [0,1]$ and  $j=0,1,2$.
Hence, for instance, the first summand of \eqref{peruzzi} can be estimated as follows
\begin{equation}
\begin{aligned}
\sup_{\x}\langle\x\rangle^{-2}&\|(\del_U\del_{u_{xx}}f_1)(W^{(0)},U_x,U_{xx})(U-V)(\ii\x)^{2}\|_{K,s_0}\\
&\leq
C_1\|U-V\|_{K,s_0}\sup_{U,V\in B_{s_0+2}^{K}(I,r)}\|(\del_U\del_{u_{xx}}f_1)(W^{(0)},U_x,U_{xx})\|_{K,s_0}\\
&\leq C_2 \|U-V\|_{K,s_0},
\end{aligned}
\end{equation}
where $C_1$ depends on $s_0$ and $C_2$ depends only on $s_0$ and $\|U\|_{K,s_0+2},\|V\|_{K,s_0+2}$ and where we have used a Moser type estimates on composition operators on $H^{s}$ since $f_1$ belongs to $C^{\infty}(\mathbb{C}^6;\mathbb{C})$. We refer the reader to Lemma $A.50$ of \cite{FP} for a complete statement (see also \cite{Ba2}, \cite{Moser-Pisa-66}).  The other terms in the r.h.s. of \eqref{peruzzi} can be treated in the same way. Hence from \eqref{vialli} and the discussion above we have obtained
\begin{equation}
\|\bonyw\big((a_{2}(U;x)-a_{2}(V;x))(\ii \x)^{2}\big)h\|_{K,s-2}\leq C \|U-V\|_{K,s_0+2} \| h\|_{K,s},
\end{equation}
with $C$ depending  on $s$ and $\|U\|_{K,s_0+2},\|V\|_{K,s_0+2}$.
One has to argue exactly as done above for the lower order terms $a_1(U;x)(\ii\xi)$ and $a_0(U;x)$ of $a(U;x,\xi)$.
In the same way one is able to prove the estimate
\begin{equation}
\begin{aligned}
\|\bonyw\big((b(U;x,\xi)-b(V;x,\xi))\big)\bar{h}\|_{K,s-2}\leq C \|U-V\|_{K,s_0+2} \| \bar{h}\|_{K,s}.
\end{aligned}
\end{equation}
Thus the \eqref{nave77} is proved renaming $s_0$ as $s_0+2.$

In order to prove \eqref{nave101} we show that the operator ${\rm d}_{U}(R(U)U)[\cdot]$ 
belongs to the class
$\RR^{-\rho}_{K,K'}[r]\otimes\MM_{2}(\CCC)$ for any $\rho\geq 0$
(where  $d_{U}(R(U)U)[\cdot]$ denotes the differential of $R(U)[U]$ w.r.t. the variable $U$).
We recall that the operator $R$ in \eqref{6.666para} 
is of the form 
\[
R(U)[\cdot]:=\left(\begin{matrix}
Q(U)[\cdot] \\ \ol{Q(U)}[\cdot]
\end{matrix}
\right),
\]
where $Q(U)[\cdot]$ is the $1\times 2$ matrix of smoothing operators in \eqref{finalmentesiparalinearizza1}
with $f$ given in \eqref{NLS}.
We claim that ${\rm d}_{U}(Q(U)U)[\cdot]$ 
is $1\times 2$ matrix of smoothing operators in $\RR^{-\rho}_{K,0}[r]$.
By Lemma \ref{paralinearizza} we know that $Q(U)[\cdot]=R_0(U)+\sum_{j=1}^{3}R_{j}(U)$,
where $R_0$ is $1\times 2$ matrix of smoothing operators
coming from the Bony paralinearization formula (see \eqref{finalmentesiparalinearizza2}),
while $R_{j}$, for $j=1,2,3$,  are the $1\times 2$ matrices of smoothing operators in \eqref{ciaone} and \eqref{ciaone1}.

One can prove the claim for the terms $R_{j}$, $j=1,2,3$, by
arguing as done in the proof of \eqref{nave77}. Indeed we know the explicit paradifferential structure of these remainders.
For instance, by \eqref{diego}, \eqref{armando}, \eqref{mara} and \eqref{ciaone}
we have that
\begin{equation}
R_{1}(U)[\cdot]:={\rm op}\Big( k(x,\x)\Big)[\cdot],
\end{equation}
where $k(x,\x)=\sum_{j\in\ZZZ}\hat{k}(j,\x)e^{\ii jx}$ and
\[
k(j,\x)=\tonde{\chi\tonde{n\langle\xi-n/2\rangle^{-1}}-\chi(n\langle\xi\rangle^{-1})}\widehat{D_Uf}(n)
\]
(see formula \eqref{mara}). The remainders $R_2,R_{3}$ have similar expressions.
We reduced to prove the  claim for the term $R_0$. 
Recalling \eqref{quaderno} we set
\[
c(U;x,\x):=c_{U}(U;x,\x)+c_{U_x}(U;x,\x)+c_{U_{xx}}(U;x,\x).
\]
Using this notation, formula \eqref{finalmentesiparalinearizza2} reads
\begin{equation}\label{finalmentesiparalinearizza4}
f(u,u_x,u_{xx})=f_1(U,U_{x},U_{xx})=\bony(c(U;x,\x))U+R_0(U)U.
\end{equation}
Differentiating \eqref{finalmentesiparalinearizza4} we get
\begin{equation}\label{finalmentesiparalinearizza5}
{\rm d}_U(f_1(U,U_{x},U_{xx}))[H]=\bony(c(U;x,\x))[H]+\bony(\del_{U}c(U;x,\x)\cdot H)[U]+{\rm d}_{U}(R_0(U)[U])[H].
\end{equation}
The l.h.s. of \eqref{finalmentesiparalinearizza5} is  nothing but
\[
\del_{U}f_1(U,U_{x},U_{xx})\cdot H+\del_{U_{x}}f_1(U,U_{x},U_{xx})\cdot H_{x}+\del_{U_{xx}}f_1(U,U_{x},U_{xx})\cdot H_{xx}=:
G(U,H).
\]
By applying  the Bony paralinearization formula to $G(U,H)$ (as a function of the six variables 
$U,U_{x},U_{xx}$, $H,H_{x},H_{xx}$) we get
\begin{equation}\label{finalmentesiparalinearizza6}
\begin{aligned}
G(U,H)&= \bony( \del_{U}G(U,H))[U]+ \bony( \del_{U_{x}}G(U,H))[U_{x}]+ \bony( \del_{U_{xx}}G(U,H))[U_{xx}]\\
&+ \bony( \del_{H}G(U,H))[H]+ \bony( \del_{H_x}G(U,H))[H_{x}]+ \bony( \del_{H_{xx}}G(U,H))[H_{xx}]+
R_{4}(U)[H],
\end{aligned}
\end{equation}
where $R_{4}(U)[\cdot]$ satisfies estimates \eqref{porto20} for any $\rho\geq 0$.
By \eqref{quaderno} and \eqref{finalmentesiparalinearizza6} we have that \eqref{finalmentesiparalinearizza5}
reads
\begin{equation}\label{finalmentesiparalinearizza7}
{\rm d}_{U}(R_0(U)U)[H]=R_{4}(U)[H].
\end{equation}
Therefore ${\rm d}_{U}(R_0(U)U)[\cdot]$
is a $1\times2$ matrix of operators in the class
 $R^{-\rho}_{K,0}[r]$ for any $\rho\geq 0$. 
\end{proof}

\zerarcounters
\section{Regularization}\label{descent1}

We consider the system
\begin{equation}\label{sistemainiziale}
\begin{aligned}
\partial_t V=\ii E\Big[\Lambda V&+\bonyw(A(U;x,\x))[V]+R_1^{(0)}(U)[V]+R_2^{(0)}(U)[U]\Big], \\
& U\in B^K_{s_0}(I,r)\cap C^K_{*\R}(I,\hcic^{s}(\TTT,\CCC^2)),
\end{aligned}
\end{equation}
for some $s_0$ large, $s\geq s_0$ and where $\Lambda$ is defined in \eqref{NLS1000}. 
The operators  $R^{(0)}_1(U)$  and $R^{(0)}_2(U)$ are in the class 
$\RR^{-\rho}_{K,0}[r]\otimes\MM_{2}(\CCC)$ for some $\rho\geq0$
and they have the reality preserving form \eqref{vinello}.
The matrix 
 $A(U;x,\x)$
 satisfies the following.
 
 \begin{const}\label{Matriceiniziale}
 The matrix $A(U;x,\xi)$ belongs to 
$ \Gamma^{2}_{K,0}[r]\otimes\MM_{2}(\CCC) $ and has 
 the following properties:
 \begin{itemize}
 \item $A(U;x,\xi)$ is \emph{reality preserving}, i.e.  
 has the form \eqref{prodotto};
 \item the components of $A(U;x,\xi)$ 
 have the form
 \begin{equation}\label{simbolidiA}
 \begin{aligned}
& a(U;x,\x)=a_{2}(U;x)(\ii\x)^{2}+a_{1}(U;x)(\ii\x),\\
& b(U;x,\x)=b_{2}(U;x)(\ii\x)^{2}+b_{1}(U;x)(\ii\x),
 \end{aligned}
 \end{equation}
for some $a_{i}(U;x),b_{i}(U;x)$ belonging to $\calF_{K,0}[r]$ for $i=1,2$.

\end{itemize}
\end{const}

In addition to Constraint \ref{Matriceiniziale} we assume that 
the matrix $A$ satisfies one the following two Hypotheses:

\begin{hyp}[{\bf Self-adjoint}]\label{ipoipo}
The operator $\bonyw(A(U;x,\x))$ is self-adjoint according to Definition \ref{selfi}, 
i.e.  
 the matrix $A(U;x,\x)$ 
satisfies conditions \eqref{quanti801}.
\end{hyp}

\begin{hyp}[{\bf Parity preserving}]\label{ipoipo2}
The operator $\bonyw(A(U;x,\x))$ is parity preserving according to Definition \ref{revmap}, i.e. the matrix  
$A(U;x,\x)$ 
 satisfies the conditions 
 \begin{equation}\label{ipoipo3}
 A(U;x,\x)=A(U;-x,-\x), \qquad a_{2}(U;x)\in \RRR.
 \end{equation}
 The function $P$ in \eqref{convpotential} 
 is such that $\hat{p}(j)=\hat{p}(-j)$ for $j\in \ZZZ$.
\end{hyp}
Finally we need the following  \emph{ellipticity condition}.
\begin{hyp}[{\bf Ellipticity}]\label{ipoipo4}
 There exist $\mathtt{c}_1, \mathtt{c}_2>0$ such that components of the matrix $A(U;x,\xi)$ satisfy the condition
\begin{equation}\label{benigni}
\begin{aligned}
& 1+a_{2}(U;x)\geq \mathtt{c_1},\\
& (1+a_{2}(U;x))^{2}-|b_{2}(U;x)|^{2}\geq \mathtt{c_2}>0,
\end{aligned}
\end{equation}
for any $U\in B^K_{s_0}(I,r)\cap C^K_{*\R}(I,\hcic^{s}(\TTT,\CCC^2))$.
\end{hyp}

The goal of this section is to transform the linear paradifferential system  \eqref{sistemainiziale} into a constant coefficient one up to bounded remainder.

 The following result is the core of our analysis.

\begin{theo}[{\bf Regularization}]\label{descent}
Fix $K\in \NNN$ with $K\geq 4$, $r>0$.  Consider the system \eqref{sistemainiziale}. 
There exists $s_0>0$ such that for any $s\geq s_0$
the following holds.
Fix 
 $U$ in  $B^K_{s_0}(I,r)\cap C^K_{*\R}(I,\hcic^{s}(\TTT,\CCC^2))$ (resp. $U\in B^K_{s_0}(I,r)\cap C^K_{*\R}(I,\hcic_{e}^{s}(\TTT,\CCC^2))$ )
 and assume that the system \eqref{sistemainiziale} has the following structure:
 \begin{itemize}
 \item
the operators 
$R_{1}^{(0)}$, $R_{2}^{(0)}$ belong
to the class $\RR^{-\rho}_{K,0}[r]\otimes\MM_{2}(\CCC)$;
\item
 the matrix 
$A(U;x,\x)$ satisfies Constraint \ref{Matriceiniziale}, 
\item the matrix $A(U;x,\x)$ satisfies Hypothesis \ref{ipoipo} (resp. together with $P$ satisfy Hyp. \ref{ipoipo2}) 
\item the matrix $A(U;x,\x)$ satisfies Hypothesis \ref{ipoipo4}.
\end{itemize}
Then there exists 
an invertible map (resp. an invertible and {parity preserving} map)
$$
\Phi=\Phi(U) :  C^{K-4}_{*\R}(I,\hcic^{s}(\TTT,\CCC^2))\to C^{K-4}_{*\R}(I,\hcic^{s}(\TTT,\CCC^2)),
$$
with
\begin{equation}\label{stimona}
\|(\Phi(U))^{\pm1}V\|_{K-4,s}\leq \|V\|_{K-4,s}(1+C\|U\|_{K,s_0}),
\end{equation}
for a constant $C>0$ depending on $s$, $\|U\|_{K,s_0}$ and $\|P\|_{C^{1}}$
such that the following holds. 
There exist operators ${R}_{1}(U),{R}_2(U)$
in $\RR^{0}_{K,4}[r]\otimes\MM_{2}(\CCC)$, and 
a diagonal matrix 
$L(U)$ in $\Gamma^{2}_{K,4}[r]\otimes\MM_{2}(\CCC)$
of the form \eqref{prodotto}
satisfying condition \eqref{quanti801} and independent of $x\in \TTT$,   such that by setting $W=\Phi(U)V$ the system \eqref{sistemainiziale} reads
\begin{equation}\label{sistemafinale}
\partial_t W=\ii E\Big[\Lambda W+\bonyw(L(U;\x))[W]+R_{1}(U)[W]+R_{2}(U)[U]\Big].
\end{equation}

\end{theo}

\begin{rmk}
Note that, under the Hypothesis \ref{ipoipo2}, if the term $R_1^{(0)}(U)[V]+R_2^{(0)}(U)[U]$ 
in  \eqref{sistemainiziale}
is \emph{parity preserving}, according to Definition \ref{revmap}, then
the flow of the system \eqref{sistemainiziale} preserves the subspace of even functions.
Since the map $\Phi(U)$ in Theorem \ref{descent} is \emph{parity preserving}, then Lemma \ref{revmap100}
implies that also the flow of the system \eqref{sistemafinale} preserves the same subspace.

\end{rmk}

The proof of Theorem \ref{descent} is divided into four steps which are performed in the remaining part of the section.
We first explain our strategy and set some notation. 
We consider the system \eqref{sistemainiziale}
\begin{equation}\label{sistemaini1}
V_{t}=\calL^{(0)}(U)[V]:=\ii E\Big[\Lambda V+\bonyw(A(U;x,\x))[V]+R_1^{(0)}(U)[V]+R_2^{(0)}(U)[U]\Big].
\end{equation}
The idea is to construct several maps
\[
\Phi_{i}[\cdot]:=\Phi_{i}(U)[\cdot] : C^{K-(i-1)}_{*\R}(I,\hcic^{s}(\TTT))\to C^{K-(i-1)}_{*\R}(I,\hcic^{s}(\TTT)),
\]
for $i=1,\ldots,4$ 
which conjugate the system $\calL^{(i)}(U)$ to $\calL^{(i+1)}(U)$, with $\calL^{(0)}(U)$ in \eqref{sistemaini1}
and 
\begin{equation}\label{sistemaiesimo}
\calL^{(i)}(U)[\cdot]:=\ii E\Big[\Lambda+\bonyw(L^{(i)}(U;\x))[\cdot]+\bonyw(A^{(i)}(U;x,\x))[\cdot]+R_{1}^{(i)}[\cdot]+R_{2}^{(i)}(U)[U]\Big],
\end{equation}
where $R^{(i)}_1$ and $R^{(i)}_{2}$ belong to $\RR^{0}_{K,i}[r]\otimes\MM_{2}(\CCC)$,
$L^{(i)}$ belong to   $\Gamma^{2}_{K,i}[r]\otimes\MM_{2}(\CCC)$ and moreover they are
 diagonal,  self-adjoint and independent of $x\in \TTT$ and finally $A^{(i)}$ are in $\Gamma^{2}_{K,i}[r]\otimes\MM_{2}(\CCC)$. 
As we will see, the idea is to look for $\Phi_{i}$ in such a way $A^{(i+1)}$ is  actually a matrix with symbols of order less or equal  than the order of 
$A^{(i)}$.

We now prove a lemma in which 
we study the conjugate of the convolution operator. 

\begin{lemma}\label{Convocoj}
Let $Q_{1},Q_{2}$ operators in the class $\RR^{0}_{K,K'}[r]\otimes \MM_{2}{(\CCC)}$
and $P : \TTT\to \RRR$ a continuous function.
Consider the operator
 $\mathfrak{P}$  defined in \eqref{convototale}.
Then there exists $R$ belonging to
$\RR^{0}_{K,K'}[r]\otimes \MM_{2}{(\CCC)}$ such that
\begin{equation}\label{convoluzionetot}
(\uno+Q_{1}(U))\circ \mathfrak{P}\circ
(\uno+Q_{2}(U))[\cdot]
=\mathfrak{P}[\cdot]+R(U)[\cdot].
\end{equation}
Moreover if $P$ is even in $x$ and the operators $Q_1(U)$ and $Q_2(U)$ are parity-preserving then the operator $R(U)$ is parity preserving according to Definition \ref{revmap}.  
\end{lemma}

\begin{proof}

By linearity it is enough to show that the terms
\[
Q_{1}(U)\circ\mathfrak{P}\circ
(\uno+Q_{2}(U))[h]
,\quad (\uno+Q_{1}(U))\circ\mathfrak{P}
\circ Q_{2}(U)[h]
, \quad Q_{1}(U)\circ\mathfrak{P}\circ
Q_{2}(U)[h]
\]
belong to $\RR^{0}_{K,K'}[r]\otimes \MM_{2}{(\CCC)}$.
Note that, for any $0\leq k\leq K-K'$,
\begin{equation}\label{stimaVVV}
\|\del_{t}^{k}(P*h)\|_{H^{s-2k}}\leq C \|\del_{t}^{k}h\|_{H^{s-2k}},
\end{equation}
for some $C>0$ depending only on $\|P\|_{L^{\infty}}$. The \eqref{stimaVVV}
and the estimate \eqref{porto20} on $Q_1$ and $Q_{2}$ imply the thesis.
If $P$ is even in $x$ then the convolution operator with kernel $P$ is a parity preserving operator according to Definition \ref{revmap}.
Therefore if in addiction $Q_1(U)$ and $Q_2(U)$ are parity preserving so is $R(U)$.
\end{proof}

\subsection{Diagonalization of the second order operator}\label{secondord}

Consider the system \eqref{sistemainiziale} and assume the Hypothesis of Theorem \ref{descent}.
The matrix $A(U;x,\x)$ 
satisfies conditions \eqref{simbolidiA}, 
therefore it can be written as  
\begin{equation}\label{espansionediA}
A(U;x,\x):=A_{2}(U;x)(\ii \x)^{2}+A_{1}(U;x)(\ii \x),
\end{equation}
with $A_{i}(U;x)$ belonging to $\calF_{K,0}[r]\otimes \MM_{2}(\CCC)$
and satisfying eighter Hyp. \ref{ipoipo} or Hyp. \ref{ipoipo2}. 
In this Section, by exploiting the structure of the matrix $A_{2}(U;x)$, we show that it is possible to diagonalize the matrix $E(\uno+A_{2})$ through a change of coordinates
 which is  a multiplication operator.
We have the following lemma.
  \begin{lemma}\label{step1}
Under the Hypotheses of Theorem \ref{descent}
there exists $s_0>0$ such that for any $s\geq s_0$
there exists  an invertible map  (resp. an invertible and parity preserving map) 
$$
\Phi_{1}=\Phi_{1}(U) :  C^{K}_{*\R}(I,\hcic^{s})\to C^{K}_{*\R}(I,\hcic^{s}),
$$
with
\begin{equation}\label{stimona1}
\|(\Phi_1(U))^{\pm1}V\|_{K,s}\leq \|V\|_{K,s}(1+C\|U\|_{K,s_0})
\end{equation}
where $C$ depends only on $s$ and  $\|U\|_{K,s_0}$
such that the following holds. 
There exists a matrix $A^{(1)}(U;x,\x)$ satisfying Constraint \ref{Matriceiniziale} and Hyp. \ref{ipoipo} (resp. Hyp. \ref{ipoipo2}) of the form 
\begin{equation}\label{gorilla}
\begin{aligned}
A^{(1)}(U;x,\x)&:=A_{2}^{(1)}(U;x)(\ii \x)^{2}+A_{1}^{(1)}(U;x)(\ii\x),\\
A_{2}^{(1)}(U;x)&:=
\left(\begin{matrix} {a}_{2}^{(1)}(U;x)& 0\\
0 & {{a_{2}^{(1)}(U;x)}}
\end{matrix}
\right)\in \calF_{K,1}[r]\otimes\MM_{2}(\CCC),\\
A^{(1)}_{ 1}(U;x)&:=
\left(\begin{matrix} {a}_{1}^{(1)}(U;x) & {b}_{1}^{(1)}(U;x)\\
{\ol{b_{1}^{(1)}(U;x,)}} & {\ol{a_{1}^{(1)}(U;x)}}
\end{matrix}
\right) \in \calF_{K,1}[r]\otimes\MM_{2}(\CCC)
\end{aligned}
\end{equation}
 and operators ${R}^{(1)}_{1}(U), \, {R}^{(1)}_2(U)$
in $\RR^{0}_{K,1}[r]\otimes\MM_{2}(\CCC)$ such that by setting $V_1=\Phi(U)V$ the system \eqref{sistemainiziale} reads
\begin{equation}\label{sistemafinale1}
\partial_t V_{1}=\ii E\Big[\Lambda V_1+\bonyw(A^{(1)}(U;x,\x))[V_{1}]+R^{(1)}_{1}(U)[V_{1}]+R^{(1)}_{2}(U)[U]\Big].
\end{equation}
Moreover there exists a constant $\mathtt{k}>0$ such that
\begin{equation}\label{elly2}
1+a_2^{(1)}(U;x)\geq \mathtt{k}.
\end{equation}
\end{lemma}

\begin{proof}
Let us consider a symbol $z(U;x)$
in the class $\calF_{K,0}[r]$ and set 
\begin{equation}\label{generatore}
Z(U;x):=\left(\begin{matrix}0 & z(U;x)\\  \ol{z(U;x)} & 0\end{matrix}\right)\in \calF_{K,0}[r]\otimes\MM_{2}(\CCC).
\end{equation}
Let $\Phi_{1}^{\tau}(U)[\cdot]$ the solution 
at time $\tau\in[0,1]$ of the system
\begin{equation}\label{generatore2}
\left\{\begin{aligned}
&\del_{\tau}\Phi_{1}^{\tau}(U)[\cdot]=\bonyw(Z(U;x))\Phi_{1}^{\tau}(U)[\cdot],\\
&\Phi_1^{0}(U)[\cdot]=\uno[\cdot].
\end{aligned}\right.
\end{equation}
Since $\bonyw(Z(U;x))$ is a bounded operator on $\hcic^{s}$, 
by standard  theory of Banach space ODE we have that
the flow $\Phi_1^{\tau}$ is well defined, moreover   by Proposition \ref{boni2} one gets
\begin{equation}\label{eneest}
\begin{aligned}
\del_{\tau}\|\Phi_{1}^{\tau}(U)V\|^{2}_{\hcic^{s}}&\leq \|\Phi_{1}^{\tau}(U)V\|_{\hcic^{s}}\|\bonyw(Z(U;x))\Phi_{1}^{\tau}(U)V\|_{\hcic^{s}}\\
&\leq \|\Phi_{1}^{\tau}(U)V\|^{2}_{\hcic^{s}}C\|U\|_{\hcic^{s_0}},
\end{aligned}
\end{equation}
hence one obtains
\begin{equation}
\|\Phi_{1}^{\tau}(U)[V]\|_{\hcic^{s}}\leq \|V\|_{\hcic^s}(1+C\|U\|_{\hcic^{s_0}}),
\end{equation}
where $C>0$ depends only on $\|U\|_{\hcic^{s_0}}$. 
The latter estimate implies \eqref{stimona1} for $K=0$.
By differentiating in $t$ the equation \eqref{generatore2}
we note that 
\begin{equation}
\del_{\tau}\del_{t}\Phi_{1}^{\tau}(U)[\cdot]=\bonyw(Z(U;x))\del_{t}\Phi_{1}^{\tau}(U)[\cdot]+\bonyw(\del_{t}Z(U;x))\Phi_{1}^{\tau}(U)[\cdot].
\end{equation}
Now note that, since 
$Z$ belongs to the class $\calF_{K,0}[r]\otimes\MM_{2}(\CCC)$,  one has that
$\del_{t}Z$ is in
$\calF_{K,1}[r]\otimes\MM_{2}(\CCC)$. By performing an energy type estimate as in \eqref{eneest}
one obtains
\[
\|\Phi_{1}^{\tau}(U)[V]\|_{C^{1}\hcic^{s}}\leq \|V\|_{C^{1}\hcic^s}(1+C\|U\|_{C^{1}\hcic^{s_0}}),
\]
which implies \eqref{stimona1} with $K=1$. Iterating $K$ times the reasoning above one gets 
 the bound \eqref{stimona1}. 
 By using Corollary \ref{esponanziale} one gets
that 
\begin{equation}\label{exp}
\Phi_{1}^{\tau}(U)[\cdot]=\exp\{\tau\bonyw(Z(U;x))\}[\cdot]=\bonyw(\exp\{\tau Z(U;x)\})[\cdot]+Q_1^{\tau}(U)[\cdot],
\end{equation}
with $Q_1^{\tau}$
belonging to $\RR^{-\rho}_{K,0}[r]\otimes\MM_{2}({\CCC})$ for any $\rho>0$ and any $\tau\in[0,1]$. 
We now set $\Phi_{1}(U)[\cdot]:=\Phi_{1}^{\tau}(U)[\cdot]_{|_{\tau=1}}$.
In particular
we have
\begin{equation}\label{iperbolici}
\begin{aligned}
\Phi_1(U)[\cdot]&=\bonyw (C(U;x))[\cdot]+Q_{1}^{1}(U)[\cdot]\\
C(U;x):=\exp\{Z(U;x)\}:&=\left(
\begin{matrix}c_1(U;x)& c_{2}(U;x) \\ \ol{c_{2}(U;x)} & c_{1}(U;x)
\end{matrix}
\right), \quad C(U;x)-\uno\in \calF_{K,0}[r]\otimes\MM_{2}(\CCC),
\end{aligned}\end{equation}
where
\begin{equation}\label{iperbolici2}
c_1(U;x):=\cosh(|z(U;x)|) , \qquad c_{2}(U;x):=\frac{z(U;x)}{|z(U;x)|}\sinh(|z(U;x)|).
\end{equation}
Note that the function $c_2(U;x)$ above is not singular indeed
\begin{equation*}
\begin{aligned}
c_{2}(U;x)&=\frac{z(U;x)}{|z(U;x)|}\sinh(|z(U;x)|)=\frac{z(U;x)}{|z(U;x)|}\sum_{k=0}^{\infty}\frac{(|z|(U;x))^{2k+1}}{(2k+1)!}\\
&=z(U;x)\sum_{k=0}^{\infty}\frac{\big(z(U;x)\ol{z(U;x)}\big)^k}{(2k+1)!}.
\end{aligned}
\end{equation*}
We note moreover that for any $x\in \TTT$ one has $\det(C(U;x))=1$, hence its inverse $C^{-1}(U;x)$ is well defined. In particular, by Propositions  \ref{componiamoilmondo} and \ref{componiamoilmare}, 
we note that
\begin{equation}\label{inversaC} 
\bonyw(C^{-1}(U;x))\circ \Phi_1=\uno+\tilde{Q}({U}),\qquad  \tilde{Q}\in \RR^{-\rho}_{K,0}[r]\otimes\MM_{2}(\CCC),
\end{equation}
for any $\rho>0$, since the expansion of $(C^{-1}(U;x)\sharp C(U;x))_{\rho}$ (see formula \eqref{sharp})
is equal to $C^{-1}(U;x)C(U;x)$ for any $\rho$.
This implies that
\begin{equation}\label{iperbolici3}
(\Phi_1(U))^{-1}[\cdot]=\bonyw( C^{-1}(U;x))[\cdot]+Q_{2}(U)[\cdot],
\end{equation}
for some $Q_{2}(U)$ in the class $\RR^{-\rho}_{K,0}[r]\otimes\MM_{2}(\CCC)$ for any $\rho>0$. By setting 
$V_{1}:=\Phi_1(U)[V]$ the system  \eqref{sistemainiziale}  in the new coordinates reads
\begin{equation}\label{expexpexp}
\begin{aligned}
(V_{1})_{t}&=\Phi_1(U)\Big(
\ii E(\Lambda+\bonyw(A(U;x,\x)) )\Phi_1^{-1}(U)
\Big)V_{1}+(\del_{t}\Phi_{1}(U))\Phi_1^{-1}(U)V_{1}+\\
&+\Phi_{1}(U)(\ii E)R_1^{(0)}(U)\Phi_1^{-1}(U)[V_1]+
\Phi_1(U)(\ii E)R_2^{(0)}(U)[U]\\
&=
\ii \Phi_1(U)\Big[E\mathfrak{P}[\Phi_1^{-1}(U)[V_1]] \Big]+
\ii\Phi_1(U) E\bonyw\big((\uno+A_{2}(U;x))(\ii\x)^{2} \big)\Phi_1^{-1}(U)[V_1]+\\
&+\ii\Phi_1(U) E\bonyw\big(A_{1}(U;x)(\ii\x) \big)\Phi_1^{-1}(U)[V_1]
+(\del_{t}\Phi_{1})\Phi_1^{-1}(U)V_{1}+\\
&+\Phi_{1}(U)(\ii E)R_1^{(0)}(U)\Phi_1^{-1}(U)[V_1]+
\Phi_1(U)(\ii E)R_2^{(0)}(U)[U],
\end{aligned}
\end{equation}
where $\mathfrak{P}$ is defined in 
\eqref{convototale}.
We have that
\[
\Phi_1(U)\circ E=E\circ\bonyw\left(\begin{matrix}
c_1(U;x)& -c_2(U;x) \\
- \ol{c_2(U;x)} & c_1(U;x)
\end{matrix}
\right),
\]
up to remainders in  $\RR^{-\rho}_{K,0}[r]\otimes\MM_{2}({\CCC})$, where $c_{i}(U;x)$, $i=1,2$, are defined in \eqref{iperbolici2}.
Since the matrix $C(U;x)-\uno \in \calF_{K,0}[r]\otimes\MM_{2}(\CCC)$ (see \eqref{iperbolici})
then by Lemma \ref{Convocoj} one has that
\[
\Phi_1(U)\circ E \mathfrak{P}\circ\Phi_1^{-1}(U)[V_1]] =E \mathfrak{P}[V_1]+Q_{3}(U)[V_1],
\]
 where $Q_3(U)$ belongs to $\RR^{0}_{K,0}[r]\otimes\MM_{2}(\CCC)$.
The term $(\del_{t}\Phi_1)$ is 
$\bonyw(\del_{t}C(U;x))$
plus a remainder in the class $\RR^{0}_{K,1}[r]\otimes\MM_{2}(\CCC)$. 
Note that, since 
$(C(U;x)-\uno)$ belongs to the class $\Gamma^{0}_{K,0}[r]\otimes\MM_{2}(\CCC)$,  one has that
$\del_{t}C(U;x)$ is in
$\Gamma^{0}_{K,1}[r]\otimes\MM_{2}(\CCC)$. 
Therefore, by the composition Propositions \ref{componiamoilmondo} and \ref{componiamoilmare}, Remark \ref{inclusione-nei-resti}, and using the discussion above  
we have that, there exist operators $R^{(1)}_{1},R^{(1)}_{2}$ belonging to $\RR^{0}_{K,1}[r]\otimes\MM_{2}(\CCC)$
such that
\begin{equation}\label{nuovosistema}
\begin{aligned}
(V_{1})_{t}&=\ii E \mathfrak{P}V_1+\ii \bonyw\big(C(U;x)E(\uno+A_{2}(U;x))C^{-1}(U;x)(\ii\x)^{2})\big)V_{1}+
\ii E\bonyw(A_{1}^{(1)}(U;x) (\ii\x))V_{1}+\\
&
+\ii E\Big(R_{1}^{(1)}(U)[V_{1}]+R_{2}^{(1)}(U)[U]\Big)
\end{aligned}
\end{equation}
where 
\begin{equation}\label{primoordine}
\begin{aligned}
A_{1}^{(1)}(U;x)&:=EC(U;x)E(\uno+A_{2}(U;x))\del_{x}C^{-1}(U;x)-(\del_{x}(C)(U;x))E(\uno+A_{2}(U;x))C^{-1}(U;x)\\
&+EC(U;x)A_{1}(U;x)C^{-1}(U;x),
\end{aligned}
\end{equation}
with $A_{1}(U;x),A_{2}(U;x)$ defined in \eqref{espansionediA}.
Our aim is to find a symbol $z(U;x)$ such that 
the matrix of symbols $C(U;x) E(\uno+A_{2}(U;x))C^{-1}(U;x)$ is diagonal. We reason as follows.
One can note that the eigenvalues of $E(\uno+A_{2}(U;x))$ are
\[
\la^{\pm}:=\pm \sqrt{(1+a_{2}(U;x))^{2}-|b_{2}(U;x)|^{2}}.
\]
We define the symbols
\begin{equation}\label{nuovadiag}
\begin{aligned}
\la_{2}^{(1)}(U;x)&:=\la^{+},\\
a_{2}^{(1)}(U;x)&:=\la_{2}^{(1)}(U;x)-1\in \calF_{K,0}[r].
\end{aligned}
\end{equation}
The symbol $\la_{2}^{(1)}(U;x)$ is well defined and  satisfies \eqref{elly2} thanks to Hypothesis \ref{ipoipo4}.
The matrix of the normalized eigenvectors associated to the eigenvalues of $E(\uno+A_{2}(U;x))$
is 
\begin{equation}\label{transC}
\begin{aligned}
S(U;x)&:=\left(\begin{matrix} {s}_1(U;x) & {s}_2(U;x)\\
{\ol{s_2(U;x)}} & {{s_1(U;x)}}
\end{matrix}
\right),\\
s_{1}(U;x)&:=\frac{1+a_{2}(U;x)+\la_{2}^{(1)}(U;x)}{\sqrt{2\la_{2}^{(1)}(U;x)\big(1+a_{2}(U;x)+
\la_{2}^{(1)}(U;x)\big) }},\\
s_{2}(U;x)&:=\frac{-b_{2}(U;x)}{\sqrt{2\la_{2}^{(1)}(U;x)\big(1+a_{2}(U;x)+\la_{2}^{(1)}(U;x)\big) }}.
\end{aligned}
\end{equation}
Note that $1+a_{2}(U;x)+\la_{2}^{(1)}(U;x) \geq \mathtt{c_1}+\sqrt{\mathtt{c_2}}>0$ by \eqref{benigni}. Therefore one can check that $S(U;x)-\uno \in \calF_{K,0}[r]\otimes\MM_{2}(\CCC)$.
Therefore the matrix $S$ is invertible and one has
\begin{equation}\label{diag1}
S^{-1}(U;x)\big[E (\uno+A_{2}(U;x))\big] S(U;x)=E
\left(\begin{matrix}1+a_{2}^{(1)}(U;x) & 0\\
0 & 1+a_{2}^{(1)}(U;x)
\end{matrix}
\right).
\end{equation}

We choose $z(U;x)$ in such a way that 
$C^{-1}(U;x):=S(U;x)$. Therefore we have to solve the following equations 
\begin{equation}\label{losappaimorisolvere???}
\cosh(|z(U;x)|)=s_1(U;x), \quad \frac{z(U;x)}{|z(U;x)|}\sinh(|z(U;x)|)=-s_{2}(U;x).
\end{equation}
Concerning the first one we note that $s_1$ satisfies 
 \[
 (s_1(U;x))^{2}-1=\frac{|b_2(U;x)|^{2}}{2\la_{2}^{(1)}(U;x)(1+a_{2}(U;x)+\la_{2}^{(1)}(U;x))}\geq0,
 \]
 indeed we remind that $1+a_{2}(U;x)+\la_{2}^{(1)}(U;x) \geq \mathtt{c_1}+\sqrt{\mathtt{c_2}}>0$ by \eqref{benigni}, therefore 
 \[
|z(U;x)|:= {\rm arccosh}(s_1(U;x))=\ln\Big(s_1(U;x)+\sqrt{(s_1(U;x))^{2}-1}\Big),
\]
is well-defined. For the second equation one observes that
the function 
\[
\frac{\sinh(|z(U;x)|)}{|z(U;x)|}=1+\sum_{k\geq0}\frac{(z(U;x)\bar{z}(U;x))^{k}}{(2k+1)!}\geq 1,
\]
hence we set
\begin{equation}\label{definizioneC}
z(U;x):=s_{2}(U;x)\frac{|z(U;x)|}{\sinh(|z(U;x)|)}.
\end{equation}


We set
\begin{equation}\label{nuovosis2}
\begin{aligned}
A^{(1)}(U;x,\x)&:=A_{2}^{(1)}(U;x)(\ii\x)^{2}+A_{1}^{(1)}(U;x)(\ii\x),\\
A_{2}^{(1)}(U;x)&:=
\left(\begin{matrix}a_{2}^{(1)}(U;x) & 0\\
0 & a_{2}^{(1)}(U;x)
\end{matrix}
\right)
\end{aligned}
\end{equation}
where $a_{2}^{(1)}(U;x)$ is defined in \eqref{nuovadiag} and $A_{1}^{(1)}(U;x)$ is defined in \eqref{primoordine}.
Equation \eqref{diag1}, together with \eqref{nuovosistema} and \eqref{nuovosis2} implies that 
\eqref{sistemafinale1} holds.
By construction one has that the matrix $A^{(1)}(U;x,\x)$
satisfies Constraint \ref{Matriceiniziale}.
It remains to show that 
$A^{(1)}(U;x,\x)$ satisfies  either Hyp. \ref{ipoipo} or Hyp \ref{ipoipo2}. 

 If $A(U;x,\x)$ satisfies Hyp. \ref{ipoipo2}
then 
we have that ${a_{2}^{(1)}}(U;x)$ in \eqref{nuovadiag} is real. Moreover by construction $S(U;x)$ in \eqref{transC} is even in $x$, therefore by Remark \ref{compsimb} we have that
the map $\Phi_{1}(U)$ in \eqref{exp}  is parity preserving according to Definition
\ref{revmap}.  This implies that the matrix $A^{(1)}(U;x,\x)$ satisfies Hyp. \ref{ipoipo2}.
 Let us consider the case when $A(U;x,\x)$ satisfies Hyp. \ref{ipoipo}.
One can check, by an explicit computation, that the map $\Phi_1(U)$ in \eqref{exp},
is such that
\begin{equation}\label{quasisimplettica}
\Phi_{1}^{*}(U)(-\ii E )\Phi_{1}(U)=(-\ii E)+\tilde{R}(U),
\end{equation}
for some smoothing operators $\tilde{R}(U)$ belonging to $\RR^{-\rho}_{K,0}[r]\otimes\MM_{2}(\CCC)$.
In other words, up to a $\rho-$smoothing operator, the map $\Phi_1(U)$ satisfies conditions
\eqref{symsym10}. By 
 following essentially word by word the proof of Lemma \ref{lemmalemma}
one obtains that, up to a smoothing operator in the class   
$\RR^{-\rho}_{K,1}[r]\otimes\MM_{2}(\CCC)$, the operator $\bonyw(A^{(1)}(U;x,\xi))$
in \eqref{sistemafinale1} is 
self-adjoint. This implies that the matrix  $A^{(1)}(U;x,\x)$
satisfies Hyp. \ref{ipoipo}. This concludes the proof.

\end{proof}

\subsection{Diagonalization of the first order operator}\label{diago1}

In the previous Section  we conjugated system \eqref{sistemainiziale} to \eqref{sistemafinale1},
where the matrix $A^{(1)}(U;x,\x)$ has the form
\begin{equation}\label{espansionediA1}
A^{(1)}(U;x,\x)=A_{2}^{(1)}(U;x)(\ii\x)^{2}+A_{1}^{(1)}(U;x)(\ii\x),
\end{equation}
with $A_{i}^{(1)}(U;x)$ belonging to $\calF_{K,1}[r]\otimes\MM_{2}(\CCC)$
and where $A_{2}^{(1)}(U;x)$ is diagonal. 
In this Section we show that, since the matrices $A_{i}^{(1)}(U;x)$ satisfy Hyp. \ref{ipoipo}
 (respectively Hyp. \ref{ipoipo2}),
 it is possible to diagonalize also the term $A_{1}^{(1)}(U;x)$
 through a change of coordinates which is 
 the identity plus a smoothing term.
This is the result of the following lemma.

\begin{lemma}\label{step2}
If the matrix $A^{(1)}(U;x,\x)$ in \eqref{sistemafinale1} satisfies 
Hypothesis \ref{ipoipo} (resp. together with $P$ satisfy Hyp. \ref{ipoipo2})
then
there exists $s_0>0$ (possibly larger than the one in Lemma \ref{step1}) such that for any $s\geq s_0$
there exists  an invertible map 
(resp. an invertible and  parity preserving map)
$$
\Phi_{2}=\Phi_{2}(U) :  C^{K-1}_{*\R}(I,\hcic^{s})\to C^{K-1}_{*\R}(I,\hcic^{s}),
$$
with
\begin{equation}\label{stimona2}
\|(\Phi_{2}(U))^{\pm1}V\|_{K-1,s}\leq \|V\|_{K-1,s}(1+C\|U\|_{K,s_0})
\end{equation}
where $C>0$ depends only on $s$ and $\|U\|_{K,s_0}$
such that the following holds. 
There exists a matrix $A^{(2)}(U;x,\x)$ 
satisfying Constraint \ref{Matriceiniziale} and Hyp. \ref{ipoipo} (resp. Hyp. \ref{ipoipo2})
of the form 
\begin{equation}\label{gorilla2}
\begin{aligned}
A^{(2)}(U;x,\x)&:=A_{2}^{(2)}(U;x)(\ii \x)^{2}+A_{1}^{(2)}(U;x)(\ii \x),\\
A_{2}^{(2)}(U;x)&:=A_{2}^{(1)}(U;x);\\
A_{1}^{(2)}(U;x)&:=
\left(\begin{matrix} {a}_{1}^{(2)}(U;x) & 0\\
0 & \ol{{a_{1}^{(2)}(U;x)}}
\end{matrix}
\right)\in \calF_{K,2}[r]\otimes\MM_{2}(\CCC),\\
\end{aligned}
\end{equation}
 and operators ${R}^{(2)}_{1}(U),{R}^{(2)}_2(U)$
in $\RR^{0}_{K,2}[r]\otimes\MM_{2}(\CCC)$, such that by setting $V_{2}=\Phi_2(U)V_{1}$ the system \eqref{sistemafinale1} reads
\begin{equation}\label{sistemafinale2}
\partial_t V_{2}=\ii E\Big[\Lambda V_{2}+\bonyw(A^{(2)}(U;x,\x))[V_{2}]+R^{(2)}_{1}(U)[V_{2}]+R^{(2)}_{2}(U)[U]\Big].
\end{equation}
\end{lemma}

\begin{proof}

We recall that by Lemma \ref{step1}
we have that
\begin{equation*}
\begin{aligned}
A^{(1)}(U;x,\x):=\left(
\begin{matrix}
a^{(1)}(U;x,\x) & b^{(1)}(U;x,\x)\\
\ol{b^{(1)}(U;x,-\x)} & \ol{a^{(1)}(U;x,-\x)}
\end{matrix}
\right).
\end{aligned}
\end{equation*}
Moreover 
by \eqref{gorilla} we can write
\[
\begin{aligned}
a^{(1)}(U;x,\x)&=a_{2}^{(1)}(U;x)(\ii\x)^{2}+a_{1}^{(1)}(U;x)(\ii\x),\\
b^{(1)}(U;x,\x)&=b_{1}^{(1)}(U;x)(\ii\x),
\end{aligned}
\]
with $a_{2}^{(1)}(U;x),a_{1}^{(1)}(U;x),b_{1}^{(1)}(U;x)\in \calF_{K,1}[r]$. 
In the case that $A^{(1)}(U;x,\x)$ satisfies Hyp. \ref{ipoipo}, 
we can note that $b_{1}(U;x)\equiv0$. 
Hence  it is enough to 
choose $\Phi_{2}(U)\equiv\uno$ to obtain the thesis.
On the other hand, assume that $A^{(1)}(U;x,\x)$ satisfies Hyp. \ref{ipoipo2}
we reason as follows.

Let us consider a symbol $d(U;x,\x)$ in the class $\Gamma^{-1}_{K,1}[r]$
and define
\begin{equation}\label{mappa3}
\begin{aligned}
&D(U;x,\x):=\left(
\begin{matrix}
0&d(U;x,\x)\\
\ol{d(U;x,-\x)} & 0
\end{matrix}
\right)\in \Gamma^{-1}_{K,1}[r]\otimes\MM_{2}(\CCC).
\end{aligned}
\end{equation}
Let $\Phi_{2}^{\tau}(U)[\cdot]$ be the flow of the system
\begin{equation}
\left\{\begin{aligned}
&\del_{\tau}\Phi_{2}^{\tau}(U)=\bonyw(D(U;x,\x))\Phi_{2}^{\tau}(U)\\
&\Phi_{2}^{0}(U)=\uno.
\end{aligned}\right.
\end{equation}
Reasoning as done for the system \eqref{generatore2} one has that there exists a unique family of
invertible bounded operators on $\hcic^{s}$
satisfying 
with
\begin{equation}\label{stimona200}
\|(\Phi_{2}^{\tau}(U))^{\pm1}V\|_{K-1,s}\leq \|V\|_{K-1,s}(1+C\|U\|_{K,s_0})
\end{equation}
for $C>0$ depending on $s$ and $\|U\|_{K,s_0}$ for $\tau\in [0,1]$.

The operator  $W^{\tau}(U)[\cdot]:=\Phi_{2}^{\tau}(U)[\cdot]-(\uno+\tau\bonyw(D(U;x,\x)))$
solves the following system:
\begin{equation}\label{sistW}
\left\{\begin{aligned}
&\del_{\tau}W^{\tau}(U)=\bonyw(D(U;x,\x))W^{\tau}(U)+\tau \bonyw(D(U;x,\x))\circ\bonyw(D(U;x,\x))\\
&W^{0}(U)=0.
\end{aligned}\right.
\end{equation}
Therefore, by Duhamel formula, one can check that $W^{\tau}(U)$ is a smoothing operator in the class
$\RR^{-2}_{K,1}[r]\otimes\MM_{2}(\CCC)$ for any $\tau\in[0,1] $.
We set  $\Phi_{2}(U)[\cdot]:=\Phi_{2}^{\tau}(U)[\cdot]_{|_{\tau=1}}$, by the discussion above
we have that there exists $Q(U)$ in $\RR^{-2}_{K,1}[r]\otimes\MM_{2}(\CCC)$
such that 
\[
\Phi_{2}(U)[\cdot]=\uno+\bonyw(D(U;x,\x))+Q(U).
\]
Since $\Phi_{2}^{-1}(U)$ exists, by symbolic calculus, it is easy to check
that there exists $\tilde{Q}(U)$ in  $\RR^{-2}_{K,1}[r]\otimes\MM_{2}(\CCC)$
such that
\[
\Phi^{-1}_{2}(U)[\cdot]=\uno-\bonyw(D(U;x,\x))+\tilde{Q}(U).
\]
We set
$V_{2}:=\Phi_{2}(U)[V_{1}]$, therefore the system \eqref{sistemafinale1} in the new coordinates reads 
\begin{equation}\label{nuovosist}
\begin{aligned}
(V_2)_{t}&=\Phi_{2}(U)\ii E\Big(\Lambda+\bonyw(A^{(1)}(U;x,\x))+R_{1}^{(1)}(U)\Big)(\Phi_{2}(U))^{-1}[V_{2}]+\\
&+\Phi_{2}(U)\ii ER_{2}^{(1)}(U)[U]+\bonyw(\del_{t}\Phi_2(U))(\Phi_{2}(U))^{-1}[V_{2}].
\end{aligned}
\end{equation}
The summand $\Phi_{2}(U)\ii ER_{2}^{(1)}(U)[\cdot] $
belongs to the class $\RR^{0}_{K,1}[r]\otimes\MM_{2}(\CCC)$ by composition Propositions.
Since $\del_{t}D(U;x,\x)$ belongs to $\Gamma^{-1}_{K,2}[r]\otimes\MM_{2}(\CCC)$ and $\partial_t Q$ is in $\RR^{-2}_{K,2}[r]\otimes\MM_{2}(\CCC)$
then the last summand in \eqref{nuovosist}  belongs to $\RR^{0}_{K,2}[r]\otimes\MM_{2}(\CCC)$.
We now study the first summand.
First we note that 
 $\Phi_{2}(U)\ii ER_{1}^{(1)}(U)\Phi_{2}^{-1}(U)$ is a bounded remainder in $\RR^{0}_{K,1}[r]\otimes\MM_{2}(\CCC)$. It remains to study the term
 \[
 \ii \Phi_{2}(U)\Big[E\mathfrak{P}\big(\Phi_{2}^{-1}(U)[V_2]\big)\Big)
 \Big]+
\ii \Phi_{2}(U)\Big[
 \bonyw\big(E(\uno+A_{2}^{(1)}(U;x))(\ii\x)^{2}+EA_{1}^{(1)}(U;x)(\ii\x)
 \big)
 \Big]\Phi_{2}^{-1}(U)[V_2],
 \]
 where $\mathfrak{P}$ is defined in \eqref{convototale}.
 The first term is equal to $\ii E(\mathfrak{P}V_2)$
up to a bounded term in $\RR^{0}_{K,1}[r]\otimes\MM_{2}(\CCC)$
by Lemma \ref{Convocoj}. The second is equal to
\begin{equation}
\begin{aligned}
&\ii \bonyw\big(E(\uno+A_{2}^{(1)}(U;x))(\ii\x)^{2}+EA_{1}^{(1)}(U;x)(\ii\x)
 \big)+\\
 &+\Big[
 \bonyw(D(U;x,\x)),
 \ii E\bonyw\big((\uno+A_{2}^{(1)}(U;x))(\ii\x)^{2}\big)
 \Big]
 \end{aligned}
\end{equation}
modulo bounded terms in  $\RR^{0}_{K,1}[r]\otimes\MM_{2}(\CCC)$.
By using formula \eqref{sbam8} one get that the commutator above is equal to
$\bonyw(M(U;x,\x))$ with 
\begin{equation}
\begin{aligned}
M(U;x,\x)&:=\left(\begin{matrix} 0& m(U;x,\x)\\
{\ol{m(U;x,-\x)}} &0
\end{matrix}
\right),\\
m(U;x,\x)&:=-2d(U;x,\x)(1+a_{2}^{(1)}(U;x))(\ii\x)^{2},
\end{aligned}
\end{equation}
up to terms in  $\RR^{0}_{K,1}[r]\otimes\MM_{2}(\CCC)$. Therefore the system obtained after the change of coordinates reads
\begin{equation}
(V_{2})_{t}=\ii E\Big[\Lambda V_{2}+\bonyw(A^{(2)}(U;x,\x))[V_{2}]+Q_1(U)[V_{2}]+Q_{2}(U)[U]\Big],
\end{equation}
where $Q_1(U)$ and $Q_2(U)$ are bounded terms in $\RR^{0}_{K,2}[r]\otimes\MM_{2}(\CCC)$ and the new matrix 
$A^{(2)}(U;x,\x))$ is
\begin{equation}\label{riassunto}
\left(\begin{matrix} {a}_{2}^{(1)}(U;x) & 0\\
0 & \ol{{a_{2}^{(1)}(U;x)}}
\end{matrix}
\right)(\ii\xi)^2+ \left(\begin{matrix} {a}_{1}^{(1)}(U;x) & b_1^{(1)}(U;x)\\
\ol{b_1^{(1)}(U;x)} & \ol{{a_{1}^{(1)}(U;x)}}
\end{matrix}
\right)(\ii\xi)+M(U;x,\x).
\end{equation}
Hence the elements on the diagonal are not affected by  the change of coordinates, now
our aim is to choose  $d(U;x,\x)$ in such a way that the symbol
\begin{equation}\label{leespansioni-belle}
b_{1}(U;x)(\ii\x)+m(U;x,\xi)=b_{1}(U;x)(\ii\x)-2d(U;x,\x)(1+a_{2}^{(1)}(U;x))(\ii\x)^{2},
\end{equation}
belongs to $\Gamma^{0}_{K,2}[r]$.
We split the symbol in \eqref{leespansioni-belle} in low-high frequencies: let $\varphi(\xi)$ a function in $C^{\infty}_0(\RRR;\RRR)$ such that $\rm{supp}(\varphi)\subset [-1,1]$ and $\varphi\equiv 1$ on $[-1/2,1/2]$. 
Trivially one has that $\varphi(\xi)(b_{1}(U;x)(\ii\x)+m(U;x,\xi))$ is a symbol in $\Gamma^0_{K,1}[r]$, so it is enough to solve the equation
\begin{equation}\label{monk}
\big(1-\varphi(\xi)\big)\left[b_{1}(U;x)(\ii\x)-2d(U;x,\x)\left(1+a_{2}^{(1)}(U;x))(\ii\x)^{2}\right)\right]=0.
\end{equation}
So we should choose the symbol $d$ as
\begin{equation}\label{simbolod}
\begin{aligned}
&d(U;x,\x)=\left(\frac{b_{1}^{(1)}(U;x)}{2(1+a_{2}^{(1)}(U;x))} \right)\cdot\gamma{(\xi)}\\
&\gamma(\xi)=\left\{
\begin{aligned}
&\frac{1}{\ii\x} \quad \rm{if\,\,} |\xi|\geq \frac12 \\
& \rm{odd\,  continuation\, of\, class\, } C^{\infty} \quad \rm{if\,\,} |\xi|\in [0,\frac12).\\
\end{aligned}\right.
\end{aligned}
 \end{equation}

 Clearly the symbol $d(U;x,\x)$ in \eqref{simbolod}  belongs to
$\Gamma^{-1}_{K,1}[r]$, hence the map $\Phi_{2}(U)$ in \eqref{mappa3} is well defined
and estimate \eqref{stimona2} holds.
It is evident that, after the choice of the symbol in \eqref{simbolod}, the matrix $A^{(2)}(U;x,\x)$ is
\begin{equation}
 \left(\begin{matrix} {a}_{2}^{(1)}(U;x) & 0\\
0 & \ol{{a_{2}^{(1)}(U;x)}}
\end{matrix}
\right)(\ii\xi)^2+\left(\begin{matrix} {a}_{1}^{(1)}(U;x) & 0\\
0 & \ol{{a_{1}^{(1)}(U;x)}}.
\end{matrix}
\right)(\ii\xi)\end{equation}
The symbol $d(U;x,\xi)$ is equal to $d(U;-x,-\xi)$ because $b_1^{(1)}(U;x)$ is odd in $x$ and 
$a_2^{(1)}(U;x)$ is even in $x$, therefore, by Remark \ref{compsimb} the map $\Phi_2(U)$ is \emph{parity preserving}.

\end{proof}

\subsection{Reduction to constant coefficients 1: paracomposition}\label{ridu2}
Consider the diagonal matrix of functions $A_{2}^{(2)}(U;x)\in \calF_{K,2}[r]\otimes\MM_{2}(\CCC)$ defined in \eqref{gorilla2}. In this section we shall reduce the operator $\bonyw(A_2^{(2)}(U;x)(\ii\xi)^{2})$ to a constant coefficient one, up to bounded terms (see \eqref{gorilla3}).
For these purposes we shall use a paracomposition operator (in the sense of Alinhac \cite{AliPARA}) associated to the diffeomorphism $x\mapsto x+\beta{(x)}$ of $\TTT$. We follow  Section 2.5 of \cite{maxdelort} and in particular we shall use their alternative definition of paracomposition operator. 

Consider a  real symbol $\beta(U;x)$ in the class $\mathcal{F}_{K,K'}[r]$
and the map 
\begin{equation}\label{diffeo1}
\Phi_U: x\mapsto x+\beta(U;x).
\end{equation}
We state the following.

\begin{lemma}\label{LEMMA252}
Let $0\leq K'\leq K$ be in $\NNN$, $r>0$ and $\be(U;x)\in \mathcal{F}_{K,K'}[r]$
for $U$ in the space $C^{K}_{*\RRR}(I,\hcic^{s_0})$.
If $s_0$ is sufficiently large and $\be$ is $2\pi$-periodic in $x$ and satisfies
\begin{equation}\label{conddiffeo}
1+ \be_{x}(U;x)\geq \Theta>0, \quad x\in \RRR,
\end{equation}
for some constant $\Theta$ depending on $\sup_{t\in I}\norm{U(t)}{\hcic^{s_0}}$, then the map 
$\Phi_{U}$ in \eqref{diffeo1} is a diffeomorphism of $\TTT$  to itself, and its inverse
may be written as
\begin{equation}\label{diffeo2}
(\Phi_U)^{-1}: y\mapsto y+\gamma(U;y)
\end{equation}
for $\gamma$ in $\mathcal{F}_{K,K'}[r]$.
\end{lemma}

\begin{proof}
Under condition \eqref{conddiffeo}  there exists $\g(U; y)$ such that
\begin{equation}\label{conddiffeo2}
x+\be(U;x)+\g(U;x+\be(U;x))=x, \quad \;\; x\in \RRR.
\end{equation}
One can prove the bound \eqref{simbo} on the function $\g(U;y)$
by differentiating in $x$ equation \eqref{conddiffeo2} and using that $\be(U;x)$
is a symbol in $\calF_{K,K'}[r]$.
\end{proof}

\begin{rmk}
The Lemma above is very similar to  Lemma $2.5.2$ of \cite{maxdelort}.
The authors use a smallness assumption on $r$ to prove the result.
Here this assumption is replaced by
\eqref{conddiffeo} in order to treat big sized initial conditions.
\end{rmk}

\begin{rmk}\label{pathdiffeo}
By Lemma \ref{LEMMA252} one has that $x\mapsto x+\tau\be(U;x)$ is a diffeomorphism of $\TTT$ for any $\tau\in [0,1]$. Indeed 
\begin{equation*}
1+\tau\beta_x({U;x})=1-\tau+\tau(1+\beta_x({U;x}))\geq (1-\tau)+\tau\Theta\geq\min\{1,\Theta\}>0,
\end{equation*}
for any $\tau\in [0,1]$. Hence the \eqref{conddiffeo} holds true with $\mathtt{c}=\min\{1,\Theta\}$ and Lemma \ref{LEMMA252} applies.

\end{rmk}
%
With the aim of simplifying  the notation we set $\beta{(x)}:=\beta(U;x)$, $\gamma(y):=\gamma(U;x)$ and we define the following quantities 
\begin{equation}\label{bbb}
\begin{aligned}
B(\tau;x,\x)=B(\tau,U;x,\x)&:=-\ii b(\tau;x)(\ii \x), \\
b(\tau;x)&:=\frac{\be(x)}{(1+\tau \be_{x}(x))}.
\end{aligned}
\end{equation}
Then one defines the paracomposition operator associated to the diffeomorphism \eqref{diffeo1} as $\Omega_{B(U)}(1)$, where $\Omega_{B(U)}(\tau)$ is the flow of the linear paradifferential equation
\begin{equation}\label{flow}
\left\{
\begin{aligned}
&\frac{d}{d\tau}\Omega_{B(U)}(\tau)=\ii \bonyw{(B(\tau;U,\x))}\Omega_{B(U)}(\tau)\\
&\Omega_{B(U)}(0)=\rm{id}.
\end{aligned}\right.
\end{equation}
We state here a  Lemma 
which asserts that the problem \eqref{flow} is well posed and whose solution is a one parameter family of bounded operators  on $H^s$, which is one of the main properties of a paracomposition operator.  For the proof of the result we refer to Lemma $2.5.3$ in \cite{maxdelort}.

\begin{lemma}\label{torodiff}
Let $0\leq K'\leq K$ be in $\NNN$, $r>0$ and $\be(U;x)\in \mathcal{F}_{K,K'}[r]$
for $U$ in the space $C^{K}_{*\RRR}(I,\hcic^{s})$.
The system \eqref{flow} has a unique solution defined for $\tau\in[-1,1]$. Moreover  for any $s$ in $\R$ there exists a constant $C_s>0$ such that for any $U$ in $B^K_{s_0}(I,r)$ and any $W$ in $H^s$
\begin{equation}\label{flow2}
C_{s}^{-1}\|W\|_{H^{s}}\leq \|\Omega_{B(U)}(\tau)W\|_{H^{s}} \leq C_{s} \|W\|_{H^{s}}, \quad \forall\; \tau\in[-1,1], \;\;\; W\in H^{s},
\end{equation}
and 
\begin{equation}\label{flow3}
\|\Omega_{B(U)}(\tau)W\|_{K-K', s}\leq (1+C\|U\|_{K,s_0})\|W\|_{K-K',s},
\end{equation}
where $C>0$ is a constant depending only on $s$ and $\|U\|_{K,s_0}$. 
\end{lemma}

\begin{rmk}
As pointed out in Remark \ref{differenzaclassidisimbo}, our classes of symbols are
slightly different from the ones in \cite{maxdelort}.
For this reason the authors in  \cite{maxdelort} are  more precise about the constant $C$
in \eqref{flow3}. 
However the proof can be adapted straightforward.
\end{rmk}
\begin{rmk}\label{pathdi}
In the following we shall study how symbols $a(U;x,\x)$ changes under
conjugation through the flow $\Omega_{B(U)}(\tau)$ introduced in Lemma \ref{torodiff}.
In order to do this we shall apply Theorem $2.5.8$ in \cite{maxdelort}. Such result 
requires that $x\mapsto x+\tau \be(U;x)$ is a path of diffemorphism for $\tau\in [0,1]$.
In \cite{maxdelort} this fact is achieved by using the smallness of $r$, here 
it is implied by Remark \ref{pathdiffeo}.
\end{rmk}

%
%

 We now study how the convolution operator $P*$
changes under the flow $\Omega_{B(U)}(\tau)$ introduced in Lemma \ref{torodiff}.

\begin{lemma}\label{Convocoj2}
Let  $P : \TTT\to \RRR$ be a $C^{1}$ function, 
let us define $P_{*}[h]=P*h$ for $h\in H^{s}$, where $*$ denote the convolution between functions, 
 and set $\Phi(U)[\cdot]:=\Omega_{B(U)}(\tau)_{|_{\tau=1}}$.
There exists $R$ belonging to
$\RR^{0}_{K,K'}[r]$ such that
\begin{equation}\label{convoluzionetot1000}
\Phi(U)\circ P_{*}\circ
\Phi^{-1}(U)[\cdot]
=P_{*}[\cdot]+R(U)[\cdot].
\end{equation}
Moreover if $P(x)$ is even in $x$ and  $\Phi(U)$ is parity preserving according to Definition \ref{revmap} then the remainder $R(U)$ in \eqref{convoluzionetot1000} is parity preserving.
\end{lemma}

\begin{proof}
Using equation \eqref{flow} and estimate \eqref{paraparaest}
one has that, for $0\leq k\leq K-K'$, the following holds true
\begin{equation}\label{biascica}
\|\del_{t}^{k}\big(\Phi^{\pm1}(U)-{\rm Id}\big) h\|_{H^{s-1-2k}}\leq \sum_{k_1+k_{2}=k}
C \|U\|_{K'+k_1,s_0}\|h\|_{k_2,s}
\end{equation}
where $C>0$ depends only on $\|U\|_{K,s_0}$ and ${\rm Id}$ is the identity map on $H^{s}$. Therefore we can write
\begin{equation}\label{duccio}
\Phi(U)\Big[
P*\big[\Phi^{-1}(U)h\big]
\Big]=P*h+\Big((\Phi(U)-{\rm Id})(P*h)\Big)+\Phi\Big[P*\Big((\Phi^{-1}(U)-{\rm Id})h\Big)\Big].
\end{equation}
Using estimate \eqref{biascica} and the fact that the function $P$ is of class $C^1(\TTT)$ we can estimate the last two summands in the r.h.s. of \eqref{duccio} as follows
\begin{equation*}
\begin{aligned}
&\norm{\del_{t}^{k}(\Phi(U)-{\rm{Id}})(P*h)}{H^{s-2k}}\leq 
\sum_{k_1+k_{2}=k}C\norm{U}{K'+k_1,s_0}\norm{P*h}{k_2,s+1}\leq \sum_{k_1+k_{2}=k}C\norm{U}{K'+k_1,s_0}\norm{h}{k_2,s}\\
&\norm{\del_{t}^{k}\Big(\Phi(U)\left[P*\left((\Phi^{-1}(U)-{\rm{Id}})h\right)\right]\Big)}{ s-2k}
\leq \sum_{k_1+k_2=k}C\norm{U}{K'+k_1,s_0} \norm{(\Phi^{-1}(U)-{\rm Id})h}{k_2,s-1}\\
&\qquad \qquad \leq  \sum_{k_1+k_2=k}C\norm{U}{K'+k_1,s_0} \norm{h}{k_2,s},
\end{aligned}
\end{equation*}
for $0\leq k\leq K-K'$ and 
where $C$ is a constant depending on $\norm{P}{C^1}$ and $\norm{U}{K,s_0}$. Hence they belong to the class $\RR^{0}_{K,K'}[r]$.
Finally if $P(x)$ is even in $x$ then the operator $P_{*}$ is parity preserving according to Definition \ref{revmap}, therefore if in addiction $\Phi(U)$ is parity preserving so must be $R(U)$ in \eqref{convoluzionetot1000}.
\end{proof}

We are now in position to prove the following.
\begin{lemma}\label{step3}
If the matrix $A^{(2)}(U;x,\xi)$ in \eqref{sistemafinale2} satisfies Hyp. \ref{ipoipo} (resp. together with $P$ satisfy Hyp. \ref{ipoipo2}) then there exists $s_0>0$ (possibly larger than the one in Lemma \ref{step2}) such that
for any $s\geq s_0$
there exists  an invertible map (resp. an invertible and {parity preserving} map)
$$
\Phi_{3}=\Phi_{3}(U) :  C^{K-2}_{*\R}(I,\hcic^{s}(\TTT,\CCC^2))\to C^{K-2}_{*\R}(I,\hcic^{s}(\TTT,\CCC^2)),
$$
with
\begin{equation}\label{stimona3}
\|(\Phi_{3}(U))^{\pm1}V\|_{K-2,s}\leq \|V\|_{K-2,s}(1+C\|U\|_{K,s_0})
\end{equation}
where $C>0$ depends only on $s$ and $\|U\|_{K,s_0}$
such that the following holds. 
There exists a matrix $A^{(3)}(U;x,\x)$ satisfying Constraint \ref{Matriceiniziale}  and Hyp. \ref{ipoipo} (resp. Hyp. \ref{ipoipo2}) of the form 
\begin{equation}\label{gorilla3}
\begin{aligned}
A^{(3)}(U;x,\x)&:=A_{2}^{(3)}(U)(\ii \x)^{2}+A_{1}^{(3)}(U;x)(\ii \x),\\
A_{2}^{(3)}(U)&:=\left(\begin{matrix} a_{2}^{(3)}(U) & 0
\\
0 & a_{2}^{(3)}(U)
\end{matrix}
\right), \quad a_{2}^{(3)}\in \calF_{K,3}[r], \quad {\rm independent\,\, of} \; x\in \TTT,
\\
A_{1}^{(3)}(U;x)&:=
\left(\begin{matrix} {a}_{1}^{(3)}(U;x) & 0\\
0 & \ol{{a_{1}^{(3)}(U;x)}}
\end{matrix}
\right)\in \calF_{K,3}[r]\otimes\MM_{2}(\CCC),\\
\end{aligned}
\end{equation}
and operators ${R}^{(3)}_{1}(U),{R}^{(3)}_2(U)$
in $\RR^{0}_{K,3}[r]\otimes\MM_{2}(\CCC)$, such that by setting $V_{3}=\Phi_{3}(U)V_{2}$ the system \eqref{sistemafinale2} reads
\begin{equation}\label{sistemafinale3}
\partial_t V_{3}=\ii E\Big[\Lambda V_{3}+\bonyw(A^{(3)}(U;x,\x))[V_{3}]+R^{(3)}_{1}(U)[V_{3}]+R^{(3)}_{2}(U)[U]\Big].
\end{equation}
\end{lemma}
\begin{proof}
Let $\beta(U;x)$ be a real symbol in $\mathcal{F}_{K,2}[r]$  to be chosen later
such that condition \eqref{conddiffeo} holds. 
Set moreover $\g(U;x)$ the symbol such that \eqref{conddiffeo2} holds.
Consider accordingly to the  
hypotheses of Lemma \ref{torodiff} the system
\begin{equation}\label{felice}
\begin{aligned}
\dot{W}=\ii E M W,\quad W(0)=\uno, \quad M:=\bonyw
\left(\begin{matrix} B(\tau,x,\x) & 0\\ 0 & \ol{B(\tau,x,-\x)}
\end{matrix}\right),
\end{aligned}
\end{equation}
where $B$ is defined in \eqref{bbb}.  Note that $\ol{B(\tau,x,-\x)}=-B(\tau, x, \xi)$. By Lemma \ref{torodiff} the flow exists and is bounded on $\hcic^s(\TTT,\CCC^2)$
and moreover \eqref{stimona3} holds.  We want to conjugate the system \eqref{sistemafinale2} through the map $\Phi_3(U)[\cdot]=W(1)[\cdot]$. 
Set $V_3=\Phi_3(U)V_2$.
The system in the new coordinates reads
\begin{equation}\label{chi}
\begin{aligned}
\frac{d}{dt}V_3= \Phi_3(U)\Big[
\ii E (\mathfrak{P}\big[\Phi_3^{-1}(U)V_3)\big]&+
(
\del_{t}\Phi_{3}(U))\Phi_{3}^{-1}(U)[V_3]
\\
&+\Phi_3(U)\Big[\ii E\bonyw((\uno+A_2^{(2)}(U;x))(\ii\xi)^2)\Big]
\Phi_{3}^{-1}(U)[V_{3}]\\
&+\Phi_3(U)\Big[\ii E\bonyw(A_1^{(2)}(U;x)(\ii\xi))\Big]
\Phi_{3}^{-1}(U)[V_{3}]\\
&+\Phi_3(U)\Big[\ii ER_1^{(2)}(U)\Big]
\Phi_{3}^{-1}(U)[V_{3}]+\Phi_3(U)\ii ER_{2}^{(2)}(U)[U],
\end{aligned}
\end{equation}
where $\mathfrak{P}$ is defined in \eqref{convototale}.
We now discuss each term in \eqref{chi}. The first one, by Lemma \ref{Convocoj2},
is equal to
$\ii E (\mathfrak{P}V_3)$ 
up to a bounded remainder in the class $\RR^{0}_{K,2}[r]\otimes\MM_{2}(\CCC)$.
The last two terms also belongs to the latter class  because the map $\Phi_3$
is a bounded operator on $\hcic^s$.
For the term $(\del_{t}\Phi_{3}(U))\Phi_{3}^{-1}(U)[V_3]$
we apply Proposition $2.5.9$ of \cite{maxdelort} and we obtain that
\begin{equation}
(\del_{t}\Phi_{3}(U))\Phi_{3}^{-1}(U)[V_3]=
\bonyw\left(\begin{matrix}e(U;x)(\ii\x) & 0 \\ 0 & \ol{e(U;x)}(\ii\x)\end{matrix}\right)[V_3]+\tilde{R}(U)[V_3],
\end{equation}
where $\tilde{R}$ belongs to $\RR^{-1}_{K,3}[r]\otimes\MM_{2}(\CCC)$
and $e(U;x)$ is a symbol in $\calF_{K,3}[r]\otimes\MM_{2}(\CCC)$ such that ${\rm{Re}}(e(U;x))=0$. It remains to study the conjugate of the paradifferential  terms in \eqref{chi}.
We note that
\[
\begin{aligned}
\Phi_3(U)&\Big[\ii E\bonyw((\uno+A_2^{(2)}(U;x))(\ii\xi)^2)\Big]
\Phi_{3}^{-1}(U)[V_{3}]
+\Phi_3(U)\Big[\ii E\bonyw(A_1^{(2)}(U;x)(\ii\xi))\Big]
\Phi_{3}^{-1}(U)[V_{3}]\\
&=\left(\begin{matrix}
T & 0 \\ 0 & 
\ol{T}
\end{matrix}
\right)
\end{aligned}
\]
where $T$ is the operator
\begin{equation}\label{egoego}
T=\Omega_{B(U)}(1)\bonyw\Big((1+a_{2}^{(2)}(U;x))(\ii\x)^2+a_{1}^{(2)}(U;x)(\ii\x)\Big)\Omega^{-1}_{B(U)}(1).
\end{equation}
The Paracomposition Theorem $2.5.8$ in \cite{maxdelort}, which can be used thanks to Remarks \ref{pathdiffeo} and \ref{pathdi}, guarantees that
\begin{equation}
T=\bonyw ( \tilde{a}_{2}^{(3)}(U;x,\x)+a_{1}^{(3)}(U;x)(\ii\x))[\cdot]
\end{equation}
up to a bounded term in $\RR^0_{K,3}[r]$ and where
\begin{equation}\label{sol}
\begin{aligned}
\tilde{a}_{2}^{(3)}(U;x,\x)&=
\big(1+a^{(2)}_2(U;y)\big)\big(1+\gamma_{y}(1,y)\big)^2_{|_{y=x+\beta(x)}}(\ii\xi)^2, \\
{a}_{1}^{(3)}(U;x)&=a^{(2)}_1(U;y)\big(1+\gamma_{y}(1,y)\big)_{|_{y=x+\beta(x)}}.
\end{aligned}
\end{equation}
Here  $\gamma(1,x)=\gamma(\tau,x)_{|_{\tau=1}}=\gamma(U;\tau,x)_{|_{\tau=1}}$ 
with
\[
y=x+\tau\be(U;x)\; \Leftrightarrow x=y+\gamma(\tau,y), \;\; \tau\in [0,1],
\]
where $x+\tau\be(U;x)$
is the path of diffeomorphism  
given by Remark \ref{pathdiffeo}.

By Lemma 2.5.4 of Section 2.5 of \cite{maxdelort} one has that the new symbols 
$\tilde{a}_{2}^{(3)}(U;x,\x), a_{1}^{(3)}(U;x)$ defined in \eqref{sol} belong to the class $\Gamma^2_{K,3}[r]$ and $\calF_{K,3}[r]$ respectively. 
At this point we want to choose the symbol $\beta(x)$ in such a way that $\tilde{a}^{(3)}_2(U;x,\x)$ does not depend on $x$. 
One can proceed as follows. Let $a_{2}^{(3)}(U)$ a $x$-independent function to be chosen later, one would like to solve the equation
\begin{equation}\label{mgrande}
\big(1+a_2^{(2)}(U;y)\big)\big(1+\gamma_{y}(1,y)\big)^2_{|_{y=x+\beta(x)}}(\ii\xi)^2=(1+a_{2}^{(3)}(U))(\ii\x)^2.
\end{equation}
The solution of this equation is given by
\begin{equation}\label{mgrande2}
\gamma(U;1,y)=\partial_y^{-1}\left(\sqrt{\frac{1+a_{2}^{(3)}(U)}{1+a^{(2)}_2(U;y)}}-1\right).
\end{equation}
In principle this solution is just formal because the operator $\partial_y^{-1}$ is defined only for function with zero mean, therefore we have to choose $a_{2}^{(3)}(U)$ in such a way that 
\begin{equation}\label{mgrande3}
\int_{\TTT}\left(\sqrt{\frac{1+a_{2}^{(3)}(U)}{1+a^{(2)}_2(U;y)}}-1\right)dx=0,
\end{equation}
which means
\begin{equation}\label{felice6}
1+a_{2}^{(3)}(U):=\left[2\pi \left(\int_{\TTT}\frac{1}{\sqrt{1+a^{(2)}_{2}(U;y)}}dy
\right)^{-1}\right]^{2}.
\end{equation}
Note that everything is well defined thanks to the positivity of $1+a_{2}^{(2)}$. Indeed $a_{2}^{(2)}=a_{2}^{(1)}$ by \eqref{gorilla2}, and 
$a_{2}^{(1)}$ satisfies \eqref{elly2}.  
Indeed every denominator in \eqref{mgrande2}, \eqref{mgrande3} and in \eqref{felice6} stays far away from $0$. 
Note that $\gamma(U;y)$ belongs to  $\calF_{K,2}[r]$ and so does $\be(U;x)$
by Lemma \ref{LEMMA252}. 
By using \eqref{conddiffeo2} one can deduce that 
\begin{equation}
1+\beta_x(U;x)=\frac{1}{1+\gamma_y(U;1,y)}
\end{equation}
where 

\begin{equation}
1+\gamma_{y}(U;1,y)=2\pi
\left(\int_{\TTT}\frac{1}{\sqrt{1+a^{(2)}_{2}(U;y)}}dy\right)^{-1}\frac{1}{\sqrt{1+a_2^{(2)}(U;y)}},
\end{equation}
thanks to \eqref{mgrande2} and \eqref{felice6}.
Since the matrix $A_2^{(2)}$ satisfies Hypothesis \ref{ipoipo4}
 one has that there exists a universal constant $\mathtt{c}>0$ such that $1+a_2^{(2)}(U;y)\geq\mathtt{c}$. Therefore one has
\[
\begin{aligned}
1+\beta_x({U;x})=\frac{1}{1+\gamma_y(U;1,y)}&\geq
\sqrt{\mathtt{c}}\frac{1}{2\pi}\int_{\TTT}\frac{1}{\sqrt{1+a_{2}^{(2)}(U;y) }}dy\\
&\geq\frac{1}{2\pi}\frac{\sqrt{\mathtt{c}}}{1+C\|U\|_{0,s_0}}:=\Theta>0,
\end{aligned}
\]
for some $C$ depending only on $\|U\|_{K,s_0}$,
 where 
 we used the fact that $a_{2}^{(2)}(U;y)$
belongs to the class $\calF_{K,2}[r]$ (see Def. \ref{apeape}).
This implies that $\be(U;x)$ satisfies condition \eqref{conddiffeo}.
 We have written system \eqref{sistemafinale2} in the form \eqref{sistemafinale3}
with matrices defined in \eqref{gorilla3}.

It remains to show that the new matrix $A^{(3)}(U;x,\xi)$ satisfies either Hyp. \ref{ipoipo}
or \ref{ipoipo2}.
If $A^{(2)}(U;x,\xi)$ is selfadjoint, i.e. satisfies Hypothesis \ref{ipoipo}, then one has that the matrix $A^{(3)}(U;x,\xi)$ is selfadjoint as well thanks to the fact that the map $W(1)$ satisfies the hypotheses (condition \eqref{symsym10}) of 
of Lemma \ref{lemmalemma},  by using  Lemma \ref{lemmalemma2}.
In the case that $A^{(2)}(U;x,\xi)$ is parity preserving, i.e. satisfies Hypothesis \ref{ipoipo2}, then 
$A^{(3)}(U;\xi)$  has the same properties for the following reasons.
The symbols $\beta(U;x)$ and $\gamma(U;x)$ are odd in $x$ if the function $U$ is even in $x$. Hence the flow map $W(1)$ defined by equation \eqref{felice} is parity preserving.  
Moreover the matrix $A^{(3)}(U;x,\x)$ satisfies Hypothesis \ref{ipoipo2} by explicit computation.
\end{proof}

\subsection{Reduction to constant coefficients 2: first order terms }\label{ridu1}

Lemmata \ref{step1}, \ref{step2}, \ref{step3} guarantee that one can conjugate the system
\eqref{sistemainiziale}
to the system \eqref{sistemafinale3} in which the matrix $A^{(3)}(U;x,\x)$ (see  \eqref{gorilla3})
has the form
\begin{equation}\label{espansionediA3}
A^{(3)}(U;x,\x)=A_{2}^{(3)}(U)(\ii\x)^{2}+A_{1}^{(3)}(U;x)(\ii\x),
\end{equation}
where the matrices $A_{2}^{(3)}(U),A_{1}^{(3)}(U;x)$ are diagonal and belong to $\calF_{K,3}[r]\otimes\MM_{2}(\CCC)$, for $i=1,2$.
Moreover $A_{2}^{(3)}(U)$ does not depend on $x\in \TTT$.
In this Section we show how to eliminate the  $x$ dependence of the symbol 
 $A_{1}^{(3)}(U;x)$ in \eqref{gorilla3}.
 We prove the following.

 \begin{lemma}\label{step4}
If the matrix $A^{(3)}(U;x,\xi)$ in \eqref{sistemafinale3} satisfies Hyp. \ref{ipoipo} (resp. together with $P$ satisfy Hyp. \ref{ipoipo2}) then there exists $s_0>0$ (possibly larger than the one in Lemma \ref{step3}) such that
for any $s\geq s_0$
there exists  an invertible map (resp. an invertible and {parity preserving} map)
$$
\Phi_{4}=\Phi_{4}(U) :  C^{K-3}_{*\R}(I,\hcic^{s}(\TTT,\CCC^2))\to C^{K-3}_{*\R}(I,\hcic^{s}(\TTT,\CCC^2)),
$$
with
\begin{equation}\label{stimona4}
\|(\Phi_{4}(U))^{\pm1}V\|_{K-3,s}\leq \|V\|_{K-3,s}(1+C\|U\|_{K,s_0})
\end{equation}
where $C>0$ depends only on $s$ and $\|U\|_{K,s_0}$
such that the following holds. 
Then there exists a matrix $A^{(4)}(U;\x)$ independent of $x\in\TTT$ of the form 
\begin{equation}\label{gorilla4}
\begin{aligned}
A^{(4)}(U;\x)&:=
\left(\begin{matrix} {a}_{2}^{(3)}(U) & 0\\
0 & \ol{{a_{2}^{(3)}(U)}}
\end{matrix}
\right)
(\ii \x)^{2}+
\left(\begin{matrix} {a}_{1}^{(4)}(U) & 0\\
0 & \ol{{a_{1}^{(4)}(U)}}
\end{matrix}
\right)
(\ii \x),\\
\end{aligned}
\end{equation}
where $a_{2}^{(3)}(U)$ is defined in \eqref{gorilla3} and $a_1^{(4)}(U)$ is a symbol in $\mathcal{F}_{K,4}[r]$, independent of $x$, which is purely imaginary in the case of Hyp. \ref{ipoipo} (resp. is equal to $0$).
There are operators ${R}^{(4)}_{1}(U),{R}^{(4)}_2(U)$
in $\RR^{0}_{K,4}[r]\otimes\MM_{2}(\CCC)$, such that by setting $V_{4}=\Phi_{4}(U)V_{3}$ the system \eqref{sistemafinale3} reads
\begin{equation}\label{sistemafinale4}
\partial_t V_{4}=\ii E\Big[\Lambda V_{4}+\bonyw(A^{(4)}(U;\x))[V_{4}]+R^{(4)}_{1}(U)[V_{4}]+R^{(4)}_{2}(U)[U]\Big].
\end{equation}

\end{lemma}

\begin{proof}
Consider a symbol $s(U;x)$ in the class $\calF_{K,3}[r]$ and define
\[
S(U;x):=\left(
\begin{matrix}
s(U;x) & 0\\ 0 & \ol{s(U;x)}
\end{matrix}
\right).
\]
Let $\Phi_{4}^{\tau}(U)[\cdot]$
be the flow of the system
\begin{equation}\label{ordine1}
\left\{\begin{aligned}
&\del_{\tau}\Phi_{4}^{\tau}(U)[\cdot]=\bonyw(S(U;x))\Phi_{4}^{\tau}(U)[\cdot]\\
&\Phi_{4}^{0}(U)[\cdot]=\uno.
\end{aligned}\right.
\end{equation}
Again one can reason
 as done for the system \eqref{generatore2} to check that  there exists a unique family of
invertible bounded operators on $\hcic^{s}$
satisfying 
\begin{equation}\label{stimona2000}
\|(\Phi_{4}^{\tau}(U))^{\pm1}V\|_{K-3,s}\leq \|V\|_{K-3,s}(1+C\|U\|_{K,s_0})
\end{equation}
for $C>0$ depending on $s$ and $\|U\|_{K,s_0}$ for $\tau\in [0,1]$.
We set 
\begin{equation}\label{defmappa4}
\Phi_{4}(U)[\cdot]=\Phi_{4}^{\tau}(U)[\cdot]_{|_{\tau=1}}=\exp\{\bonyw(S(U;x))\}.
\end{equation}
By 
Corollary \ref{esponanziale}
we get that there exists $Q(U)$ in the class of smoothing remainder $\RR^{-\rho}_{K,3}[r]\otimes\MM_{2}(\CCC)$ for any $\rho>0$ such that
\begin{equation}\label{fluxapprox}
\Phi_{4}(U)[\cdot]:=\bonyw(\exp\{S(U;x)\})[\cdot]+Q(U)[\cdot].
\end{equation}
Since $\Phi_{4}^{-1}(U)$ exists, by symbolic calculus, it is easy to check
that there exists $\tilde{Q}(U)$ in  $\RR^{-\rho}_{K,3}[r]\otimes\MM_{2}(\CCC)$
such that
\[
\Phi^{-1}_{4}(U)[\cdot]=\bonyw(\exp\{-S(U;x)\})[\cdot]+\tilde{Q}(U)[\cdot].
\]

We set $G(U;x)=\exp\{S(U;x)\}$ and
$V_4=\Phi_4(U)V_3$. Then the system \eqref{sistemafinale3} becomes

\begin{equation}\label{nuovosist600}
\begin{aligned}
(V_4)_{t}&=\Phi_{4}(U)\ii E\Big(\Lambda+\bonyw(A^{(3)}(U;x,\x))+R_{1}^{(3)}(U)\Big)(\Phi_{4}(U))^{-1}[V_{4}]+\\
&+\Phi_{4}(U)\ii ER_{2}^{(3)}(U)[U]+\bonyw(\del_{t}G(U;x,\x))(\Phi_{4}(U))^{-1}[V_{4}].
\end{aligned}
\end{equation}
Recalling that $\Lambda =\mathfrak{P}+\frac{d^2}{dx^2}$
(see \eqref{DEFlambda2})
we note that by Lemma \ref{Convocoj} the term
$\ii \Phi_{4}(U)\big[E\mathfrak{P}\big(\Phi_{4}^{-1}(U)[V_4]\big]$ is equal to $\ii E \mathfrak{P}V_4$ up to a remainder in $\RR^0_{K,4}[r]\otimes\MM_2(\CCC)$.
Secondly  we note that the operator
\begin{equation}\label{pezzibbd}
\hat{Q}(U)[\cdot]:=\Phi_{4}(U)
\ii ER_{1}^{(3)}(U)
\Phi_{4}^{-1}(U)[\cdot]+\Phi_{4}(U)
\ii ER_{2}^{(3)}(U)[U]+\bonyw(\del_{t}G(U;x))\circ\Phi_{4}^{-1}(U)[
\cdot]
\end{equation}
belongs to the class of operators $\RR^{0}_{K,4}[r]\otimes\MM_{2}(\CCC)$. This follows by applying Propositions
\ref{componiamoilmondo}, \ref{componiamoilmare}, Remark \ref{inclusione-nei-resti} and the fact that $\partial_t G(U;x)$ is a matrix in $\calF_{K,4}[r]\otimes\MM_2(\CCC)$.
It remains to study the term
\begin{equation}\label{ferretti}
\Phi_{4}(U)\ii E\Big(\bonyw\big((\uno+A_2^{(3)}(U))(\ii\xi)^2\big)+\bonyw(A^{(3)}_1(U;x)(\ii\xi))\Big)(\Phi_{4}(U))^{-1}.
\end{equation}
By using formula \eqref{sbam8} and Remark \ref{inclusione-nei-resti}
one gets that, up to remainder in $\RR^{0}_{K,4}[r]\otimes\MM_2(\CCC)$,
 the term in \eqref{ferretti} is equal to
 \begin{equation}\label{boris}
\ii E\bonyw\big((\uno+A_2^{(3)}(U))(\ii\xi)^2\big)+\ii E\bonyw\left(
\begin{matrix}
r(U;x)(\ii\x) & 0\\
0 & \ol{r(U;x)}(\ii\x)
\end{matrix}
\right)
\end{equation}
where
\begin{equation}\label{nonmifermo}
\begin{aligned}
r(U;x):=a_{1}^{(3)}(U;x)+2(1+a_{2}^{(3)}(U))\del_{x}s(U;x).
\end{aligned}
\end{equation}
We look for a symbol $s(U;x)$ such that, the term of order one has constant coefficient in $x$. This requires to solve the equation
\begin{equation}\label{cancella1}
a_{1}^{(3)}(U;x)+2(1+a_{2}^{(3)}(U))\del_{x}s(U;x)=a_{1}^{(4)}(U),
\end{equation}
for some symbol  $a_{1}^{(4)}(U)$ constant in $x$ to be chosen. Equation \eqref{cancella1}
is equivalent to
\begin{equation}\label{cancella2}
\del_{x}s(U;x)=\frac{-a_{1}^{(3)}(U;x)+a_{1}^{(4)}(U)}{2(1+a_{2}^{(3)}(U))}.
\end{equation}
We choose the constant $a_{1}^{(4)}(U)$ as
\begin{equation}\label{sceltaA1}
a_{1}^{(4)}(U):=\frac{1}{2\pi}\int_{\TTT}a_{1}^{(3)}(U;x)dx,
\end{equation}
so that the r.h.s. of \eqref{cancella2} has zero average, hence the solution of \eqref{cancella2}
is given by 
\begin{equation}\label{solcancella}
s(U;x):=\del_{x}^{-1}\left(\frac{-a_{1}^{(3)}(U;x)+a_{1}^{(4)}(U)}{2(1+a_{2}^{(3)}(U))}\right).
\end{equation}
It is easy to check that 
$s(U;x)$ belongs to $\calF_{K,4}[r]$.
Using equation \eqref{nonmifermo}
we get \eqref{sistemafinale4} with $A^{(4)}(U;\x)$ as in \eqref{gorilla4}.

 It remains to prove that the constant $a_{1}^{(4)}(U)$ in \eqref{sceltaA1} is purely imaginary.
On one hand, if $A^{(3)}(U;x,\x)$ 
satisfies Hyp. \ref{ipoipo}, 
we note the following. The coefficient  $a_{1}^{(3)}(U;x)$ must be purely imaginary
hence 
the constant $a_{1}^{(4)}(U)$
in \eqref{sceltaA1} is purely imaginary. 

 On the other hand, if $A^{(3)}(U;x,\x)$ satisfies Hyp. \ref{ipoipo2},
we note that the function $a_{1}^{(3)}(U;x)$ is \emph{odd} in $x$. 
This means that the constants $a_{1}^{(4)}(U)$ in \eqref{sceltaA1} is zero. 
Moreover the symbol $s(U;x)$ in \eqref{solcancella} is even in $x$, hence the map 
$\Phi_{4}(U)$ in \eqref{ordine1} is parity preserving according to Def. \ref{revmap}
thanks to Remark \ref{compsimb}.
This concludes the proof.

\end{proof}

\begin{proof}[{\bf Proof of Theorem \ref{descent}}]
It is enough to choose  $\Phi(U):=\Phi_4(U)\circ\cdots\circ\Phi_1(U)$.
The estimates 
\eqref{stimona} follow by collecting the bounds \eqref{stimona1}, \eqref{stimona2},  \eqref{stimona3} and  \eqref{stimona4}. We define the matrix of symbols $L(U;\x)$ as
\begin{equation}
L(U;\x):=\left(\begin{matrix}{\mathtt{m}}(U,\x) & 0\\0 & \mathtt{m}(U,-\x)\end{matrix}\right), \quad 
\mathtt{m}(U,\x):=a_{2}^{(3)}(U)(\ii\x)^{2}+a_{1}^{(4)}(U)(\ii\x)
\end{equation}
where the coefficients $a_{2}^{(3)}(U),a_{1}^{(4)}(U)$ are $x$-independent (see \eqref{gorilla4}).  One concludes the proof by
setting $R_{1}(U):=R_{1}^{(4)}(U)$ and $R_{2}(U):=R_{2}^{(4)}(U)$ .
\end{proof}

An important consequence of Theorem \ref{descent} is that system \eqref{sistemainiziale}
admits a regular and unique solution. More precisely we have the following.

\begin{prop}\label{energia} Let $s_0$ given by Theorem \ref{descent} with $K=4$.  For any $s\geq s_0+2$ let $U=U(t,x)$ be
 a function in $ B^4_{s}([0,T),\theta)$ 
 for some $T>0$, $r>0$, $\theta\geq r$ with $\|U(0,\cdot)\|_{\hcic^{s}}\leq r$ and consider the system
\begin{equation}\label{energia11}
\left\{\begin{aligned}
&\partial_t V=\ii E\Big[\Lambda V+\bonyw(A(U;x,\x))[V]+R_1^{(0)}(U)[V]+R_2^{(0)}(U)[U]\Big],\\
& V(0,x)=U(0,x)\in\hcic^s, 
\end{aligned}\right.
\end{equation}
where the matrix $A(U;x,\x)$, the operators $R_1^{(0)}(U)$ and $R_{2}^{(0)}(U)$ satisfy the hypotheses of Theorem \ref{descent}. Then the following holds true.

\begin{itemize}
\item[$(i)$]
There exists
a unique solution  $\psi_{U}(t)U(0,x)$  of the system \eqref{energia11}  defined for any $t\in[0,T)$
such that
\begin{equation}\label{energia2}
\norm{\psi_U(t)U(0,x)}{4,s}\leq\mathtt{C}\norm{U(0,x)}{\hcic^s}(1+t\mathtt{C}\norm{U}{4,s})e^{t\mathtt{C}\norm{U}{4,s}}+t\mathtt{C}\norm{U}{4,s}e^{t\mathtt{C}\norm{U}{4,s}}+\mathtt{C},
\end{equation}
where $\mathtt{C}$ is  constant depending on $s,r$,
$\sup_{t\in[0,T]}\norm{U}{4,s-2}$ and $\|P\|_{C^{1}}$.
 
\item[$(ii)$]
 In the case that $U$  is even in $x$, 
 the matrix $A(U;x,\x)$ and the operator $\Lambda$ satisfy Hyp. \ref{ipoipo2},
the operator  $R_1^{(0)}(U)[\cdot]$ is parity preserving according to Def. \ref{revmap} and $R_{2}^{(0)}(U)[U]$ is even in $x$, then the solution  $\Psi_{U}(t)U(0,x)$ is even in $x\in \TTT$.
\end{itemize}
\end{prop}

\begin{proof}
We apply to  system  \eqref{energia11} Theorem \ref{descent} defining $W=\Phi(U)V$. The system 
in the new coordinates reads
\begin{equation}\label{modificato}
\left\{
\begin{aligned}
&\del_{t}W-\ii E\Big[\Lambda W+\bonyw(L(U;\xi))W+{R}_1(U)W+R_{2}(U)[U]\Big]=0 \\
&W(0,x)=\Phi(U(0,x))U(0,x):=W^{(0)}(x),
\end{aligned}\right.
\end{equation}
where $L(U;\xi)$ is a diagonal, self-adjoint  and constant coefficient in $x$ matrix in $\Gamma^2_{4,4}[r]\otimes\mathcal{M}_2(\CCC)$, $R_1(U), R_2(U)$ are in $\mathcal{R}^0_{4,4}[r]\otimes\mathcal{M}_2(\CCC)$. Therefore the solution 
of the linear problem
\begin{equation}
\left\{
\begin{aligned}
&\del_{t}W-\ii E\Big[\Lambda W+ \bonyw(L(U;\xi))\Big]W=0 \\
&W(0,x)=W^{(0)}(x),
\end{aligned}\right.
\end{equation}
is well defined as long as $U$ is well defined, moreover it is an isometry of $\hcic^s$. We denote by $\psi_L^t$ the flow at time $t$ of such equation. Then one can
define the operator
\begin{equation}\label{mappacontr}
 T_{W^{(0)}}(W)(t,x)=\psi_L^t  \big(W^{(0)}(x)\big)+\psi_L^t\int_0^t (\psi_L^s)^{-1}\ii E\Big({R}_1(U(s,x))W(s,x)+R_{2}(U(s,x))U(s,x)\Big)ds.
\end{equation}
Thanks to \eqref{stimona} and by the hypothesis on $U(0,x)$ one has that 
$\norm{W^{(0)}}{\hcic^s}\leq (1+c r)r$ 
for some constant $c>0$ depending only on $r$. 
In order to construct a fixed point for the operator $ T_{W^{(0)}}(W)$ in \eqref{mappacontr} we consider the 
sequence of approximations defined as follows:
\[
\left\{
\begin{aligned}
W_0(t,x)&=\psi_{L}^tW^{(0)}(x), \\
W_{n}(t,x)&=T_{W^{(0)}}(W_{n-1})(t,x), \qquad n\geq 1,
\end{aligned}
\right.
\]
for $t\in [0,T)$. For the rest of the proof we will denote by $\mathtt{C}$ any constant depending on $r$, $s$,  $\sup_{t\in[0,T)}\norm{U(t,\cdot)}{4,s-2}$ and $\norm{P}{C^1}$.
Using estimates \eqref{porto20}  one gets for $n\geq 1$
\[
\|(W_{n+1}-W_{n})(t,\cdot)\|_{\hcic^{s}}\leq \mathtt{C}\norm{U(t,\cdot)}{\hcic^s}\int_{0}^{t}\|(W_{n}-W_{n-1})(\tau,\cdot)\|_{\hcic^{s}}d\tau.
\]
Arguing by induction over $n$, one deduces
\begin{equation}\label{referee}
\|(W_{n+1}-W_{n})(t,\cdot)\|_{\hcic^{s}}\leq \frac{({\mathtt{C}}\norm{U(t,\cdot)}{\hcic^s})^{n}t^n}{n!}\|(W_{1}-W_{0})(t,\cdot)\|_{\hcic^{s}},
\end{equation}
which implies that $W(t,x)=\sum_{n=1}^{\infty}(W_{n+1}-W_{n})(t,x)+W_0(t,x)$ is a fixed point of the operator in \eqref{mappacontr} belonging to the space
$C^{0}_{*\R}([0,T);\hcic^{s}(\TTT;\CCC^2))$. 
Therefore by Duhamel principle the function $W$ is the   unique solution of the problem 
\eqref{modificato}. Moreover, by using \eqref{porto20}, we have that the following inequality holds true
\begin{equation*}
\norm{W_1(t,\cdot)-W_{0}(t,\cdot)}{\hcic^s}\leq t \mathtt{C}\big(\norm{U}{\hcic^s}\norm{W^{(0)}}{\hcic^s}+\norm{U}{\hcic^{s-2}}\norm{U}{\hcic^s}\big),
\end{equation*}
from which, together with estimates \eqref{referee}, one deduces that 
\begin{equation*}
\begin{aligned}
\norm{W(t,\cdot)}{\hcic^s}&\leq \sum_{n=0}^{\infty}\norm{(W_{n+1}-W_n)(t,\cdot)}{\hcic^s}+\norm{W^{(0)}}{\hcic^s} \\
&\leq \norm{W^{(0)}}{\hcic^s}\Big(1+t\mathtt{C}\norm{U}{\hcic^s}\sum_{n=0}^{\infty}\frac{(t\mathtt{C}\norm{U}{\hcic^s})^n}{n!}\Big)+t\mathtt{C}\norm{U}{\hcic^s}\sum_{n=0}^{\infty}\frac{(t\mathtt{C}\norm{U}{\hcic^s})^n}{n!}\\
&=\norm{W^{(0)}}{\hcic^s}\big(1+t\mathtt{C}\norm{U}{\hcic^s}e^{t\mathtt{C}\norm{U}{\hcic^s}}\big)+t\mathtt{C}\norm{U}{\hcic^s}e^{{t\mathtt{C}\norm{U}{\hcic^s}}}\\
&\leq \mathtt{C}\norm{U^{(0)}}{\hcic^s}\big(1+t\mathtt{C}\norm{U}{\hcic^s}e^{t\mathtt{C}\norm{U}{\hcic^s}}\big)+t\mathtt{C}\norm{U}{\hcic^s}e^{{t\mathtt{C}\norm{U}{\hcic^s}}}
\end{aligned}
\end{equation*}
Applying the inverse transformation $V=\Phi^{-1}(U)W$ and using \eqref{stimona} we find a solution $V$ of the problem \eqref{energia11} 
such that
\begin{equation*}
\norm{V}{\hcic^s}\leq\mathtt{C}\norm{U^{(0)}}{\hcic^s}\big(1+t\mathtt{C}\norm{U}{\hcic^s}e^{t\mathtt{C}\norm{U}{\hcic^s}}\big)+t\mathtt{C}\norm{U}{\hcic^s}e^{{t\mathtt{C}.\norm{U}{\hcic^s}}}
\end{equation*}
We now prove a similar estimate for $\del_{t}V$. More precisely one has
\begin{equation}\label{derivata1}
\begin{aligned}
\|\del_{t}V\|_{\hcic^{s-2}}&\leq \|\Lambda V+\bonyw(A(U;x,\xi))V\|_{\hcic^{s-2}}+\|{R}_1^{(0)}(U)V\|_{\hcic^{s-2}}
+\|{R}_2^{(0)}(U)U\|_{\hcic^{s-2}}\\
&\leq\mathtt{C}\norm{V}{\hcic^s}+\mathtt{C}\norm{V}{\hcic^{s-2}}+\mathtt{C}\\
&\leq\mathtt{C}\norm{U(0,x)}{\hcic^s}(1+t\mathtt{C}\norm{U}{4,s})e^{t\mathtt{C}\norm{U}{4,s}}+t\mathtt{C}\norm{U}{4,s}e^{t\mathtt{C}\norm{U}{4,s}}+\mathtt{C},
\end{aligned}
\end{equation}
where we used estimates \eqref{paraest} and \eqref{porto20}. 
By differentiating equation \eqref{energia11} and arguing as done in \eqref{derivata1}
one can bound the terms $\|\del_{t}^{k}V\|_{\hcic^{s-2k}}$, for $2\leq k\leq 4$,
and hence obtain the \eqref{energia2}.

In the case that $U$  is even in $x$, $\Lambda, A(U;x,\x)$ satisfy Hyp.  \ref{ipoipo2}, $R_1(U)[\cdot]$ is parity preserving according to Def. \ref{revmap} and $R_{2}^{(0)}(U)[U]$ is even in $x$
we have, by Theorem \ref{descent}, that the map $\Phi(U)$ is parity- preseving.
Hence the flow of the system \eqref{modificato} preserves the subspace of even functions.
This follows by Lemma \ref{revmap100}. Hence the solution of  \eqref{energia11} defined as
$V=\Phi^{-1}(U)W$ is even in $x$. This concludes the proof.

\end{proof}

\begin{rmk}\label{stimaprecisa2}
In the notation of Prop. \ref{energia} the following holds true.
 \begin{itemize}
 \item  If  $R_{2}^{(0)}\equiv 0$ in \eqref{energia11}, then 
the estimate \eqref{energia2}  may be improved as follows:
\begin{equation}\label{energia2tris}
\norm{\psi_U(t)U(0,x)}{4,s}\leq\mathtt{C}\norm{U(0,x)}{\hcic^s}(1+t\mathtt{C}\norm{U}{4,s})e^{t\mathtt{C}\norm{U}{4,s}}.
\end{equation}
 This follows straightforward from the proof of  Prop. \ref{energia}.
\item  If $R_{2}^{(0)}\equiv R_1^{(0)}\equiv 0$ then the flow $\psi_{U}(t)$ of \eqref{energia11} is invertible and $(\psi_{U}(t))^{-1}U(0,x)$ satisfies an estimate similar to \eqref{energia2tris}. To see this one proceed as follows. Let $\Phi(U)[\cdot]$ the map given by Theorem \ref{descent} and set $\Gamma(t):=\Phi(U)\psi_U(t)$. Thanks to Theorem \ref{descent}, $\Gamma(t)$ is the flow of the linear para-differential equation 
\begin{equation*}
\left\{\begin{aligned}
&\partial_t \Gamma(t)=\ii E\bonyw(L(U;\xi))\Gamma(t)+R(U)\Gamma(t),\\
&\Gamma(0)={\rm Id},
\end{aligned}\right.
\end{equation*}
where $R(U)$ is a remainder in $\mathcal{R}^0_{K,4}[r]$ and $\bonyw(L(U;\xi))$ is diagonal, self-adjoint and constant coefficients in $x$.
Then, if $\psi_L(t)$ is the flow generated by $\ii\bonyw(L(U;\xi))$ (which exists and is an isometry of $\hcic^s$), we have that $\Gamma(t)=\psi_L(t)\circ F(t)$, where $F(t)$  solves the Banach space ODE
\begin{equation*}
\left\{\begin{aligned}
&\partial_tF(t)=\big((\psi_L(t))^{-1}R(U)\psi_L(t)\big) F(t),\\
&F(0)={\rm{Id}}.
\end{aligned}\right.
\end{equation*}
To see this one has to use the fact that the operators $\ii \bonyw({L(U;\xi)})$ and $\psi_L(t)$ commutes.
Standard theory of Banach spaces ODE implies that $F(t)$ exists and is invertible, therefore $\psi_U(t)$ is invertible as well and $(\psi_U(t))^{-1}=(F(t))^{-1}\circ(\psi_L(t))^{-1}\circ\Phi(U)$. To deduce the estimate satisfied by $(\psi_U(t))^{-1}$ one has to use \eqref{porto20} to control the contribution coming from $R(U)$, the fact that $\psi_L(t)$ is an isometry and \eqref{stimona}.
\end{itemize}
\end{rmk}

\zerarcounters
\section{Local Existence}\label{local}

In this Section we prove Theorem \ref{teototale}. 
By previous discussions we know that \eqref{NLS} is equivalent to the system \eqref{6.666para} (see Proposition \ref{montero}). 
Our method relies on an iterative scheme. Namely we introduce the following sequence of linear problems.
Let $U^{(0)}\in \hcic^{s}$ such that $\|U^{(0)}\|_{\hcic^s}\leq r$ for some $r>0$.
For $n=0$ we set
\begin{equation}\label{rondine0}
\mathcal{A}_0:=\left\{
\begin{aligned}
&\del_{t}U_0-\ii E\Lambda U_0=0,\\
&U_0(0)=U^{(0)}.
\end{aligned}\right.
\end{equation}
The solution of this problem  exists and is unique, defined for any $t\in \RRR$ by standard linear theory, it is a group of isometries of $\hcic^s$ (its $k$-th derivative is a group of isometries of $\hcic^{s-2k}$) and hence satisfies  $\norm{U_0}{4,s}\leq r$ for any $t\in\R$.

For $n\geq1$, 
 assuming  $U_{n-1}\in B^K_{s_0}(I,r)\cap C^K_{*\R}(I,H^{s}(\TTT,\CCC^2))$ 
for some $s_{0},K>0$ and $s\geq s_0$, 
we define 
the  Cauchy problem
\begin{equation}\label{rondinen}
\mathcal{A}_n:=\left\{
\begin{aligned}
&\del_{t}U_n-\ii E\Big[\Lambda U_n+\bonyw(A(U_{n-1};x,\xi))U_n+R(U_{n-1})[U_{n-1}]\Big]=0 ,\\
&U_n(0)=U^{(0)},
\end{aligned}\right.
\end{equation}
where the matrix of symbols $A(U;x,\x)$ and the operator $R(U)$ are defined   in Proposition \ref{montero} (see \eqref{6.666para}).

One has to show that each problem $\mathcal{A}_{n}$ admits a unique solution $U_{n}$ defined for $t\in I$.
We use  
Proposition \ref{energia}  in order to prove the following lemma.

\begin{lemma}\label{esistenzaAN}

Let $f$ be a $C^{\infty}$ function from $\mathbb{C}^3$ in $\mathbb{C}$ satisfying Hyp. \ref{hyp1} (resp. Hyp. \ref{hyp2}).
Let $r>0$ and consider $U^{(0)}$  in the ball of radius $r$ of $\hcic^s$  
(resp. of $\hcic^s_{e}$) centered at the origin. 
Consider the operators $\Lambda, R(U)$ and the matrix of symbols $A(U;x,\x)$ 
given by Proposition 
\ref{montero} with $K=4$, $\rho=0$.
If $f$ satisfies Hyp. \ref{hyp3}, or $r$ is sufficiently small, then there exists $s_0>0$  such that for all $s\geq s_0$ the following holds. 
There exists a time $T$ and a constant $\theta$, both of them depending on $r$ and $s$,
such that 
for any $n\geq0$ one has:

\begin{description}
\item [$(S1)_{n}$] for $0\leq m\leq n$ there exists a function $U_{m}$ in
\begin{equation}\label{uno}
U_{m}\in B_{s}^{4}([0,T),\theta ),
\end{equation}
which is the unique solution of the problem $\mathcal{A}_{m}$;
in the case of parity preserving Hypothesis \ref{hyp2} 
the functions
$U_{m}$ for $0\leq m\leq n$ are even in $x\in \TTT$;

\item[$(S2)_{n}$] for $0\leq m\leq n$ one has
\begin{equation}\label{due}
\norm{U_{m} -U_{m-1} }{4,s'}\leq 2^{-m} r,\quad s_0\leq s'\leq s-2, 
\end{equation}
where $U_{-1}:=0$.
\end{description}

\end{lemma}
\begin{proof}

We argue by induction. The $(S1)_0$ and $(S2)_0$ are true thanks to the discussion following the equation \eqref{rondine0}.
Suppose  that  
$(S1)_{n-1}$,$(S2)_{n-1}$ hold
with a constant $\theta=\theta(s,r,\|P\|_{C^{1}})\gg1$ and
 a time $T=T(s,r,\|P\|_{C^{1}},\theta)\ll 1$.
We show that $(S1)_{n}$,$(S2)_{n}$ hold with the same constant $\theta$ and $T$.

The Hypothesis \ref{hyp1}, together with Lemma \ref{struttura-ham-para}
(resp. Hyp. \ref{hyp2} together with Lemma \ref{struttura-rev-para}) implies that the matrix $A(U;x,\x)$ satisfies Hyp. \ref{ipoipo} (resp. Hyp. \ref{ipoipo2}) and  Constraint \ref{Matriceiniziale}.
 The Hypothesis \ref{hyp3}, together with Lemma \ref{simboli-ellittici}, (or $r$ small enough)
 implies  that $A(U;x,\x)$ satisfies also the 
 Hypothesis \ref{ipoipo4}.
 Therefore the hypotheses of Theorem \ref{descent} are fulfilled.  
 In particular, in the case of Hyp. \ref{hyp2}, Lemma \ref{struttura-rev-para} guarantees also that the matrix of operators $R(U)[\cdot]$ is parity preserving according to Def. \ref{revmap}.

Moreover
by \eqref{uno}, we have that $\|U_{n-1}\|_{4,s}\leq \theta $, hence 
 the hypotheses of Proposition \ref{energia} are fullfilled
 by system \eqref{rondinen} 
 with $R_{1}^{(0)}=0$, $R_{2}^{(0)}=R$, $U\rightsquigarrow U_{n-1}$
 and $V\rightsquigarrow U_{n}$ in \eqref{energia11}.
%
%
We note that, by $(S2)_{n-1}$, one has that the constant $\mathtt{C}$ in 
\eqref{energia2}
does not depend on $\theta$, but it depend only on $r>0$.
Indeed \eqref{due} implies
\begin{equation}\label{666beast}
\norm{U_{n-1}}{4,s-2}\leq \sum_{m=0}^{n-1}\norm{U_{m} -U_{m-1} }{4,s-2}\leq r\sum_{m=0}^{n-1}\frac{1}{2^{m}}\leq 2r, \quad \forall \; t\in [0,T].
\end{equation}

Proposition \ref{energia} provides a solution $U_{n}(t)$ defined for $t\in [0,T]$. 
By \eqref{energia2} one has that 
\begin{equation}\label{energia2000}
\norm{U_{n}(t)}{4,s}
\leq\mathtt{C}\norm{U^{(0)}}{\hcic^s}(1+t\mathtt{C}\norm{U_{n-1}}{4,s})e^{t\mathtt{C}\norm{U_{n-1}}{4,s}}
+t\mathtt{C}\norm{U_{n-1}}{4,s}e^{t\mathtt{C}\norm{U_{n-1}}{4,s}}+\mathtt{C},
\end{equation}
where $\mathtt{C}$ is a constant depending 
on $\|U_{n-1}\|_{4,s-2}$, $r$, $s$ and $\|P\|_{C^{1}}$, hence, thanks to \eqref{666beast}, it depends only on   $r$, $s$, $\|P\|_{C^{1}}$.
We deduce that, if 
\begin{equation}\label{pane}
T\mathtt{C}\theta \ll 1, \qquad \theta> \mathtt{C}r2e+e+\mathtt{C},
\end{equation}
then $\|U_{n}\|_{4,s}\leq \theta$. If $A(U_{n-1};x,\x)$ and  $\Lambda$ satisfy Hyp. \ref{ipoipo2}, $R(U_{n-1})$ is parity preserving 
then  the solution $U_{n}$ is even in $x\in \TTT$.
Indeed by the inductive hypothesis $U_{n-1}$
is even, hence item $(ii)$ of Proposition \ref{energia} applies.
This proves $(S1)_n$.

\smallskip
Let us check $(S2)_n$. 
Setting $V_n=U_n-U_{n-1}$ we have that
\begin{equation}\label{eqdiff}
\left\{\begin{aligned}
&\del_{t}V_n-\ii E\Big[\Lambda V_n +\bonyw(A(U_{n-1};x,\xi))V_n + f_n\Big]=0,\\
&V_n(0)=0,
\end{aligned}\right.
\end{equation}
where 
\begin{equation}\label{eqdiff2}
f_n:=\bonyw\Big(A(U_{n-1};x,\xi)-A(U_{n-2};x,\xi)\Big)U_{n-1}+R(U_{n-1})U_{n-1}-R(U_{n-2})U_{n-2}.
\end{equation}
Note that, by   \eqref{nave77}, \eqref{nave101},  we have 
\begin{equation}\label{eqdiff3}
\begin{aligned}
\|f_{n}\|_{4,s'}&\leq\norm{\bonyw\Big(A(U_{n-1};x,\xi)-A(U_{n-2};x,\xi)\Big)U_{n-1}}{4,s'}+\norm{R(U_{n-1})U_{n-1}-R(U_{n-2})U_{n-2}}{4,s'}\\
&\leq C\Big[\norm{V_{n-1}}{4,s_0}\norm{U_{n-1}}{4,s'+2}+(\norm{U_{n-1}}{4,s'}+\norm{U_{n-2}}{4,s'})\norm{V_{n-1}}{4,s'}\Big]\\
&\leq C\Big(\norm{U_{n-1}}{4,s'+2}+\norm{U_{n-2}}{4,s'+2}\Big)\norm{V_{n-1}}{4,s'},
\end{aligned}
\end{equation}
where $C>0$ 
depends only on $s$, $\|U_{n-1}\|_{4,s_0}, \|U_{n-2}\|_{4,s_0}$.
Recalling the estimate \eqref{666beast} we can conclude that
the constant $C$ in \eqref{eqdiff3} depends only on $s,r$.

The  system \eqref{eqdiff} with $f_n=0$ has the form \eqref{energia11} with $R_2^{(0)}=0$ and $R_1^{(0)}=0$.
Let 
 $\psi_{U_{n-1}}(t)$ be the flow of system \eqref{eqdiff} with $f_n=0$, which is given by Proposition \ref{energia}.
The Duhamel formulation of \eqref{eqdiff}
is
\begin{equation}
V_{n}(t)=\psi_{U_{n-1}}(t)\int_0^t (\psi_{U_{n-1}}(\tau))^{-1}\ii E f_{n}(\tau)d\tau.
\end{equation}
Then using  the inductive hypothesis \eqref{uno}, inequality \eqref{energia2tris} and the second item of Remark \ref{stimaprecisa2}
we get
\begin{equation}
\|V_n\|_{4,s'}\leq \theta \mathtt{K}_1 T \|V_{n-1}\|_{4,s'}, \quad \forall \; t\in [0,T],
\end{equation}
where $\mathtt{K}_1>0$ is a constant depending $r$, $s$ and $\|P\|_{C^{1}}$.
If  $\mathtt{K}_1 \theta  T\leq1/2$ then we have 
$\|V_n\|_{4,s'}\leq 2^{-n}r$ for any  $ t\in [0,T)$ which is the $(S2)_n$.
\end{proof}
 We are now in position to prove Theorem \ref{teototale}.

\begin{proof}[{\bf Proof of Theorem \ref{teototale}}]
Consider the equation \eqref{NLS}. By Lemma \ref{paralinearizza} we know that \eqref{NLS}
is equivalent to the system \eqref{6.666para}.
Since $f$ satisfies Hyp. \ref{hyp1} (resp. Hyp \ref{hyp2}) and Hyp. \ref{hyp3}, then Lemmata \ref{struttura-ham-para} (resp \ref{struttura-rev-para}) and \ref{simboli-ellittici} imply that the matrix $A(U;x,\x)$ satisfies Constraint \ref{Matriceiniziale} and  
 Hypothesis \ref{ipoipo}
(resp. Hypothesis
\ref{ipoipo2} and $R(U)$ is parity preserving according to Definition \ref{revmap}). 
According to this setting consider the problem $\mathcal{A}_{n}$ in \eqref{rondinen}.

By Lemma \ref{esistenzaAN} we know that the sequence $U_n$ defined by \eqref{rondinen} converges strongly to a function $U$ in $C_{*\R}^0 ([0,T),\hcic^{{s'}})$ for any ${s'}\leq s-2$ and, up to subsequences, 
\begin{equation}\label{debolesol}
\begin{aligned}
&U_{n}(t) \rightharpoonup U(t), \;\;\; {\rm in } \;\;\; \hcic^s,\\
&\del_{t}U_{n}(t)  \rightharpoonup \del_{t}U(t) , \;\;\; {\rm in } \;\;\; \hcic^{s-2},
\end{aligned}
\end{equation}
for any $t\in [0,T)$, moreover the function $U$ is in $L^{\infty}([0,T),\hcic^s)\cap{\rm  Lip}([0,T),\hcic^{s-2})$.
In order to prove that $U$ solves \eqref{6.666para} it is enough to show that
\begin{equation*}
\norm{\bonyw(A(U_{n-1};x,\xi))]U_n+R(U_{n-1})[U_{n-1}]-\bonyw(A(U;x,\xi))]U -R(U)[U]}{{\bf H}^{s-2}}
\end{equation*}
goes to $0$ as $n$ goes to $\infty$.
 Using \eqref{nave77} and \eqref{paraest} we obtain
\begin{equation*}
\begin{aligned}
&\|\bonyw(A(U_{n-1};x,\xi)U_n-\bonyw(A(U;x,\xi)U)\|_{{\bf H}^{s-2}} \leq\\
& \|\bonyw(A(U_{n-1};x,\xi)-A(U;x,\xi))U_n\|_{{\bf H}^{s-2}}+ \|\bonyw(A(U;x,\xi))(U-U_n)\|_{{\bf H}^{s-2}}\leq \\
&C\Big(\norm{U-U_n}{{\bf H}^{s-2}}\norm{U}{{\bf H}^{s_0}}+\norm{U-U_{n-1}}{{\bf H}^{s_0}}\norm{U_n}{{\bf H}^{s}}\Big),
\end{aligned}
\end{equation*}
which tends to $0$ since $s-2\geq s'$. In order to show that $R(U_{n-1})[U_{n-1}]$ tends to $R(U)[U]$ in ${\bf H}^{s-2}$ it is enough to use \eqref{nave101}. 
Using the equation \eqref{6.666para} and the discussion above   the solution $U$ 
has the following regularity:
\begin{equation}\label{regolarita}
\begin{aligned}
U\in 
B^{4}_{s'}([0,T);\theta )\cap L^{\infty}&([0,T),\hcic^s)\cap {\rm Lip}([0,T),\hcic^{s-2}), 
\quad \forall \; s_0\leq s'\leq  s-2,\\
&\norm{U}{L^{\infty}([0,T),\hcic^s)}\leq\theta ,
\end{aligned}
\end{equation}
where $\theta$ and $s_0$ are given by Lemma \ref{esistenzaAN}.
We show that $U$ 
actually belongs to $C^{0}_{*\mathbb{R}}([0,T), \hcic^{s})$.
Let us consider the  problem 
\begin{equation}\label{morteneran}
\left\{
\begin{aligned}
&\del_{t}V-\ii E\Big[\Lambda V+\bonyw(A(U;x,\xi))V+R(U)[U]\Big]=0 ,\\
&V(0)=U^{(0)},\quad U^{(0)}\in {\bf H}^{s},
\end{aligned}\right.
\end{equation}
where the matrices  $A$ and  $R$ are defined   in Proposition \ref{montero} (see \eqref{6.666para})
and $U$ is defined in \eqref{debolesol} (hence satisfies \eqref{regolarita}).
 Theorem \ref{descent} applies to system \eqref{morteneran} and provides a map 
 \begin{equation}\label{regolarita2}
 \Phi(U)[\cdot] :   C^{0}_{*\R}([0,T),\hcic^{s'}(\TTT,\CCC^2))\to C^{0}_{*\R}([0,T),\hcic^{s'}(\TTT,\CCC^2)),\quad 
 \end{equation}
 which satisfies \eqref{stimona} with $K=4$ and $s'$ as in \eqref{regolarita}.
 One has that the function $W:=\Phi(U)[U]$ solves, the problem
\begin{equation}\label{modificato1}
\left\{
\begin{aligned}
&\del_{t}W-\ii E\Big[\Lambda +\bonyw(L(U;\xi))\Big]W+R_2(U)[U]+{R}_1(U)W=0 \\
&W(0)=\Phi(U^{(0)})U^{(0)}:=W^{(0)},
\end{aligned}\right.
\end{equation}
where  $L(U)$ is a diagonal, self-adjoint and constant coefficient in $x$ matrix of symbols in $\Gamma^{2}_{K,4}[\theta ]$, and $R_1(U), {R}_2(U)$ are  matrices of 
bounded operators (see eq. \eqref{sistemafinale}). 
 We prove that $W$ is weakly-continuous in time with values in $\hcic^{s}$.
 First of all note that $U\in C^{0}([0,T); \hcic^{s'})$ with $s'$
given in \eqref{regolarita}, therefore $W$ belongs to the same space thanks to \eqref{regolarita2}.
Moreover $W$ is in $L^{\infty}([0,T),\hcic^s)$ (again by \eqref{regolarita} and \eqref{regolarita2}).
 Consider a sequence $\tau_{n}$ converging to $\tau$ as $n\to\infty$.
Let $\phi\in \hcic^{-s}$ and $\phi_{\e}\in C^{\infty}_{0}(\TTT;\CCC^2)$ 
such that $\|\phi-\phi_{\e}\|_{\hcic^{-s}}\leq \e$. Then we have
\begin{equation}
\begin{aligned}
\left|\int_{\TTT}\big(W(\tau_{n})-W(\tau)  \big) \phi d x\right|&\leq \left|\int_{\TTT}\big(W(\tau_{n})-W(\tau)  \big) \phi_{\e} d x\right|+\left|\int_{\TTT}(W(\tau_{n})-W(\tau)  ) (\phi-\phi_\e) d x\right|\\
&\leq \|W(\tau_n)-W(\tau)\|_{\hcic^{s'}}\|\phi_\e\|_{\hcic^{-s'}}+
\|W(\tau_n)-W(\tau)\|_{\hcic^{s}}\|\phi-\phi_\e\|_{\hcic^{-s}}\\
&\leq C \e+2\|W\|_{L^{\infty}\hcic^{s}}\e 
\end{aligned}
\end{equation}
for $n$ sufficiently large and where $s'\leq s-2$ as above. 

Therefore  $W$ is weakly continuous in time with values in $\hcic^s$. In order to prove that $W$ is in 
$C^0_{*\RRR}([0,T),\hcic^s)$, we show that the map $t\mapsto \norm{W(t)}{\hcic^s}$ is continuous on $[0,T)$. 
We introduce,  for $0<\epsilon\leq 1$, the Friedrichs mollifier $J_{\epsilon}:=(1-\epsilon\partial_{xx})^{-1}$
and the Fourier multiplier $\Lambda^{s}:=(1-\del_{xx})^{s/2}$.  Using the equation \eqref{modificato1} and estimates \eqref{porto20} one gets
\begin{equation}\label{energia1}
\frac{d}{dt}\norm{\Lambda^s J_{\epsilon}W(t)}{\hcic^0}^2\leq C\Big[ \norm{U(t)}{\hcic^s}^2\norm{W(t)}{\hcic^s}+\norm{W(t)}{\hcic^s}^2\norm{U(t)}{\hcic^s}\Big],
\end{equation}
where the right hand side is independent of $\epsilon$ and the constant $C$ depends on $s$ and $\norm{U}{\hcic^{s_0}}$.
Moreover, since $U,\, W$ belong to  $L^{\infty}([0,T),\hcic^s)$, the right hand side of inequality \eqref{energia1} is bounded from above by a constant independent of $t$. Therefore the function $t\mapsto \norm{J_{\epsilon}W(t)}{\hcic^0}$ is Lipschitz continuous in $t$, uniformly in $\epsilon$. As $J_{\epsilon}W(t)$ converges to $W(t)$ in the $\hcic^s$-norm, the function $t\mapsto \norm{W(t)}{\hcic^0}$ is Lipschitz continuous as well. Therefore $W$ belongs to $C^0_{*\RRR}([0,T),\hcic^s)$ and so does $U$. To recover the regularity of $\frac{d}{dt}U$ one may use equation \eqref{6.666para}.

 Let us show the uniqueness. Suppose that there are two solution $U$ and $V$ in $C^0_{*\RRR}([0,T),\hcic^s)$ of the problem \eqref{6.666para} . Set $H:=U-V$, then $H$ solves the problem
\begin{equation}\label{ultimo}
\left\{
\begin{aligned}
&\del_{t}H-\ii E\Big[\Lambda H +\bonyw(A(U;x,\xi))[H]+R(U)[H]\Big]+\ii E F=0 \\
&H(0)=0,
\end{aligned}\right.
\end{equation}
where $$F:=\bonyw\big(A(U;x,\xi)-A(V;x,\xi)\big)V+\big(R(U)-R(V)\big)[V].$$
Thanks to estimates \eqref{nave77} and \eqref{nave101} we have the bound
\begin{equation}\label{stimetta}
\norm{F}{\hcic^{s-2}}\leq C \norm{H}{\hcic^{s-2}}\Big(\norm{U}{\hcic^s}+\norm{V}{\hcic^s}\Big).
\end{equation}
By Proposition \ref{energia}, using Duhamel principle and \eqref{stimetta},
it is easy to show the following:
\[
\|H(t)\|_{\hcic^{s-2}}\leq C(r) \int_{0}^{t}\|H(\s)\|_{\hcic^{s-2}}d\s.
\]

Thus by Gronwall Lemma the solution is equal to zero for almost everywhere time $t$ in $[0,T)$ . By continuity one gets the unicity.
\end{proof}

\begin{proof}[{\bf Proof of Theorem \ref{teototale1}}]
The proof is the same of the one of Theorem \ref{teototale}, one only has to note that the matrix 
$A(U;x,\x)$ satisfies 
Hypothesis \ref{ipoipo4}
 thanks to the smallness of the initial datum instead of Hyp. \ref{hyp3}.
\end{proof}


\begin{thebibliography}{10}

\bibitem{alaz-baldi-periodic}
T.~Alazard and P.~Baldi.
\newblock Gravity capillary standing water waves.
\newblock {\em Archive for Rational Mechanics and Analysis}, 217(3):741--830,
  2015.
\newblock DOI: 10.1007/s00205-015-0842-5.

\bibitem{ABK}
T.~Alazard, P.~Baldi, and D.~Han-Kwan.
\newblock {C}ontrol of water waves.
\newblock {\em J. Eur. Math. Soc. (JEMS)}, in print.
\newblock arXiv:1501.06366.

\bibitem{AliPARA}
S.~Alinhac.
\newblock {P}aracomposition et op\'erateurs paradiff\'erentiels.
\newblock {\em Comm. Partial Differential Equations}, 1986.

\bibitem{Ba2}
P.~Baldi.
\newblock {P}eriodic solutions of fully nonlinear autonomous equations of
  {B}enjamin-{O}no type.
\newblock {\em Ann. I. H. Poincar\'e (C) Anal. Non Lin\'eaire}, 30, 2013.
\newblock 10.1016/j.anihpc.2012.06.001.

\bibitem{BBHM}
P.~Baldi, M.~Berti, E.~Haus, and R.~Montalto.
\newblock Time quasi-periodic gravity water waves in finite depth.
\newblock preprint - arXiv:1708.01517, 2017.

\bibitem{BBM}
P.~Baldi, M.~Berti, and R.~Montalto.
\newblock {KAM} for quasi-linear and fully nonlinear forced perturbations of
  {A}iry equation.
\newblock {\em Math. Ann.}, 359, 2014.
\newblock 10.1007/s00208-013-1001-7.

\bibitem{BBM1}
P.~Baldi, M.~Berti, and R.~Montalto.
\newblock {KAM} for autonomous quasilinear perturbations of {K}d{V}.
\newblock {\em Ann. I. H. Poincar\'e (C) Anal. Non Lin\'eaire}, 33, 2016.
\newblock 10.1016/j.anihpc.2015.07.003.

\bibitem{BHM}
P.~Baldi, E.~Haus, and R.~Montalto.
\newblock Controllability of quasi-linear hamiltonian {N}{L}{S} equations.
\newblock {\em J. Differential Equations}, 2017.
\newblock DOI:10.1016/j.jde.2017.10.009.

\bibitem{BG}
D.~Bambusi and B.~Gr\'ebert.
\newblock {B}irkhoff normal form for partial differential equations with tame
  modulus.
\newblock {\em Duke Math. J.}, 135 n. 3:507--567, 2006.
\newblock 10.1215/S0012-7094-06-13534-2.

\bibitem{maxdelort}
M.~Berti and J.-M. Delort.
\newblock Almost global existence of solutions for capillarity-gravity water
  waves equations with periodic spatial boundary conditions.
\newblock preprint - arXiv:1702.04674, 2017.

\bibitem{BM1}
M.~Berti and R.~Montalto.
\newblock Quasi-periodic standing wave solutions of gravity-capillary water
  waves.
\newblock {\em Memoirs of the American Math. Society, MEMO 891}, 2016.
\newblock to appear - arXiv:1602.02411.

\bibitem{caze}
T.~Cazenave.
\newblock {\em {S}emilinear {S}chr\"odinger {E}quations}, volume~10.
\newblock Courant lecture notes, 2003.

\bibitem{Cris}
M.~Christ.
\newblock Illposedness of a {S}chr{\"o}dinger equation with derivative
  nonlinearity.
\newblock Preprint - https://math.berkeley.edu/~mchrist/Papers/dnls.ps.

\bibitem{Delort-2009}
J.~M. Delort.
\newblock {A} quasi-linear {B}irkhoff normal forms method. {A}pplication to the
  quasi-linear {K}lein-{G}ordon equation on $\mathds{S}^1$.
\newblock {\em Ast\'erisque}, 341, 2012.

\bibitem{Delort-Sphere}
J.~M. Delort.
\newblock {\em {Q}uasi-{L}inear {P}erturbations of {H}amiltonian
  {K}lein-{G}ordon {E}quations on {S}pheres}.
\newblock American Mathematical Society, 2015.
\newblock 10.1090/memo/1103.

\bibitem{F}
R.~Feola.
\newblock {KAM} for a quasi-linear forced {H}amiltonian {NLS}.
\newblock preprint- arXiv:1602.01341, 2015.

\bibitem{FP}
R.~Feola and M.~Procesi.
\newblock Quasi-periodic solutions for fully nonlinear forced reversible
  {S}chr{\"o}dinger equations.
\newblock {\em Journal of Differential Equations}, 2014.
\newblock 10.1016/j.jde.2015.04.025.

\bibitem{FP2}
R.~Feola and M.~Procesi.
\newblock {K}{A}{M} for quasi-linear autonomous {N}{L}{S}.
\newblock preprint- arXiv:1705.07287, 2017.

\bibitem{IPT}
G.~Iooss, P.I. Plotnikov, and J.F. Toland.
\newblock Standing waves on an infinitely deep perfect fluid under gravity.
\newblock {\em Arch. Ration. Mech. Anal.}, 177(3):367--478, 2005.
\newblock 10.1007/s00205-005-0381-6.

\bibitem{KPV1}
C.~E. Kenig, G.~Ponce, and L.~Vega.
\newblock The {C}auchy problem for quasi-linear {S}chr\"odinger equations.
\newblock {\em Invent. math.}, 158(2):343--388, 2004.
\newblock 10.1007/s00222-004-0373-4.

\bibitem{Metivier}
G.~M\'etivier.
\newblock {\em {P}ara-{D}ifferential {C}alculus and {A}pplications to the
  {C}auchy {P}roblem for {N}onlinear {S}ystems}, volume~5.
\newblock Edizioni della Normale, 2008.

\bibitem{Moser-Pisa-66}
J.~Moser.
\newblock {R}apidly convergent iteration method and non-linear partial
  differential equations - i.
\newblock {\em Ann. Sc. Norm. Sup. Pisa}, 20(2):265--315, 1966.

\bibitem{Pop1}
M.~Poppenberg.
\newblock {S}mooth solutions for a class of fully nonlinear {S}chr{\"o}dinger
  type equations.
\newblock {\em Nonlinear Anal., Theory Methods Appl.}, 45(6):723--741, 2001.

\bibitem{Tay-Para}
M.~Taylor.
\newblock {\em {T}ools for {P}{D}{E}}.
\newblock Amer. Math. Soc., 2007.

\bibitem{ZGY}
J.~Zhang, M.~Gao, and X.~Yuan.
\newblock {K}{A}{M} tori for reversible partial differential equations.
\newblock {\em Nonlinearity}, 24, 2011.
\newblock 10.1088/0951-7715/24/4/010.

\end{thebibliography}
\def\cprime{$'$}

\end{document}